\documentclass[11pt]{article}
\usepackage{amssymb, amsmath, amsthm, mathrsfs, ascmac, color}
\usepackage[pdftex]{graphicx} 
\usepackage[subrefformat=parens]{subcaption}
\usepackage[top=3cm, bottom=3cm, left=3cm,right=3cm]{geometry}
\usepackage{comment}
\usepackage{multirow}
\theoremstyle{plain}
\newtheorem{thm}{Theorem}[section]
\newtheorem{prop}[thm]{Proposition}
\newtheorem{lem}{Lemma}

\newtheorem{rmk}{Remark}
\makeatletter

\@addtoreset{equation}{section}
\makeatother
\newcommand{\ind}{\bs 1}
\newcommand{\dd}{\mathrm d}
\newcommand{\pd}{\partial}
\newcommand{\ee}{\mathrm e}
\newcommand{\EE}{\mathbb E}
\newcommand{\PP}{\mathbf{P}}
\newcommand{\OO}{\mathcal O}
\newcommand{\oo}{{\scriptsize\text{$\mathcal O$}}}

\newcommand{\VV}{\mathrm{Var}}
\newcommand{\cov}{\mathrm{Cov}}
\newcommand{\diag}{\mathrm{diag}}

\newcommand{\pto}{\stackrel{\PP}{\to}}
\newcommand{\dto}{\stackrel{d}{\to}}
\newcommand{\TT}{\mathsf T}
\newcommand{\bs}{\boldsymbol}
\newcommand{\mf}{\mathfrak}
\newcommand{\lb}{\langle}
\newcommand{\rb}{\rangle}
\newcommand{\tand}{\widetilde \land}

\allowdisplaybreaks
\title{\textbf{Small dispersion asymptotics for an SPDE 
in two space dimensions using triple increments
}}
\date{}

\author{\textbf{Yozo Tonaki}\thanks{Graduate School of Engineering Science, Osaka University, Toyonaka, Japan}
\thanks{Center for Mathematical Modeling and Data Science (MMDS), Osaka University, Toyonaka, Japan} 
\footnote{e-mail: \texttt{y.tonaki.es@osaka-u.ac.jp}}
\and \textbf{Yusuke Kaino}\thanks{Graduate School of Maritime Sciences, Kobe University, Kobe, Japan}
\and \textbf{Masayuki Uchida}$^{* \dag}$\thanks{CREST, Japan Science and Technology Agency, Tokyo, Japan}
}

\begin{document}
\bibliographystyle{plain}
\maketitle

\begin{abstract}
We consider parametric estimation 
for a second order linear parabolic stochastic partial differential equation (SPDE)
in two space dimensions driven by a $Q$-Wiener process with a small noise
based on high frequency spatio-temporal data.
We first provide estimators of the diffusive and advective parameters 
in the SPDE using temporal and spatial increments.
We then construct an estimator of the reaction parameter in the SPDE
based on an approximate coordinate process.
We also give simulation results of the proposed estimators.

\begin{center}
\textbf{Keywords and phrases}
\end{center}
High frequency spatio-temporal data,
linear parabolic stochastic partial differential equations,
parametric estimation, 
reaction parameter,
small noise.
\end{abstract}

\section{Introduction}\label{sec1}

Stochastic partial differential equations (SPDEs) 
enable mathematical modeling of various spatio-temporal phenomena with 
uncertainty found in the real world.
Interest in SPDEs is growing in applied sciences 
as they can provide more realistic descriptions of phenomena.
For example, SPDEs have been applied in various fields, such as modeling of
sea surface temperature fluctuations \cite{Piterbarg_Ostrovskii1997}, 
membrane potentials of neurons \cite{Tuckwell2013}, 
and concentration variations of air pollutants \cite{Kallianpur_Xiong1995}.
Such growing interest in SPDEs extends to statistical inference of SPDE models.

We study parametric estimation for the linear parabolic SPDE
\begin{align}
\dd X_t(y,z)
&=\biggl\{
\theta_2
\biggl(\frac{\pd^2}{\pd y^2}+\frac{\pd^2}{\pd z^2}\biggr)
+\theta_1\frac{\pd}{\pd y} 
+\eta_1\frac{\pd}{\pd z} 
+\theta_0 
\biggr\} X_t(y,z) \dd t
\nonumber
\\
&\qquad
+\epsilon \dd W_t^{Q}(y,z),
\quad (t,y,z) \in [0,1] \times D
\label{2d_spde}
\end{align}
with a deterministic initial value $X_0 \in L^2(D)$ 
and the Dirichlet boundary condition
$X_t(y,z) = 0$, $(t,y,z) \in [0,1] \times \pd D$,
where $D = (0,1)^2$, $W_t^{Q}$ is a $Q$-Wiener process in a Sobolev space on $D$,
$\theta=(\theta_0, \theta_1, \eta_1, \theta_2) \in \mathbb R^3 \times (0,\infty)$ 
is an unknown parameter, the parameter space $\Theta$ is a compact convex subset of 
$\mathbb R^3 \times (0,\infty)$ and
$\theta^*=(\theta_0^*, \theta_1^*, \eta_1^*, \theta_2^*) \in \mathrm{Int}\,(\Theta)$ 
is the true value of $\theta$. 
In addition, let $\epsilon \in (0,1)$ be a known small dispersion parameter.

Statistical inference for SPDEs based on discrete observations 
has been developed in recent years, see 
Markussen \cite{Markussen2003},
Bibinger and Trabs \cite{Bibinger_Trabs2020},
Bibinger and Bossert \cite{Bibinger_Bossert2023},
Cialenco and Huang \cite{Cialenco_Huang2020},
Hildebrandt and Trabs \cite{Hildebrandt_Trabs2021, Hildebrandt_Trabs2023},
Tonaki et al.\,\cite{TKU2023,TKU2024b}
and Bossert \cite{Bossert2023arXiv}.
Cialenco \cite{Cialenco2018} provides an overview of existing theories of 
statistical inference for parabolic SPDEs.
The literature on statistics for SPDEs is also listed in  \cite{Altmeyer_Web}.
In general, we cannot construct a consistent estimator of the reaction parameter
in second order parabolic SPDEs driven by a $Q$-Wiener process without $\theta_0$
as long as we consider infill asymptotics in
a bounded domain, finite temporal interval.
Therefore, parametric estimation of the reaction term
in second order parabolic SPDEs
under various asymptotics has been studied.
See, for example, 
Kaino and Uchida \cite{Kaino_Uchida2020, Kaino_Uchida2021},
Tonaki et al.\,\cite{TKU2024a, TKU2024arXiv1} and
Gaudlitz and Reiss \cite{Gaudlitz_Reiss2023}.

Tonaki et al.\,\cite{TKU2024a} considered parametric estimation of $\theta_0$ 
in SPDE \eqref{2d_spde} under $\epsilon \to 0$ 
and found through numerical simulations that the estimation of $\theta_0$ 
is strongly affected by the bias of the estimator of $\theta_2$. 
For this reason, we need to improve the estimation accuracy of 
$\theta_1$, $\eta_1$ and $\theta_2$.
In this paper, we thus construct estimators of $\theta_1$, $\eta_1$ and $\theta_2$
under $\epsilon \to 0$ based on the spatio-temporal increments 
proposed by Tonaki et al.\,\cite{TKU2024b}, and estimate $\theta_0$.
The aim of this paper is
to show that our estimators are superior to those of \cite{TKU2024a} 
in the sense that the estimation accuracy of $\theta_1$, $\eta_1$ and $\theta_2$ 
is improved, 
and consequently, 
we can construct a consistent estimator of $\theta_0$ under weaker conditions
than \cite{TKU2024a}. 
We also verify this fact through numerical simulations.
As a side result, we show that the restriction of the damping parameter $\alpha$ 
in the $Q$-Wiener process \eqref{QW2} below 
in order to construct consistent estimators of 
the parameters is $0< \alpha < 3$.

This paper is organized as follows.
Section \ref{sec2} provides the main results.
We first construct estimators of the coefficient parameters 
$\theta_2$, $\theta_1$ and $\eta_1$ in SPDE \eqref{2d_spde} 
in line with the approach of \cite{TKU2024b}.
We then propose an estimator of $\theta_0$ based on 
the estimators of $\theta_2$, $\theta_1$ and $\eta_1$ 
and an approximate process of the coordinate process of SPDE \eqref{2d_spde}.
In Section \ref{sec3}, we perform numerical simulations of the proposed estimators 
and compare our estimators with those in \cite{TKU2024a}.
In Section \ref{sec4}, we give the proofs of our results.

\section{Main results}\label{sec2}
\subsection{Setting and notation}
We set the following notation.
\begin{enumerate}
\item
We define an operator
\begin{equation*}
-A_\theta = 
\theta_2\biggl(\frac{\pd^2}{\pd y^2} + \frac{\pd^2}{\pd z^2} \biggr)
+ \theta_1\frac{\pd}{\pd y} + \eta_1\frac{\pd}{\pd z} + \theta_0.
\end{equation*}
We then obtain the following representation of SPDE \eqref{2d_spde}.
\begin{equation*}
\dd X_t = -A_\theta X_t \dd t + \epsilon \dd W_t^{Q}.
\end{equation*}

\item
\textit{Eigenpairs $\{ \lambda_{l_1,l_2}, e_{l_1,l_2}\}_{l_1,l_2 \ge 1}$.}
The operator $A_\theta$ has the eigenpairs 
$\{\lambda_{l_1,l_2}, e_{l_1,l_2}\}_{l_1,l_2 \ge 1}$ with 
\begin{equation*}
e_{ l_1, l_2}(y,z)=e_{ l_1}^{(1)}(y)e_{ l_2}^{(2)}(z),
\quad
\lambda_{ l_1, l_2}=\theta_2(\pi^2( l_1^2+ l_2^2)+\Gamma),
\end{equation*}
for $ l_1, l_2 \ge 1$ and $y,z\in [0,1]$, where
\begin{equation*}
e_{ l_1}^{(1)}(y)
=\sqrt 2 \sin(\pi l_1 y) \ee^{-\kappa y/2},
\quad
e_{ l_2}^{(2)}(z)
=\sqrt 2 \sin(\pi l_2 z) \ee^{-\eta z/2},
\end{equation*}
\begin{equation*}
\kappa=\frac{\theta_1}{\theta_2}, 
\quad
\eta=\frac{\eta_1}{\theta_2},
\quad
\Gamma=-\frac{\theta_0}{\theta_2} +\frac{\kappa^2+\eta^2}{4}.
\end{equation*}
We assume $\lambda_{1,1}^* > 0$ 
so that $A_\theta$ is a positive-definite, self-adjoint operator.

\item
\textit{Inner product $\langle \cdot , \cdot \rangle$.}
The eigenfunctions $\{e_{ l_1, l_2}\}_{ l_1, l_2\ge1}$ 
are orthonormal with respect to the weighted $L^2$-inner product
\begin{equation*}
\langle u, v \rangle
= \int_0^1 \int_0^1 u(y,z)v(y,z)\ee^{\kappa y +\eta z} \dd y \dd z, 
\quad 
\| u \| = \sqrt{\langle u, u \rangle}.
\end{equation*}
for $u, v \in L^2(D)$.

\item
\textit{$Q$-Wiener process.}
Let $\mu_0 \in (-2\pi^2, \infty)$ be a known or an unknown parameter.
The parameter space of $\mu_0$ is a compact convex subset of $(-2\pi^2, \infty)$ 
and the true value $\mu_0^*$ belongs to its interior. 
The $Q$-Wiener process in SPDE \eqref{2d_spde} is given by
\begin{equation}\label{QW2}
W_t^{Q} = \sum_{ l_1, l_2\ge1} \mu_{ l_1, l_2}^{-\alpha/2} 
e_{ l_1, l_2} w_{ l_1, l_2}(t),
\quad t \ge 0
\end{equation}
with $\alpha \in (0,3)$, $\mu_{l_1, l_2} = \pi^2( l_1^2 + l_2^2) + \mu_0$ 
and $\{w_{ l_1, l_2}\}_{ l_1, l_2 \ge 1}$ are independent $\mathbb R$-valued standard Brownian motions. 
In this paper, we assume that $\alpha$ is known. 
See \cite{Bossert2023arXiv} and \cite{TKU2024arXiv2} for estimation of $\alpha$.
We here omit the discussion of statistical inference for SPDE \eqref{2d_spde} 
driven by the $Q_1$-Wiener process in \cite{TKU2024a}
since it can be discussed in the same way.

\item
SPDE \eqref{2d_spde} has a unique mild solution which is given by
\begin{equation*}
X_t=\ee^{-t A_\theta}X_0 +\epsilon \int_0^t \ee^{-(t-s)A_\theta}\dd W_s^{Q},
\end{equation*}
where $\ee^{-t A_\theta} u 
= \sum_{ l_1, l_2 \ge 1} \ee^{-\lambda_{ l_1, l_2}t}
\langle u, e_{ l_1, l_2}\rangle e_{ l_1, l_2}$ for $u \in L^2(D)$.
$X_t$ is then decomposed as follows.

\begin{equation*}
X_t(y,z)=\sum_{ l_1, l_2\ge1} 
x_{ l_1, l_2}(t)
e_{ l_1}^{(1)}(y)e_{ l_2}^{(2)}(z),
\quad t\ge0,\ y,z\in[0,1]
\end{equation*}
with the coordinate process
\begin{equation}\label{cor-pro}
x_{ l_1, l_2}(t) = \langle X_t, e_{ l_1, l_2} \rangle =
\ee^{-\lambda_{ l_1, l_2} t} \langle X_0, e_{ l_1, l_2} \rangle 
+\epsilon \mu_{ l_1, l_2}^{-\alpha/2} 
\int_0^t \ee^{-\lambda_{ l_1, l_2}(t-s)} \dd w_{ l_1, l_2}(s).
\end{equation}
Note that 
$\{ x_{l_1,l_2} \}_{ l_1, l_2 \ge 1}$ are 
one-dimensional independent processes satisfying the Ornstein-Uhlenbeck dynamics
\begin{equation}\label{OU}
\dd x_{ l_1, l_2}(t) =
-\lambda_{ l_1, l_2} x_{ l_1, l_2}(t)\dd t
+\epsilon \mu_{ l_1, l_2}^{-\alpha/2} \dd w_{ l_1, l_2}(t),
\quad
x_{ l_1, l_2}(0) = \langle X_0, e_{ l_1, l_2} \rangle.
\end{equation}

\item
\textit{Observations $\mathbb X_{M,N}$.}
Suppose that we have the discrete observations
$\mathbb X_{M,N} = \{X_{t_i}(y_j,z_k) \}$ with
$t_i = i\Delta = i/N$, $y_j = j/M_1$, $z_k = k/M_2$
for $i=0,\ldots,N$, $j=0,\ldots,M_1$ and $k=0,\ldots,M_2$.
\begin{enumerate}
\item
Fix $b \in (0,1/2)$.
Let $\mathbb X_{m,N}^{(1)} = \{X_{t_i}(\widetilde y_j,\widetilde z_k) \}$
be a thinned data of $\mathbb X_{M,N}$ such that
\begin{equation*}
b \le \widetilde y_0 < \widetilde y_1 < \cdots < \widetilde y_{m_1} \le 1-b,
\quad
b \le \widetilde z_0 < \widetilde z_1 < \cdots < \widetilde z_{m_2} \le 1-b,
\end{equation*}
$m=m_1 m_2$, $m=\OO(N)$, $N=\OO(m)$. 
For simplicity, 
we set $\widetilde y_0 = \widetilde z_0 = b$, 
$\widetilde y_{m_1} = \widetilde z_{m_2} = 1-b$ and $m_1=m_2$,
that is, 
$\widetilde y_j = b + j\delta$ and $\widetilde z_k = b + k\delta$,
where $\delta = \frac{1-2b}{m_1} = \frac{1-2b}{\sqrt{m}}$.
We also write 
\begin{equation*}
\overline y_j = \frac{\widetilde y_{j-1}+\widetilde y_j}{2},
\quad
\overline z_k = \frac{\widetilde z_{k-1}+\widetilde z_k}{2},
\quad
j=1,\ldots, m_1, k=1,\ldots, m_2.
\end{equation*}

\item
Let $\mathbb X_{M,n}^{(2)} = \{X_{\widetilde t_i}(y_j,z_k) \}$ 
be a thinned data of $\mathbb X_{M,N}$ such that
\begin{equation*}
\widetilde t_i = i \Delta_n = \biggl\lfloor \frac{N}{n} \biggr\rfloor \frac{i}{N},
\quad i = 0,\ldots,n \ (\le N).
\end{equation*}
\end{enumerate}

\item
For a sequence $\{a_n\}$, 
we write $a_n \equiv a$ if $a_n = a$ for some $a \in \mathbb R$ and all $n$.

\item
For $h \in (0,1)$, $a, b >0$ and $c \in \mathbb R$, we define
\begin{equation*}
h^{a \tand b} =  
\begin{cases}
h^{a}, & a < b,
\\
-h^{b} \log h, & a = b,
\\
h^{b}, & a > b,
\end{cases}
\end{equation*}
and $h^{c+ a \tand b} = h^c \cdot h^{a \tand b}$,
$h^{c(a \tand b)} = (h^c)^{a \tand b}$.
Furthermore, for $L \in (1,\infty)$, we write
\begin{equation*}
L^{a \tand b} = \frac{1}{(1/L)^{a \tand b}} =
\begin{cases}
L^{a}, & a < b,
\\
L^{b} /\log L, & a = b,
\\
L^{b}, & a > b.
\end{cases}
\end{equation*}

\end{enumerate}

\subsection{Estimation of the diffusive and advective parameters}\label{sec2-1}
We first consider the estimation of the diffusive and advective parameters 
$\theta_2$, $\theta_1$ and  $\eta_1$ based on the thinned data $\mathbb X_{m,N}^{(1)}$.
Tonaki et al.\,\cite{TKU2024a} utilized a contrast function based on temporal increments
\begin{equation*}
\Delta_i X = X_{t_{i}}-X_{t_{i-1}}
\end{equation*}
and constructed minimum contrast estimators of 
$\theta_2$, $\theta_1$ and $\eta_1$.
In this subsection, we construct a contrast function using 
triple increments 
\begin{align*}
T_{i,j,k}X &= 
X_{t_{i}}(\widetilde y_{j},\widetilde z_{k})
- X_{t_{i}}(\widetilde y_{j-1},\widetilde z_{k})
- \bigl( X_{t_{i}}(\widetilde y_{j},\widetilde z_{k-1})
- X_{t_{i}}(\widetilde y_{j-1},\widetilde z_{k-1}) \bigr)
\\
&\quad- \Bigl( X_{t_{i-1}}(\widetilde y_{j},\widetilde z_{k})
- X_{t_{i-1}}(\widetilde y_{j-1},\widetilde z_{k})
- \bigl( X_{t_{i-1}}(\widetilde y_{j},\widetilde z_{k-1})
- X_{t_{i-1}}(\widetilde y_{j-1},\widetilde z_{k-1}) \bigr) \Bigr).
\end{align*}

Let $J_0$ be the Bessel function of the first kind of order $0$, that is,
\begin{equation*}
J_0(x) = 1 + \sum_{k\ge1} \frac{(-1)^k}{(k!)^2} \biggl(\frac{x}{2}\biggr)^{2k}.
\end{equation*}
For $r, \alpha >0$, we define
\begin{equation}\label{psi}
\phi_{r,\alpha}(\theta_2)
=\frac{2}{\theta_2^{1-\alpha} \pi}
\int_0^\infty 
\frac{1-\ee^{-x^2}}{x^{1+2\alpha}}
\biggl(
J_0\Bigl(\frac{\sqrt{2}r x}{\sqrt{\theta_2}}\Bigr)
-2J_0\Bigl(\frac{r x}{\sqrt{\theta_2}}\Bigr)+1
\biggr) \dd x.
\end{equation}

Let $\alpha_0 \in (0,3)$. We make the following conditions 
of the deterministic initial value $X_0$.
\begin{description}
\item[{[A1]$_{\alpha_0}$}]
$\| A_\theta^{(1+\alpha_0)/2} X_0 \| < \infty$.

\item[{[A2]}]
$X_0 \neq 0$, that is, there exist $\mf l_1, \mf l_2 \ge 1$ such that 
$\lb X_0, e_{\mf l_1,\mf l_2} \rb \neq 0$.
\end{description}
For $\alpha \in (0,3)$, Proposition \ref{prop1} below yields 
\begin{equation}\label{eq-prop1-1}
\frac{1}{\epsilon^2 N \Delta^\alpha}\sum_{i=1}^N \EE[(T_{i,j,k}X)^2]
= \ee^{-\kappa \overline y_j -\eta \overline z_k} \phi_{r,\alpha}(\theta_2)
+ \OO \bigl( \Delta^{1-\alpha+ \alpha \tand 2} 
\lor \epsilon^{-2}\Delta^{1+\alpha_0-\alpha} \bigr).
\end{equation}
In order to asymptotically ignore the remainder in \eqref{eq-prop1-1}, 
we assume the following balance condition of $N$ and $\epsilon$. 
\begin{description}
\item[{[B1]$_{\alpha, \alpha_0}$}]
$\epsilon^2 N^{1+\alpha_0 -\alpha} \to \infty$ as $N \to \infty$ and $\epsilon \to 0$. 
\end{description}
We then consider the contrast function 
\begin{equation*}
\mathcal U_{m,N}^{\epsilon}(\kappa, \eta, \theta_2) = 
\frac{1}{m}\sum_{k=1}^{m_2}\sum_{j=1}^{m_1} 
\Biggl\{
\frac{1}{\epsilon^2 N \Delta^{\alpha}}\sum_{i=1}^{N} (T_{i,j,k}X)^2
-\ee^{-\kappa \overline y_j -\eta \overline z_k} \phi_{r,\alpha}(\theta_2)
\Biggr\}^2.
\end{equation*}
Let $\Xi$ be a compact convex subset of 
$\mathbb R^2 \times (0,\infty)$.
We define the estimators of the coefficient parameters 
$\theta_1$, $\eta_1$ and $\theta_2$ by
\begin{equation*}
(\widehat \kappa, \widehat \eta, \widehat \theta_2) 
= \underset{(\kappa, \eta, \theta_2) \in \Xi}{\mathrm{argmin}}\,
\mathcal U_{m,N}^{\epsilon}(\kappa, \eta, \theta_2),
\end{equation*}
\begin{equation*}
\widehat \theta_1 = \widehat \kappa \widehat \theta_2,
\quad 
\widehat \eta_1 = \widehat \eta \widehat \theta_2.
\end{equation*}
Let
\begin{equation*}
\mathcal R_{\alpha,\alpha_0} = \mathcal R_{\alpha,\alpha_0}(m, N, \epsilon) 
= \frac{\sqrt{m N}}{\Delta^{\alpha \tand 2 -\alpha}
\lor \epsilon^{-2}\Delta^{\alpha_0 -\alpha}}
= \sqrt{m N} (N^{\alpha \tand 2 -\alpha}
\land \epsilon^2 N^{\alpha_0 -\alpha}).
\end{equation*}
We obtain the following result.
\begin{thm}\label{th1}
Let $\alpha, \alpha_0 \in (0,3)$ and $\delta/\sqrt{\Delta} \equiv r \in(0,\infty)$.
Assume that [A1]$_{\alpha_0}$ and [B1]$_{\alpha, \alpha_0}$ hold.
Then, it holds that
\begin{equation*}
\mathcal R_{\alpha,\alpha_0}
\begin{pmatrix}
\widehat \theta_1 -\theta_1^*
\\
\widehat \eta_1 -\eta_1^* 
\\
\widehat \theta_2 -\theta_2^*
\end{pmatrix}
= \OO_\PP(1)
\end{equation*}
as $m \to \infty$, $N \to \infty$ and $\epsilon \to 0$.
\end{thm}

\begin{rmk}
\begin{enumerate}
\item[(i)]
The estimators $\widehat \kappa$ and $\widehat \eta$ also have
$\mathcal R_{\alpha,\alpha_0}$-consistency
as $m \to \infty$, $N \to \infty$ and $\epsilon \to 0$.

\item[(ii)]
For $\alpha \in (0,1)$, we assume that [A1]$_{1}$ holds, 
which corresponds to the assumption of \cite{TKU2024a}.
Let $\widetilde {\mathcal R}_{\alpha}$ be the rate of Theorem 3.5 in \cite{TKU2024a}.
Since
\begin{align*}
\widetilde {\mathcal R}_{\alpha} &= N^{(1/2) \land (1-\alpha)} 
\land \epsilon N^{5/4 -\alpha} \land \frac{\epsilon^2 N^{2-\alpha}}{\log N},
\\
\mathcal R_{\alpha,1} &= \sqrt{m N} \land \epsilon^2 \sqrt{m N^{3 -2\alpha}}
= \OO(N \land \epsilon^2 N^{2-\alpha}),
\end{align*}
we have $|\widetilde {\mathcal R}_{\alpha}| \leq  C |\mathcal R_{\alpha,1}|$
for some $C>0$.
Therefore, we find that the rate of the estimators of  
$\theta_1$, $\eta_1$ and $\theta_2$ 
is better than that of \cite{TKU2024a}.

\item[(iii)]
The estimators $\widehat \theta_1$, $\widehat \eta_1$ and $\widehat \theta_2$
have the convergence rate $\mathcal R_{\alpha,\alpha_0}^{-1} = (m N)^{-1/2}$ 
under the following condition:
\begin{description}
\item[{[B2]$_{\alpha, \alpha_0}$}]
$
\frac{1}{\epsilon^2 N^{\alpha_0 -\alpha}} =\OO(1)$
\end{description}
with $\alpha \in (0,2)$ and $\alpha_0 \in (\alpha, 3)$.
This means that one needs to take $\alpha \in (0,2)$ as well as \cite{TKU2024b}
in order to construct the estimators $\widehat \theta_1$, $\widehat \eta_1$ and $\widehat \theta_2$
with 
$(m N)^{1/2}$-consistency.

\item[(iv)]
Theorem \ref{th1} suggests that for the SPDE model in \cite{TKU2024b}, 
one can construct consistent estimators of the coefficient parameters 
under [A1]$_\alpha$ for $\alpha \in (0,3)$.
\end{enumerate}
\end{rmk}

\subsection{Estimation of the reaction parameter}\label{sec2-2}
Next, we estimate the reaction parameter $\theta_0$ 
using the thinned data $\mathbb X_{M,n}^{(2)}$. 
We can estimate the drift parameter of diffusion processes 
under small dispersion asymptotics 
based on statistical inference for diffusion processes. 
The coordinate process \eqref{cor-pro} is a diffusion process with a small noise, 
but we cannot observe it.
Thus, we construct an approximate process $\widehat {\textbf x}_{l_1,l_2} 
= \{ \widehat x_{l_1,l_2}(\widetilde t_i) \}_{i=1}^n$
of the coordinate process using the above estimators.

Since 
the coordinate process
\begin{equation*}
x_{l_1, l_2}(t)
=2\int_0^1 \int_0^1 X_t(y,z)\sin(\pi  l_1 y) \sin(\pi  l_2 z)
\ee^{(\kappa y +\eta z)/2} \dd y \dd z
\end{equation*}
by the definition of the inner product $\langle \cdot, \cdot \rangle$, 
we set the approximate coordinate process by
\begin{equation*}
\widehat x_{l_1, l_2}(t)
=2\int_0^1 \int_0^1 \Psi_{M} X_t(y,z)\sin(\pi  l_1 y) \sin(\pi  l_2 z)
\ee^{(\widehat \kappa y +\widehat \eta z)/2} \dd y \dd z,
\end{equation*}
where the operator $\Psi_M$ is defined by $\Psi_M f(y,z) = f(y_{j-1}, z_{k-1})$
for $(y,z) \in [y_{j-1}, y_j) \times [z_{k-1}, z_k)$, 
$j=1,\ldots, M_1$, $k=1,\ldots, M_2$.
A simple calculation yields
\begin{equation*}
\widehat x_{l_1, l_2}(t)
= \sum_{j=1}^{M_1} \sum_{k=1}^{M_2} X_t(y_{j-1}, z_{k-1})
\delta_j^{[y]} g_{l_1}(\widehat \kappa) \delta_k^{[z]} g_{l_2}(\widehat \eta),
\end{equation*}
where
\begin{align*}
g_l(x:a) &= 
\frac{\sqrt{2}\ee^{ax/2}}{(a/2)^2+(\pi l)^2}
\biggl( \frac{a}{2} \sin(\pi l x) - \pi l \cos(\pi l x) \biggr),
\quad a, x \in \mathbb R, \ l \in \mathbb N,
\\
\delta_j^{[y]} g_l(a) &= g_l(y_{j}:a) - g_l(y_{j-1}:a),
\quad
\delta_k^{[z]} g_l(a) = g_l(z_{k}:a) - g_l(z_{k-1}:a).
\end{align*}
Based on \eqref{CF} below, we define the contrast function
\begin{equation*}
\mathcal V_n^{\epsilon} (\lambda,\mu:\widehat {\textbf x}_{l_1,l_2})
=\sum_{i=1}^n
\frac{(\widehat x_{l_1,l_2}(\widetilde t_i)
- \ee^{-\lambda \Delta_n} \widehat x_{l_1,l_2}(\widetilde t_{i-1}))^2}
{\frac{\epsilon^2(1-\ee^{-2\lambda \Delta_n})}{2\lambda \mu^\alpha}}
+n\log \frac{1-\ee^{-2\lambda \Delta_n}}{2\lambda \mu^\alpha \Delta_n}.
\end{equation*}
Fix $\mf l_1, \mf l_2 \ge 1$ in [A2].
If $\mu_0$ is known, then we set $\mu_0 = \mu_0^*$ and define
\begin{equation*}
\widehat \lambda_{\mf l_1,\mf l_2}
=\underset{\lambda}{\mathrm{arginf}}\, 
\mathcal V_n^{\epsilon} (\lambda,\mu_{\mf l_1,\mf l_2}^*: 
\widehat {\textbf x}_{\mf l_1,\mf l_2})
\end{equation*}
as the estimator of $\lambda_{\mf l_1,\mf l_2}$, 
or if $\mu_0$ is unknown, then we set
\begin{equation*}
(\widehat \lambda_{\mf l_1,\mf l_2}, \widehat \mu_{\mf l_1,\mf l_2})
=\underset{\lambda,\mu}{\mathrm{arginf}}\, 
\mathcal V_n^{\epsilon} (\lambda,\mu: \widehat {\textbf x}_{\mf l_1,\mf l_2})
\end{equation*}
as the estimator of $(\lambda_{\mf l_1,\mf l_2},\mu_{\mf l_1,\mf l_2})$. 
We then define the estimators
\begin{equation*}
\widehat \theta_0 = \widehat \theta_{0,\mf l_1, \mf l_2}
=-\widehat \lambda_{\mf l_1,\mf l_2}
+\widehat \theta_2 \biggl( \frac{\widehat \kappa^2 + \widehat \eta^2}{4} 
+\pi^2(\mf l_1^2 +\mf l_2^2) \biggr),
\quad
\widehat \mu_0 = \widehat \mu_{0,\mf l_1, \mf l_2} 
= \widehat \mu_{\mf l_1,\mf l_2}- \pi^2(\mf l_1^2 +\mf l_2^2).
\end{equation*}
We also define 
\begin{equation*}
G(\lambda,\mu,x) = 
\frac{1 -\ee^{-2\lambda}}{2\lambda} \mu^\alpha x^2,
\quad 
H(\mu) = \frac{\alpha^2}{2\mu^2},
\quad
I(\lambda, \mu, x) = \diag \{ G(\lambda,\mu,x), H(\mu) \}
\end{equation*}
and
\begin{equation*}
\mathcal G_{\mf l_1, \mf l_2} = G(\lambda_{\mf l_1, \mf l_2}^*,
\mu_{\mf l_1, \mf l_2}^*,x_{\mf l_1, \mf l_2}(0))^{-1},
\quad
\mathcal I_{\mf l_1, \mf l_2} = I(\lambda_{\mf l_1, \mf l_2}^*,
\mu_{\mf l_1, \mf l_2}^*,x_{\mf l_1, \mf l_2}(0))^{-1}.
\end{equation*}
In order to control the asymptotic behavior of the approximate process 
$\widehat {\textbf x}_{l_1,l_2}$, we consider the following conditions. 
\begin{description}
\item[{[C1]$_{\alpha, \alpha_0}$}]
$\frac{n^2}{(M_1 \land M_2)^{2(\alpha_0 \tand 1)}} \to 0$, 
$\frac{n^2 \epsilon^2}{(M_1 \land M_2)^{2(\alpha \tand 1)}} \to 0$,
$\frac{n^{2 -\alpha_0 \tand 2} \lor n^{2-\alpha \tand 1} \epsilon^2}
{\mathcal R_{\alpha,\alpha_0}^2} \to 0$.

\item[{[C2]$_{\alpha, \alpha_0}$}]
$\frac{n^3 \epsilon^2}{(M_1 \land M_2)^{2(\alpha_0 \tand 1)}} \to 0$, 
$\frac{n^3 \epsilon^4}{(M_1 \land M_2)^{2(\alpha \tand 1)}} \to 0$,
$\frac{n^{3 -\alpha_0 \tand 2} \epsilon^2
\lor n^{3-\alpha \tand 1} \epsilon^4}{\mathcal R_{\alpha,\alpha_0}^2} \to 0$.

\item[{[C3]$_{\alpha, \alpha_0}$}]
$\frac{\epsilon^{-4}}{(M_1 \land M_2)^{2(\alpha_0 \tand 1)}} \to 0$, 
$\frac{\epsilon^{-2}}{(M_1 \land M_2)^{2 (\alpha \tand 1)}} \to 0$,
$\frac{ n^{-\alpha_0 \tand 2} \epsilon^{-4}
\lor n^{-\alpha \tand 1} \epsilon^{-2}}{\mathcal R_{\alpha,\alpha_0}^2} \to 0$.

\item[{[C4]$_{\alpha, \alpha_0}$}]
$\frac{n^2 \epsilon^{-2}}{(M_1 \land M_2)^{2(\alpha_0 \tand 1)}} \to 0$, 
$\frac{n^2}{(M_1 \land M_2)^{2 (\alpha \tand 1)}} \to 0$,
$\frac{ n^{2-\alpha_0 \tand 2} \epsilon^{-2} \lor n^{2-\alpha \tand 1}}
{\mathcal R_{\alpha,\alpha_0}^2} \to 0$.

\item[{[C5]$_{\alpha, \alpha_0}$}]
$\frac{n \lor \epsilon^{-2}}{(M_1 \land M_2)^{2(\alpha_0 \tand 1)}} \to 0$, 
$\frac{n \epsilon^2}{(M_1 \land M_2)^{2 (\alpha \tand 1)}} \to 0$,
$\frac{n \lor \epsilon^{-2}}{\mathcal R_{\alpha,\alpha_0}^2} \to 0$.
\end{description}
We then obtain the following theorem.
\begin{thm}\label{th2}
Let $\alpha, \alpha_0 \in (0,3)$.
Assume that [A1]$_{\alpha_0}$, [A2] and [B1]$_{\alpha,\alpha_0}$ hold. 
\begin{enumerate}
\item[(1)]
Suppose that 
$\mu_0 $ is known.
As $n \to \infty$ and $\epsilon \to 0$, it holds that
under [C1]$_{\alpha, \alpha_0}$ and [C2]$_{\alpha, \alpha_0}$, 
\begin{equation*}
\widehat \theta_0 \pto \theta_0^*,
\end{equation*}
and that under [C2]$_{\alpha, \alpha_0}$, [C4]$_{\alpha, \alpha_0}$ 
and [C5]$_{\alpha}$, 
\begin{equation*}
\epsilon^{-1}(\widehat \theta_0 - \theta_0^*) 
\dto N (0,\mathcal G_{\mf l_1, \mf l_2}).
\end{equation*}

\item[(2)]
Suppose that $\mu_0$ is unknown.
As $n \to \infty$ and $\epsilon \to 0$, it holds that
under [C2]$_{\alpha, \alpha_0}$ and [C3]$_{\alpha, \alpha_0}$, 
\begin{equation*}
(\widehat \theta_0, \widehat \mu_0) \pto (\theta_0^*, \mu_0^*),
\end{equation*}
and that under [C2]$_{\alpha, \alpha_0}$-[C5]$_{\alpha, \alpha_0}$, 
\begin{equation*}
\begin{pmatrix}
\epsilon^{-1}(\widehat \theta_0 - \theta_0^*)
\\
\sqrt n(\widehat \mu_0-\mu_0^*)
\end{pmatrix}
\dto N (0, \mathcal I_{\mf l_1, \mf l_2}).
\end{equation*}
\end{enumerate}
\end{thm}

\begin{rmk}
\begin{itemize}
\item[(i)]
The sufficient conditions [C1]$_{\alpha, \alpha_0}$ and [C2]$_{\alpha, \alpha_0}$
in Theorem \ref{th2}-(1) 
are as weak as or weaker than the sufficient conditions 
[C2]$_{\alpha, \alpha_0}$ and [C3]$_{\alpha, \alpha_0}$
in Theorem \ref{th2}-(2) since 
$1 \lor n \epsilon^2 \le n \epsilon^2 \lor (n \epsilon^2)^{-2}$.

\item[(ii)]
We see that the conditions of 
Proposition \ref{prop3} below
are more relaxed than those of Lemma 5.3 in \cite{TKU2024a}.
This implies that the sufficient conditions 
[C1]$_{\alpha, \alpha_0}$-[C5]$_{\alpha, \alpha_0}$
are more relaxed than the conditions in \cite{TKU2024a}.
This is caused by the facts that the rates of the proposed estimators 
$\widehat \theta_1$, $\widehat \eta_1$, $\widehat \theta_2$ are better 
than those of the existing estimators.
\end{itemize}
\end{rmk}

\section{Simulations}\label{sec3}

A numerical solution of SPDE \eqref{2d_spde} is generated by
\begin{equation*}
\tilde X_{t_{i}}(y_j, z_k)
= \sum_{l_1=1}^{L_1} \sum_{l_2=1}^{L_2} 
x_{l_1,l_2}(t_{i}) e_{l_1,l_2}(y_j, z_k), 
\quad i = 1,\ldots, N, j = 1,\ldots, M_1, k = 1,\ldots, M_2
\end{equation*}
with
\begin{equation*}
\dd x_{l_1,l_2}(t) = -\lambda_{l_1,l_2} x_{l_1,l_2}(t) \dd t
+\epsilon \mu_{l_1,l_2}^{-\alpha/2} \dd w_{l_1,l_2}(t), 
\quad
x_{l_1,l_2}(0) = \langle X_0, e_{l_1,l_2} \rangle.
\end{equation*}
In this simulation, the true values of the parameters
$(\theta_0^*, \theta_1^*,\eta_1^*, \theta_2^*) = (0,0.2,0.2,0.2)$
and $\mu_0^*=-19.5$.
We set $N = 10^3$, $M_1 = M_2 = 200$, $L_1 = L_2 = 10^4$,  
$X_0 = 3 e_{1,1}$, $\alpha = 0.5$ and $\epsilon=0.1$.
We used R language to compute the estimators of Theorems \ref{th1} and \ref{th2}.
The number of Monte Carlo simulations is 200.

We first estimated the diffusive and advective parameters 
$\theta_2$, $\theta_1$ and $\eta_1$ using the thinned data $\mathbb X_{m,N}^{(1)}$
with $b = 0.05$ and $m_1 = m_2 \in \{15, 30, 60 \}$.
Since $A_\theta^{(1+\alpha_0)/2} X_0 = 3 \lambda_{1,1}^{(1+\alpha_0)/2} e_{1,1}$
and $\langle X_0, e_{1,1} \rangle = 3$,
the conditions [A1]$_{\alpha_0}$ and [A2] hold with $\alpha_0 = 2.99$.
Moreover, 
the condition [B2]$_{\alpha,\alpha_0}$ 
holds in this example since we have 
$\frac{1}{\epsilon^2 N^{\alpha_0-\alpha}} = 10^{-4.98}$.
Therefore, 
$\mathcal R_{\alpha,\alpha_0} = \sqrt{m N}$.
Table \ref{tab1} contains the simulation results of the means and the standard deviations
of the estimators $\widehat \theta_1$, $\widehat \eta_1$ and $\widehat \theta_2$.

We next estimated the parameters $\theta_0$ and $\mu_0$
base on the thinned data $\mathbb X_{M,n}^{(2)}$ with $n \in \{50, 100 \}$.
Since we have $n \epsilon^2 \le 1$, 
we only check [C3]$_{\alpha,\alpha_0}$ in order to verify 
that the estimators $\widehat \theta_0$ and $\widehat \mu_0$ are consistent. 
In this case, we have
$\frac{\epsilon^{-4}}{(M_1 \land M_2)^{2(\alpha_0 \tand 1)}} = 0.25$, 
$\frac{\epsilon^{-2}}{(M_1 \land M_2)^{2 (\alpha \tand 1)}} = 0.5$
and $\frac{ n^{-\alpha_0 \tand 2} \epsilon^{-4}
\lor n^{-\alpha \tand 1} \epsilon^{-2}}{\mathcal R_{\alpha,\alpha_0}^2} 
\le 6.3 \times 10^{-5}$,
and thus 
[C3]$_{\alpha,\alpha_0}$ holds in this example.
Table \ref{tab1} shows the simulation results of the means and the standard deviations
of the estimators $\widehat \theta_0$ and $\widehat \mu_0$.

We also calculated the estimators
$\widehat \theta_1^{\mathrm{S}}$, 
$\widehat \eta_1^{\mathrm{S}}$, $\widehat \theta_2^{\mathrm{S}}$,
$\widehat \theta_0^{\mathrm{S}}$ and $\widehat \mu_0^{\mathrm{S}}$
of $\theta_1$, $\eta_1$, $\theta_2$, $\theta_0$ and $\mu_0$
proposed by \cite{TKU2024a}
based on the thinned data $\mathbb X_{m,N}^{(1)}$
with $b = 0.05$ and $m_1 = m_2 \in \{15, 30, 60 \}$
and the thinned data $\mathbb X_{M,n}^{(2)}$ with $n \in \{50, 100 \}$.
Table \ref{tab2} shows the simulation results of the means and the standard deviations
of the estimators $\widehat \theta_1^{\mathrm{S}}$, 
$\widehat \eta_1^{\mathrm{S}}$, $\widehat \theta_2^{\mathrm{S}}$,
$\widehat \theta_0^{\mathrm{S}}$ and $\widehat \mu_0^{\mathrm{S}}$.

We observe that for all cases $m_1=m_2 \in \{15,30,60\}$ and $n = \{50,100\}$,
the biases of $\widehat \theta_1$, $\widehat \eta_1$, $\widehat \theta_2$,
$\widehat \theta_0$ and $\widehat \mu_0$
are smaller than those of $\widehat \theta_1^{\mathrm{S}}$, 
$\widehat \eta_1^{\mathrm{S}}$, $\widehat \theta_2^{\mathrm{S}}$,
$\widehat \theta_0^{\mathrm{S}}$ and $\widehat \mu_0^{\mathrm{S}}$, respectively.

\begin{table}[h]
\caption{Means and standard deviations of the proposed estimators with $b =0.05$}
\label{tab1}
\captionsetup{margin=5pt}
\begin{center}
\begin{tabular*}{0.8\textwidth}{@{\extracolsep{\fill}}ccccccc}\hline
 & $\widehat \theta_1$ & $\widehat \eta_1$ & $\widehat \theta_2$ 
& & $\widehat \theta_0$ & $\widehat \mu_0$ \rule[0mm]{0cm}{4.5mm} \\ \cline{2-4}\cline{6-7} 
$m_1$, $m_2$ & $\theta_1^*=0.2$ & $\eta_1^*=0.2$ & $\theta_2^*=0.2$ 
& $n$ & $\theta_0^*=0$ & $\mu_0^*=-19.5$ \rule[0mm]{0cm}{4.5mm} \\ \hline 
 & & & & \multirow{2}{*}{$50$} & $0.032$ & $-19.467$ \\
 \multirow{2}{*}{$15$} 
 & $0.201$ & $0.202$ & $0.201$ & & ($0.1489$) & ($0.1136$) \\ \cline{6-7}
 & ($0.0031$) & ($0.0029$) & ($0.0036$) 
 & \multirow{2}{*}{$100$} & $0.033$ & $-19.483$ 
 \\
  & & & & & ($0.1481$) & ($0.0760$) \\ \hline
 & & & & \multirow{2}{*}{$50$} & $0.056$ & $-19.466$ \\
 \multirow{2}{*}{$30$}  
 & $0.202$ & $0.202$ & $0.202$ & & ($0.1322$) & ($0.1155$) \\ \cline{6-7}
 & ($0.0010$) & ($0.0009$) & ($0.0015$) 
 & \multirow{2}{*}{$100$} & $0.057$ & $-19.483$ 
 \\
 & & & & & ($0.1315$) & ($0.0770$) \\ \hline
 & & & & \multirow{2}{*}{$50$} & $0.062$ & $-19.465$ \\
 \multirow{2}{*}{$60$}  
 & $0.203$ & $0.203$ & $0.203$ & & ($0.1295$) & ($0.1156$) \\ \cline{6-7}
 & ($0.0003$) & ($0.0004$) & ($0.0004$) 
 & \multirow{2}{*}{$100$} & $0.063$ & $-19.482$ 
 \\
 & & & & & ($0.1288$) & ($0.0768$) \\ \hline
\end{tabular*}
\end{center}
\end{table}

\begin{table}[h]
\caption{Means and standard deviations of the estimators in \cite{TKU2024a} 
with $b =0.05$}
\label{tab2}
\captionsetup{margin=5pt}
\begin{center}
\begin{tabular*}{0.8\textwidth}{@{\extracolsep{\fill}}ccccccc}\hline
 & $\widehat \theta_1^{\mathrm{S}}$ & $\widehat \eta_1^{\mathrm{S}}$ & 
 $\widehat \theta_2^{\mathrm{S}}$ 
& & $\widehat \theta_0^{\mathrm{S}}$ 
& $\widehat \mu_0^{\mathrm{S}}$ \rule[0mm]{0cm}{4.5mm} \\ \cline{2-4}\cline{6-7} 
$m_1$, $m_2$ & $\theta_1^*=0.2$ & $\eta_1^*=0.2$ & $\theta_2^*=0.2$ 
& $n$ & $\theta_0^*=0$ & $\mu_0^*=-19.5$ \rule[0mm]{0cm}{4.5mm} \\ \hline 
 & & & & \multirow{2}{*}{$50$} & $-0.114$ & $-19.359$ \\
 \multirow{2}{*}{$15$} 
 & $0.163$ & $0.165$ & $0.195$ & & ($0.1646$) & ($0.1592$) \\ \cline{6-7}
 & ($0.0024$) & ($0.0024$) & ($0.0041$) 
 & \multirow{2}{*}{$100$} & $-0.114$ & $-19.383$ 
 \\
  & & & & & ($0.1639$) & ($0.1049$) \\ \hline
 & & & & \multirow{2}{*}{$50$} & $0.165$ & $-19.319$ \\
 \multirow{2}{*}{$30$}  
 & $0.164$ & $0.167$ & $0.210$ & & ($0.1560$) & ($0.1763$) \\ \cline{6-7}
 & ($0.0019$) & ($0.0020$) & ($0.0029$) 
 & \multirow{2}{*}{$100$} & $0.166$ & $-19.345$ 
 \\
 & & & & & ($0.1554$) & ($0.1169$) \\ \hline
 & & & & \multirow{2}{*}{$50$} & $0.313$ & $-19.298$ \\
 \multirow{2}{*}{$60$}  
 & $0.165$ & $0.168$ & $0.217$ & & ($0.1543$) & ($0.1854$) \\ \cline{6-7}
 & ($0.0019$) & ($0.0020$) & ($0.0028$) 
 & \multirow{2}{*}{$100$} & $0.313$ & $-19.325$ 
 \\
 & & & & & ($0.1537$) & ($0.1228$) \\ \hline
\end{tabular*}
\end{center}
\end{table}

\section{Proofs}\label{sec4}
We set the following notation.
\begin{enumerate}
\item[1.]
Let $\mathbb R_{+}=(0,\infty)$.

\item[2.]
For $a,b \in \mathbb R$, 
we write $a \lesssim b$ if $|a| \le C|b|$ for some constant $C>0$.

\item[3.]
For two functions $f, g : \mathbb R^d \to \mathbb R$,
we write $f(x) \lesssim g(x)$ $(x \to a)$ 
if $f(x) \lesssim g(x)$ in a neighborhood of $x = a$. 

\item[4.]
For $x=(x_1,\ldots, x_d) \in \mathbb R^d$ and $f:\mathbb R^d \to \mathbb R$,
we write $\pd_{x_i} f(x) = \frac{\pd}{\pd x_i}f(x)$,
$\pd_x f(x) = (\pd_{x_1}f(x), \ldots, \pd_{x_d}f(x))$ and 
$\pd_x^2 f(x) = (\pd_{x_j}\pd_{x_i}f(x))_{i,j=1}^d$.

\item[5.]
Let $\ind_A$ be the indicator function of $A$.

\item[6.]
For a function $f:\mathbb R \to \mathbb R$ and a positive number $h$, we write
$D_h f(x) = f(x+h) - f(x)$.
Note that $D_h^2 f(x) =D_h[D_h f](x) =  f(x+2h) - 2f(x+h) +f(x)$.

\item[7.]
For a function $f:\mathbb R^2 \to \mathbb R$ and a positive number $h$, define
\begin{equation*}
D_{1,h} f(x,y) = f(x+h,y) - f(x,y),
\quad 
D_{2,h} f(x) = f(x,y+h) - f(x,y).
\end{equation*} 
\end{enumerate}

Let $\beta_1, \beta_2 > 0$, $\gamma,\gamma_1,\gamma_2,J \ge0$ 
and $p \in \mathbb N \cup \{0\}$. 
\begin{enumerate}
\item[8.]
Let $\boldsymbol F_{\beta_1,\beta_2}^J$ be 
the subspace of functions in $C^2(\mathbb R_{+})$, 
whose element $f$ satisfies
\begin{itemize}
\item[(i)]
$x^{\beta_1} f(x)$, $x^{\beta_1+1} f'(x)$, 
$x^{\beta_1+2} f''(x) \lesssim \ee^{-J x}$ ($x \to 0$), 

\item[(ii)]
$x^{\beta_2 +1} f(x)$, $x^{\beta_2 +2} f'(x)$,  
$x^{\beta_2 +3} f''(x) \lesssim \ee^{-J x}$ ($x \to \infty$).
\end{itemize}
We shall set $\boldsymbol F_{\beta_1,\beta_2} = \boldsymbol F_{\beta_1,\beta_2}^0$.

\item[9.]
Let $\boldsymbol G^p_\gamma$ be 
the subspace of functions in $C^p(\mathbb R^2)$,
whose element $g$ satisfies 
\begin{itemize}
\item[(i)]
$|\pd^j g(x,y)| \lesssim (x^2+y^2)^{\gamma -j/2}$ ($x, y \to 0$) for $j=0,\ldots, p$,

\item[(ii)]
$|\pd^j g(x,y)| \lesssim 1$ ($x, y \to \infty$) for $j=0,\ldots, p$.
\end{itemize}
We shall write $\boldsymbol G_\gamma = \boldsymbol G^0_\gamma$.

\item[10.]
Let $\widetilde {\boldsymbol G}_{\gamma_1,\gamma_2}$ 
be the subspace of functions in $C(\mathbb R^2)$,
whose element $g$ satisfies 
\begin{itemize}
\item[(i)]
there exist $g_1,g_2 \in C(\mathbb R)$ such that $g(x,y) = g_1(x)g_2(y)$,

\item[(ii)]
$g_j(s) \lesssim s^{2\gamma_j}$ ($s \to 0$) for $j=1,2$,

\item[(iii)]
$g_j(s) \lesssim 1$ ($s \to \infty$) for $j=1,2$. 
\end{itemize}

\item[11.]
Let $\boldsymbol B_b$
be the space of bounded and measurable functions on $\mathbb R^2$.

\end{enumerate}

\subsection{Proof of Theorem \ref{th1}}
The lemmas used in the proof can be found at the end of this subsection.

For $\beta \in (0,3)$ and $J \ge 0$, we define 
\begin{equation*}
\psi(x) = \frac{1-\ee^{-x}}{x},
\quad
\psi_J(x) = \frac{(1-\ee^{-x})^2 \ee^{-J x}}{x},
\quad
h_\beta(x) = x^{-\beta},
\quad 
x > 0.
\end{equation*}
\begin{equation*}
\widetilde f_\beta(x, y) = \psi(x) h_\beta(y),
\quad
\widetilde f_{J,\beta}(x, y) = \psi_J(x) h_\beta(y),
\quad
f_\beta(x) = \widetilde f_\beta(x, x),
\quad
f_{J,\beta}(x) = \widetilde f_{J,\beta}(x, x).
\end{equation*}
We can easily see that $f_\beta \in \boldsymbol F_{\beta, \beta}$ and
$f_{J,\beta} \in \boldsymbol F^{J/2}_{\beta -1, \beta}$.
For $y,z \in (0,2)$, we define
\begin{align*}
F_{\beta,\Delta}(y,z) &=
\sum_{ l_1, l_2\ge1}
\frac{1-\ee^{-\lambda_{ l_1, l_2}\Delta}}
{\lambda_{l_1, l_2} \mu_{l_1,l_2}^{\beta}}
\cos(\pi l_1 y)\cos(\pi l_2 z)
\\
&= \Delta^{1+\beta}
\sum_{ l_1, l_2\ge1}
\widetilde f_\beta(\lambda_{l_1,l_2} \Delta, \mu_{l_1,l_2} \Delta) 
\cos(\pi l_1 y) \cos(\pi l_2 z),
\\
\Phi_{j',k'}^{j,k}(\Delta;\beta) &=
\Delta^{1+\beta} \sum_{l_{1}, l_{2} \ge 1}
\widetilde f_{\beta}(\lambda_{l_1, l_2} \Delta, \mu_{l_1, l_2} \Delta)
\\
&\qquad\times
(e_{ l_1}^{(1)}(\widetilde y_j)-e_{ l_1}^{(1)}(\widetilde y_{j-1}))
(e_{ l_2}^{(2)}(\widetilde z_k)-e_{ l_2}^{(2)}(\widetilde z_{k-1}))
\\
&\qquad\times
(e_{ l_1}^{(1)}(\widetilde y_{j'})-e_{ l_1}^{(1)}(\widetilde y_{j'-1}))
(e_{ l_2}^{(2)}(\widetilde z_{k'})-e_{ l_2}^{(2)}(\widetilde z_{k'-1})),
\\
\Phi_{J,j',k'}^{j,k}(\Delta;\beta) &=
\Delta^{1+\beta} \sum_{ l_{1}, l_{2}\ge1}
\widetilde f_{J,\beta}(\lambda_{l_1, l_2} \Delta, \mu_{l_1, l_2} \Delta)
\\
&\qquad\times
(e_{ l_1}^{(1)}(\widetilde y_j)-e_{ l_1}^{(1)}(\widetilde y_{j-1}))
(e_{ l_2}^{(2)}(\widetilde z_k)-e_{ l_2}^{(2)}(\widetilde z_{k-1}))
\\
&\qquad\times
(e_{ l_1}^{(1)}(\widetilde y_{j'})-e_{ l_1}^{(1)}(\widetilde y_{j'-1}))
(e_{ l_2}^{(2)}(\widetilde z_{k'})-e_{ l_2}^{(2)}(\widetilde z_{k'-1})).
\end{align*}
We then obtain the following two propositions. 
\begin{prop}\label{prop1}
Let $\alpha, \alpha_0 \in (0,3)$ and $\delta/\sqrt{\Delta} \equiv r \in(0,\infty)$.
Under [A1]$_{\alpha_0}$, it holds that
\begin{equation*}
\EE[(T_{i,j,k}X)^2]
=\epsilon^2 \bigl(\Delta^\alpha 
\ee^{-\kappa \overline y_j -\eta \overline z_k}
\phi_{r,\alpha}(\theta_2)
+ R_{i,j,k}(\alpha) + \OO(\Delta^{1+ \alpha \tand 2}) \bigr) 
+ R_{i,j,k}(\alpha_0),
\end{equation*}
where 
the remainder $R_{i,j,k}(\beta)$ for  $\beta \in (0,3)$ satisfies 
$\sum_{i=1}^N R_{i,j,k}(\beta) = \OO(\Delta^{\beta})$ uniformly in $j,k$.
\end{prop}

\begin{prop}\label{prop2}
Let $\alpha \in (0,3)$. 
Assume that [A1]$_{\alpha_0}$ with $\alpha_0 \in (0,3)$ holds. 
Then, for $\delta = r \sqrt{\Delta}$, it holds that
\begin{itemize}
\item[(1)]
$\cov[T_{i,j,k}X, T_{i',j',k'}X]
= \epsilon^2 \times 
\begin{cases}
\Phi_{j',k'}^{j,k}(\Delta;\alpha) 
-\frac{1}{2}\Phi_{2(i-1),j',k'}^{j,k}(\Delta;\alpha), & i = i',
\\
-\frac{1}{2}(\Phi_{|i-i'|-1,j',k'}^{j,k}(\Delta;\alpha) 
+ \Phi_{i+i'-2,j',k'}^{j,k}(\Delta;\alpha)), & i \neq i',
\end{cases}
$

\item[(2)]
$\cov[T_{i,j,k}X, T_{i',j',k'}X]
\\=
\OO \Bigl(
\frac{\epsilon^2 \Delta^{\alpha \tand 2}}{|i-i'|+1}
\Bigl(
\Delta
+ \frac{1}{(|j-j'|+1)(|k-k'|+1)}
\Bigr) \Bigr)
+\OO \Bigl(
\frac{\epsilon^2 \Delta^{1/2 +\alpha \tand 2}}{|i-i'|+1}
\Bigl(
\frac{\ind_{\{j \neq j'\}}}{|j-j'|+1} +\frac{\ind_{\{k \neq k'\}}}{|k-k'|+1}
\Bigr) \Bigr)$,

\item[(3)]
$\displaystyle \sum_{i,i'=1}^N \cov[T_{i,j,k}X, T_{i',j,k}X]^2
=\OO(\epsilon^4 N \Delta^{2\alpha})$
uniformly in $j, k$,

\item[(4)]
$\displaystyle \sum_{k,k'=1}^{m_2} \sum_{j,j'=1}^{m_1} \sum_{i,i'=1}^N 
\cov[T_{i,j,k}X, T_{i',j',k'}X]^2
=\OO \bigl( \epsilon^4 m N (\Delta^{\alpha \tand 2})^2 \bigr)$,

\item[(5)]
$\displaystyle \sum_{i,i'=1}^N \cov[(T_{i,j,k}X)^2, (T_{i',j,k}X)^2 ]
=\OO \bigl(\epsilon^4 N \Delta^{2\alpha}
( \epsilon^{-2} \Delta^{\alpha_0 -\alpha +1/2} \lor 1 )\bigr)$ uniformly in $j, k$,

\item[(6)]
$\displaystyle \sum_{k,k'=1}^{m_2} \sum_{j,j'=1}^{m_1} \sum_{i,i'=1}^N 
\cov[(T_{i,j,k}X)^2, (T_{i',j',k'}X)^2]
=\OO \bigl(\epsilon^4 m N (\Delta^{\alpha \tand 2})^2 
(\epsilon^{-2} \Delta^{\alpha_0 -\alpha \tand 2} \lor 1 ) \bigr)$.
\end{itemize}
\end{prop}

We will postpone the proofs of Propositions \ref{prop1} and \ref{prop2}, 
and first show Theorem \ref{th1} using the above results.
For $b \in (0,1/2)$ and $u,v \in L^2((b,1-b)^2)$, we denote
\begin{equation*}
\langle u,v \rangle_b = 
\frac{1}{(1-2b)^2} \int_b^{1-b}\int_b^{1-b} u(x,y) v(x,y)\dd x \dd y,
\quad
\| u \|_b = \sqrt{\langle u,u \rangle_b}.
\end{equation*}
Let $\vartheta = (\kappa, \eta ,\theta_2)$ and 
$\vartheta^* = (\kappa^*, \eta^* ,\theta_2^*)$. We write 
$h_{r,\alpha}(y,z:\vartheta) = \ee^{-\kappa y -\eta z} \phi_{r,\alpha}(\theta_2)$,
\begin{align*}
\boldsymbol U(\vartheta) &= 
\Bigl(
\bigl\langle 
\pd_{\vartheta_p} h_{r,\alpha}(\cdot,\cdot:\vartheta), 
\pd_{\vartheta_q} h_{r,\alpha}(\cdot,\cdot:\vartheta)
\bigr\rangle_b
\Bigr)_{1\le p,q \le 3},
\\
\mathcal U(\vartheta,\vartheta^*) 
&= \|h_{r,\alpha}(\cdot,\cdot:\vartheta) -h_{r,\alpha}(\cdot,\cdot:\vartheta^*) \|_b^2.
\end{align*}
Using the mean value theorem, we obtain
\begin{equation*}
-\mathcal R_{\alpha,\alpha_0} \pd_{\vartheta} \mathcal U_{m,N}^{\epsilon}(\vartheta^*)^\TT
=\int_0^1 \pd_{\vartheta}^2 
\mathcal U_{m,N}^{\epsilon}(\vartheta^* +u(\widehat\vartheta -\vartheta^*)) \dd u 
\mathcal R_{\alpha,\alpha_0} (\widehat\vartheta-\vartheta^*).
\end{equation*}
To complete the proof, we will show that
\begin{enumerate}
\item[(i)]
$\mathcal U(\vartheta,\vartheta^*)$ takes its unique minimum in 
$\vartheta = \vartheta^*$,

\item[(ii)]
$\mathcal U_{m,N}^{\epsilon}(\vartheta) \pto \mathcal U(\vartheta,\vartheta^*)$ 
uniformly in $\vartheta \in \Xi$
as $\epsilon \to 0$, $m \to \infty$ and $N \to \infty$,

\item[(iii)]
$\mathcal R_{\alpha,\alpha_0} 
\pd_{\vartheta} \mathcal U_{m,N}^{\epsilon}(\vartheta^*) = \OO_\PP(1)$,

\item[(iv)]
$\displaystyle \sup_{|\vartheta-\vartheta^*|\le \delta_{m,N}^{\epsilon}}
|\pd^2 \mathcal U_{m,N}^{\epsilon}(\vartheta) - 2\boldsymbol U(\vartheta^*)| \pto 0$
for $\delta_{m,N}^{\epsilon} \to 0$ 
as $\epsilon \to 0$, $m \to \infty$ and $N \to \infty$,

\item[(v)]
$\boldsymbol U = \boldsymbol U(\vartheta^*)$ is strictly positive definite.
\end{enumerate}


\begin{proof}[Proof of (i)]
We show that $\mathcal U(\vartheta,\vartheta^*)=0$ if and only if 
$\vartheta=\vartheta^*$.
Since 
\begin{equation*}
\phi_{r,\alpha}(\theta_2)\ee^{-\kappa y -\eta z} 
= h_{r,\alpha}(y,z:\vartheta) 
= h_{r,\alpha}(y,z:\vartheta^*)
= \phi_{r,\alpha}(\theta_2^*)\ee^{-\kappa^* y -\eta^* z}
\end{equation*}
for any $y,z \in [b,1-b]$ if $\mathcal U(\vartheta,\vartheta^*)=0$
and the function $(0, \infty) \ni \theta_2 \mapsto \phi_{r,\alpha}(\theta_2)$
is injective for any $r > 0$ and $\alpha \in (0,3)$
according to Lemma \ref{lem11} below, we have $\vartheta = \vartheta^*$.
It is obvious that $\mathcal U(\vartheta,\vartheta^*)=0$ if $\vartheta = \vartheta^*$.
\end{proof}

\begin{proof}[Proof of (ii)]
Let 
\begin{equation*}
Z_{i,j,k} =
\frac{(T_{i,j,k} X)^2}{\epsilon^2 \Delta^\alpha} 
-h_{r,\alpha}(\overline y_j, \overline z_k:\vartheta^*),
\quad
\mathcal Z_{j,k} = \frac{1}{N}\sum_{i=1}^N Z_{i,j,k},
\end{equation*}
and $\varphi_{r,\alpha}(y,z:\vartheta) = 
h_{r,\alpha}(y,z:\vartheta^*) - h_{r,\alpha}(y,z:\vartheta)$.
We then have
\begin{equation*}
\mathcal U_{m,N}^{\epsilon}(\vartheta)
=\frac{1}{m}\sum_{k=1}^{m_2} \sum_{j=1}^{m_1} \mathcal Z_{j,k}^2
+\frac{2}{m}\sum_{k=1}^{m_2} \sum_{j=1}^{m_1} \mathcal Z_{j,k} 
\varphi_{r,\alpha}(\overline y_j, \overline z_k:\vartheta)
+\frac{1}{m}\sum_{k=1}^{m_2} \sum_{j=1}^{m_1}
\varphi_{r,\alpha}^2(\overline y_j, \overline z_k:\vartheta),
\end{equation*}
and
\begin{align}
\sup_{\vartheta \in \Xi}|\mathcal U_{m,N}^{\epsilon}(\vartheta)-\mathcal U(\vartheta,\vartheta^*)|
&\le
\frac{1}{m}\sum_{k=1}^{m_2} \sum_{j=1}^{m_1} \mathcal Z_{j,k}^2
+2\sup_{\vartheta \in \Xi}
\Biggl|
\frac{1}{m}\sum_{k=1}^{m_2} \sum_{j=1}^{m_1} \mathcal Z_{j,k} 
\varphi_{r,\alpha}(\overline y_j, \overline z_k:\vartheta)
\Biggr|
\nonumber
\\
&\qquad
+\sup_{\vartheta \in \Xi}
\Biggl|
\frac{1}{m}\sum_{k=1}^{m_2} \sum_{j=1}^{m_1}
\varphi_{r,\alpha}^2(\overline y_j, \overline z_k:\vartheta)
-\mathcal U(\vartheta,\vartheta^*)
\Biggr|.
\label{eq-2-4}
\end{align}
In order to get the desired result, we will show that 
all terms of \eqref{eq-2-4} converge to $0$ in probability.
Since the function $(y,z,\vartheta) \mapsto \varphi_{r,\alpha}(y,z:\vartheta)$ is
continuous on the compact set $[b,1-b]^2 \times \Xi$
and we can write $\mathcal U(\vartheta, \vartheta^*) 
= \frac{1}{(1 -2b)^2} 
\int_b^{1-b} \int_b^{1-b} \varphi_{r,\alpha}^2(y,z:\vartheta) \dd y \dd z$, 
the last term of \eqref{eq-2-4} converges to $0$ as $m \to \infty$.
Define $\mathcal S_{\Delta, \epsilon}(\alpha,\alpha_0)
= \Delta^{1-\alpha +\alpha \tand 2} 
\lor \epsilon^{-2}\Delta^{1+\alpha_0-\alpha}$. 
By Proposition \ref{prop1}, one has 
\begin{equation*}
Z_{i,j,k} =
\epsilon^{-2} \Delta^{-\alpha}
\bigl((T_{i,j,k} X)^2-\EE[(T_{i,j,k} X)^2] \bigr) + r_{i,j,k}(\alpha,\alpha_0),
\end{equation*}
where 
$\sum_{i=1}^N r_{i,j,k}(\alpha,\alpha_0) 
= \OO(\Delta^{-1} \mathcal S_{\Delta, \epsilon}(\alpha,\alpha_0))$ uniformly in $j,k$.
Noting that
$\mathcal S_{\Delta, \epsilon}(\alpha,\alpha_0) = \oo(1)$ 
under [B1]$_{\alpha, \alpha_0}$ with $\alpha, \alpha_0 \in (0,3)$
and $\varphi_{r,\alpha}$ is bounded, 
one sees from [B1]$_{\alpha, \alpha_0}$, Proposition \ref{prop2}-(5) and the Schwarz inequality
that for $\alpha, \alpha_0 \in (0,3)$, 
\begin{align*}
\EE[\mathcal Z_{j,k}^2]
&=\frac{1}{N^2}\sum_{i,i'=1}^N \EE[Z_{i,j,k} Z_{i',j,k}]
\\
&= \frac{1}{N^2}\sum_{i,i'=1}^N 
\bigl(\epsilon^{-4} \Delta^{-2\alpha}\cov[(T_{i,j,k} X)^2,(T_{i',j,k} X)^2]
+r_{i,j,k}(\alpha,\alpha_0)r_{i',j,k}(\alpha,\alpha_0) \bigr)
\\
&=\OO \biggl(
\frac{\epsilon^{-2} \Delta^{\alpha_0 -\alpha +1/2} \lor 1}{N}  
\biggr)
+\OO \biggl( 
\Bigl(\frac{\Delta^{-1} \mathcal S_{\Delta,\epsilon}(\alpha,\alpha_0)}{N} \Bigr)^2
\biggr)
\\
&=\OO (\epsilon^{-2} \Delta^{1 +\alpha_0 -\alpha} \cdot \Delta^{1/2} \lor \Delta)
+\OO \bigl( 
\mathcal S_{\Delta,\epsilon}(\alpha,\alpha_0)^2
\bigr)
\\
&=\oo(1)
\end{align*}
uniformly in $j,k$, and
\begin{align*}
\EE\Biggl[
\sup_{\vartheta \in \Xi}
\biggl|\frac{1}{m}
\sum_{k=1}^{m_2} \sum_{j=1}^{m_1} 
\mathcal Z_{j,k} \varphi_{r,\alpha}(\overline y_j, \overline z_k:\vartheta)\biggr|^2
\Biggr]
&\lesssim
\Biggl(
\frac{1}{m} \sum_{k=1}^{m_2} \sum_{j=1}^{m_1} \EE[\mathcal Z_{j,k}^2]^{1/2}
\Biggr)^2
=\oo(1),
\end{align*}
which mean the first and second terms of \eqref{eq-2-4} converge to $0$ in probability.
\end{proof}

\begin{proof}[Proof of (iii)]
We find from Proposition \ref{prop1} and Proposition \ref{prop2}-(6) that
\begin{align*}
&\VV \Biggl[
\sum_{k=1}^{m_2} \sum_{j=1}^{m_1} \sum_{i=1}^N 
(T_{i,j,k}X)^2 \pd_\vartheta h_{r,\alpha}(\overline y_j, \overline z_k:\vartheta^*)^\TT
\Biggr]
\\
&=
\sum_{k,k'=1}^{m_2} \sum_{j,j'=1}^{m_1} \sum_{i,i'=1}^N 
\cov[(T_{i,j,k}X)^2,(T_{i',j',k'}X)^2] 
\pd_\vartheta h_{r,\alpha}(\overline y_j, \overline z_k:\vartheta^*)^\TT
\pd_\vartheta h_{r,\alpha}(\overline y_{j'}, \overline z_{k'}:\vartheta^*)
\\
&=\OO(\epsilon^4 m N (\Delta^{\alpha \tand 2})^2
(1 \lor \epsilon^{-2} \Delta^{\alpha_0 -\alpha \tand 2}))
\end{align*}
and 
\begin{equation*}
\sum_{i=1}^N 
\Bigl( \EE[(T_{i,j,k}X)^2] 
- \epsilon^2 \Delta^\alpha h_{r,\alpha}(\overline y_j, \overline z_k:\vartheta^*) \Bigr) 
= \OO(\epsilon^2 \Delta^{\alpha -1} \mathcal S_{\Delta,\epsilon}(\alpha,\alpha_0))
\end{equation*}
uniformly in $j,k$. Therefore, we have
\begin{align*}
&\mathcal R_{\alpha,\alpha_0} 
\pd_{\vartheta} \mathcal U_{m,N}^{\epsilon} (\vartheta^*)
\\
&=
\mathcal R_{\alpha,\alpha_0}
\times \frac{-2}{m}\sum_{k=1}^{m_2} \sum_{j=1}^{m_1}
\Biggl(
\frac{1}{\epsilon^2 N \Delta^\alpha}\sum_{i=1}^N (T_{i,j,k}X)^2 
- h_{r,\alpha}(\overline y_j, \overline z_k:\vartheta^*)
\Biggr)
\pd_\vartheta h_{r,\alpha}(\overline y_j, \overline z_k:\vartheta^*)
\\
&=
\frac{-2 \mathcal R_{\alpha,\alpha_0}}{\epsilon^2 m N \Delta^\alpha}
\sum_{k=1}^{m_2} \sum_{j=1}^{m_1} 
\sum_{i=1}^N 
\Bigl( (T_{i,j,k}X)^2 - \EE[(T_{i,j,k}X)^2] \Bigr)
\pd_\vartheta h_{r,\alpha}(\overline y_j, \overline z_k:\vartheta^*)
\\
&\qquad
-\frac{2 \mathcal R_{\alpha,\alpha_0}}{\epsilon^2 m N \Delta^\alpha}
\sum_{k=1}^{m_2} \sum_{j=1}^{m_1} 
\sum_{i=1}^N 
\Bigl( \EE[(T_{i,j,k}X)^2] 
- \epsilon^2 \Delta^\alpha h_{r,\alpha}(\overline y_j, \overline z_k:\vartheta^*) \Bigr)
\pd_\vartheta h_{r,\alpha}(\overline y_j, \overline z_k:\vartheta^*)^\TT
\\
&=\OO_\PP \biggl(\frac{1}
{\sqrt{1 \lor \epsilon^{-2} \Delta^{\alpha_0 -\alpha \tand 2}}} \biggr) 
+\OO(1)
\\
&= \OO_\PP(1).
\end{align*}
\end{proof}

\begin{proof}[Proof of (iv)]
We have the representation
\begin{align*}
\pd_{\vartheta}^2 \mathcal U_{m,N}^{\epsilon}(\vartheta) &=
\frac{2}{m}\sum_{k=1}^{m_2}\sum_{j=1}^{m_1}
\pd_\vartheta h_{r,\alpha}(\overline y_j, \overline z_k:\vartheta)^\TT
\pd_\vartheta h_{r,\alpha}(\overline y_j, \overline z_k:\vartheta)
\nonumber
\\
&\qquad
-\frac{2}{m}\sum_{k=1}^{m_2}\sum_{j=1}^{m_1}\Biggl(
\frac{1}{\epsilon^2 N \Delta^\alpha}\sum_{i=1}^N (T_{i,j,k}X)^2 
-h_{r,\alpha}(\overline y_j, \overline z_k:\vartheta)
\Biggr)\pd_\vartheta^2 h_{r,\alpha}(\overline y_j, \overline z_k:\vartheta)
\nonumber
\\
&=:
\frac{1}{m}\sum_{k=1}^{m_2}\sum_{j=1}^{m_1}
\biggl\{
g_1(\overline y_j, \overline z_k:\vartheta)
+\biggl(
\frac{1}{\epsilon^2 N \Delta^\alpha}\sum_{i=1}^N (T_{i,j,k}X)^2 
\biggr)g_2(\overline y_j, \overline z_k:\vartheta)
\biggr\}.
\end{align*}
Noting that the function $(y,z,\vartheta) \mapsto g_ l(y,z:\vartheta)$ is continuous 
on the compact set $[b,1-b]^2 \times \Xi$ and 
\begin{equation*}
\sup_{m,N}
\EE\Biggl[
\frac{1}{\epsilon^2 m N\Delta^\alpha}
\sum_{k=1}^{m_2}\sum_{j=1}^{m_1}\sum_{i=1}^N (T_{i,j,k}X)^2
\Biggr]
<\infty
\end{equation*}
under [B1]$_{\alpha,\alpha_0}$, we obtain 
\begin{align*}
&\sup_{|\vartheta-\vartheta^*|\le \delta_{m,N}^{\epsilon}}
|\pd_{\vartheta}^2 \mathcal U_{m,N}^{\epsilon}(\vartheta)
- \pd_{\vartheta}^2 \mathcal U_{m,N}^{\epsilon}(\vartheta^*)|
\\
&\le
\sup_{\substack{y,z \in [b,1-b],\\ |\vartheta-\vartheta^*|\le \delta_{m,N}^{\epsilon}}}
|g_1(y,z:\vartheta) - g_1(y,z:\vartheta^*)|
\\
&\qquad
+
\Biggl(
\frac{1}{\epsilon^2 m N \Delta^\alpha}
\sum_{k=1}^{m_2}\sum_{j=1}^{m_1}\sum_{i=1}^N (T_{i,j,k}X)^2
\Biggr)
\\
&\qquad \qquad \times
\sup_{\substack{y,z \in [b,1-b],\\ |\vartheta-\vartheta^*|\le \delta_{m,N}^{\epsilon}}}
|g_2(y,z:\vartheta) - g_2(y,z:\vartheta^*)|
\\
&\pto 0
\end{align*}
as $\epsilon \to 0$, $m \to \infty$ and $N \to \infty$.
Since we can write
\begin{equation*}
\boldsymbol U(\vartheta) 
= \frac{1}{(1-2b)^2} \int_b^{1-b} \int_b^{1-b} 
\pd_{\vartheta} h_{r,\alpha}(y,z:\vartheta)^\TT 
\pd_{\vartheta} h_{r,\alpha}(y,z:\vartheta) \dd y \dd z,
\end{equation*}
in the same manner as the proof of (ii), we obtain
\begin{equation*}
\frac{1}{m}
\sum_{k=1}^{m_2}\sum_{j=1}^{m_1}
\pd_\vartheta h_{r,\alpha}(\overline y_j, \overline z_k:\vartheta^*)^\TT
\pd_\vartheta h_{r,\alpha}(\overline y_j, \overline z_k:\vartheta^*)
\to \boldsymbol U(\vartheta^*),
\end{equation*}
\begin{equation*}
\frac{1}{m}\sum_{k=1}^{m_2}\sum_{j=1}^{m_1}\Biggl(
\frac{1}{\epsilon^2 N \Delta^\alpha}\sum_{i=1}^N (T_{i,j,k}X)^2 
-h_{r,\alpha}(\overline y_j, \overline z_k:\vartheta^*)
\Biggr)\pd_\vartheta^2 h_{r,\alpha}(\overline y_j, \overline z_k:\vartheta^*)
\pto 0
\end{equation*}
as $\epsilon \to 0$, $m \to \infty$ and $N \to \infty$, 
which indicate that 
\begin{equation*}
\pd_{\vartheta}^2 \mathcal U_{m,N}^{\epsilon}(\vartheta^*) 
\pto 2 \boldsymbol U(\vartheta^*)
\end{equation*}
as $\epsilon \to 0$, $m \to \infty$ and $N \to \infty$, 
This concludes the proof of (iv).
\end{proof}

\begin{proof}[Proof of (v)]
Let $g_p = \pd_{\vartheta_p} h_{r,\alpha}(\cdot,\cdot:\vartheta^*)$ for $p=1,2,3$.
Since $g=(g_p)_{p=1}^3$ is given by 
\begin{equation*}
g(y,z) = \ee^{-\kappa^* y} \ee^{-\eta^* z} (-y \phi_{r,\alpha}(\theta_2^*), 
-z \phi_{r,\alpha}(\theta_2^*), \phi_{r,\alpha}'(\theta_2^*))
\end{equation*}
and $\phi_{r,\alpha}(\theta_2^*)\phi_{r,\alpha}'(\theta_2^*) \neq 0$,
we find that for any $v = (v_p)_{p=1}^3 \in \mathbb R^3 \setminus \{0\}$,
\begin{equation*}
v^\TT \boldsymbol U v 
=\sum_{p,q=1}^3 v_p v_q \langle g_p, g_q \rangle_b
= \Biggl\| \sum_{p=1}^3 v_p g_p \Biggr\|_b^2 > 0.
\end{equation*}
Indeed, if 
$\| \sum_{p=1}^3 v_p g_p \|_b =0$,
then by the identity
\begin{equation*}
\sum_{p=1}^3 v_p g_p(y,z) = 0,
\quad (y,z) \in [b,1-b]^2  
\end{equation*}
and the linear independence of the functions $(g_p)_{p=1}^3$, we have $v_1 = v_2 = v_3 = 0$.
\end{proof}

\subsubsection{Proofs of Propositions \ref{prop1} and \ref{prop2}}
Here, we provide the proofs of Propositions \ref{prop1} and \ref{prop2}.
See the next subsubsection for auxiliary results used in the proofs.

\begin{lem}\label{lem9}
For $\beta \in (0,3)$ and $\delta = r\sqrt{\Delta}$, it follows that

\begin{itemize}
\item[(1)]
\begin{align*}
\Phi_{j',k'}^{j,k}(\Delta;\beta)
&= \ee^{-\kappa(\widetilde y_{j-1}+\widetilde y_{j'})/2}
\ee^{-\eta(\widetilde z_{k-1}+\widetilde z_{k'})/2}
D_{1,\delta}^2 D_{2,\delta}^2 
F_{\beta, \Delta}(\widetilde y_{j'-1} -\widetilde y_{j},
\widetilde z_{k'-1} -\widetilde z_{k})
\nonumber
\\
&\qquad
+\OO(\Delta^{1+\beta \tand 2})
+\OO \Biggl(
\Delta^{1/2+\beta \tand 2}
\biggl( \frac{\ind_{\{j \neq j'\}}}{|j-j'|+1}+ \frac{\ind_{\{k \neq k'\}}}{|k-k'|+1}
\biggr) \Biggr),
\end{align*}

\item[(2)]
$\Phi_{j,k}^{j,k}(\Delta;\beta)
= 4\ee^{-\kappa \overline y_ j -\eta \overline z_k}
D_{1,\delta} D_{2,\delta} F_{\beta, \Delta} (0,0) + \OO(\Delta^{1+\beta \tand 2})$,

\item[(3)]
\begin{align*}
\Phi_{j',k'}^{j,k}(\Delta;\beta) &= 
\OO(\Delta^{1+\beta \tand 2}) + \OO \Biggl(
\Delta^{1/2+\beta \tand 2}
\biggl(
\frac{\ind_{\{j \neq j'\}}}{|j-j'|+1}+ \frac{\ind_{\{k \neq k'\}}}{|k-k'|+1}
\biggr) \Biggr)
\\
&\qquad+
\begin{cases}
\OO(\Delta^\beta),
& j = j' \text{ and } k = k',
\\
\OO \Bigl( \frac{\Delta^{\beta \tand 2}}{(|j-j'|+1)(|k-k'|+1)} \Bigr),
& otherwise,
\end{cases}
\end{align*}

\item[(4)]
\begin{align*}
\Phi_{J,j',k'}^{j,k}(\Delta;\beta) &= 
\OO \biggl(\frac{\Delta^{1+\beta \tand 2}}{J+1} \biggr) 
+ \OO \Biggl(
\frac{\Delta^{1/2+\beta \tand 2}}{J+1}
\biggl(
\frac{\ind_{\{j \neq j'\}}}{|j-j'|+1}+ \frac{\ind_{\{k \neq k'\}}}{|k-k'|+1}
\biggr) \Biggr)
\\
&\qquad+
\begin{cases}
\OO \Bigl(\frac{\Delta^\beta}{J+1} \Bigr),
& j = j' \text{ and } k = k',
\\
\OO \Bigl( \frac{\Delta^{\beta \tand 2}}{(J+1)(|j-j'|+1)(|k-k'|+1)} \Bigr),
& otherwise.
\end{cases}
\end{align*}
\end{itemize}
\end{lem}

\begin{proof}
(1) Let $g_1(x)=\ee^{-\kappa x/2}$ and $g_2(x)=\ee^{-\eta x/2}$.
Using (22) in Hildebrandt and Trabs \cite{Hildebrandt_Trabs2021},
we have
\begin{align}
&\ee^{\kappa y/2}(e_{ l_1}^{(1)}(y+\delta)-e_{ l_1}^{(1)}(y))
\ee^{\kappa y'/2}(e_{ l_1}^{(1)}(y'+\delta)-e_{ l_1}^{(1)}(y'))
\nonumber
\\
&=
D_\delta^2g_1(0)\cos(\pi l_1(y'-y))
-D_\delta^2 \bigl[g_1(\cdot)\cos(\pi  l_1 (y+y'+\cdot))\bigr](0)
\nonumber
\\
&\qquad
-g_1(\delta) D_\delta^2\bigl[\cos(\pi  l_1 (y'-y-\delta+\cdot))\bigr](0).
\label{eq-7-1}
\end{align}
For $y',z'\in [0,2)$, we define
\begin{align*}
G_{0,\beta,\Delta}(y,z:y',z') &= g_1(y)g_2(z)F_{\beta, \Delta}(y+y',z+z'),
\\
G_{1,\beta,\Delta}(y,z:y',z') &= g_1(y)F_{\beta, \Delta}(y+y',z+z'),
\\
G_{2,\beta,\Delta}(y,z:y',z') &= g_2(z)F_{\beta, \Delta}(y+y',z+z').
\end{align*}
Noting that 
\begin{align*}
D_{1,\delta} D_{2,\delta}^j F_{\beta, \Delta}(0,0) &= 
\sum_{ l_1, l_2\ge1}
\frac{1-\ee^{-\lambda_{l_1, l_2}\Delta}}{\lambda_{l_1, l_2}\mu_{l_1,l_2}^{\beta}}
D_\delta[\cos(\pi l_1 \cdot)](0) D_\delta^j [\cos(\pi  l_2 \cdot)](0),
\\
D_{1,\delta}^2 D_{2,\delta}^j G_{1,\beta,\Delta}(0,0:y',0) &= 
\sum_{ l_1, l_2\ge1}
\frac{1-\ee^{-\lambda_{l_1, l_2}\Delta}}{\lambda_{l_1, l_2}\mu_{l_1,l_2}^{\beta}}
D_\delta^2[g_1(\cdot)\cos(\pi  l_1 (y'+\cdot))](0)
\\
&\qquad \qquad \times 
D_\delta^j [\cos(\pi l_2 \cdot)](0),
\\
D_{1,\delta}^2 D_{2,\delta}^2 G_{0,\beta,\Delta}(0,0:y',z') &= 
\sum_{ l_1, l_2\ge1}
\frac{1-\ee^{-\lambda_{l_1, l_2}\Delta}}{\lambda_{l_1, l_2}\mu_{l_1,l_2}^{\beta}}
D_\delta^2[g_1(\cdot)\cos(\pi l_1 (y'+\cdot))](0)
\\
&\qquad \qquad \times 
D_\delta^2[g_2(\cdot)\cos(\pi l_2 (z'+\cdot))](0)
\end{align*}
for $j = 0,1,2$, 
we obtain from \eqref{eq-7-1} the representation
\begin{align*}
&\ee^{\kappa (\widetilde y_{j-1}+\widetilde y_{j'-1})/2}
\ee^{\eta (\widetilde z_{k-1}+\widetilde z_{k'-1})/2}
\Phi^{j,k}_{j',k'}(\Delta; \beta)
\nonumber
\\
&=
\sum_{ l_1,  l_2 \ge 1}
\frac{1-\ee^{-\lambda_{ l_1, l_2}\Delta}}{\lambda_{l_1, l_2}\mu_{l_1,l_2}^{\beta}}
\nonumber
\\
&\qquad\qquad\times
\ee^{\kappa \widetilde y_{j-1}/2}
(e_{ l_1}^{(1)}(\widetilde y_j)-e_{ l_1}^{(1)}(\widetilde y_{j-1}))
\ee^{\kappa \widetilde y_{j'-1}/2}
(e_{ l_1}^{(1)}(\widetilde y_{j'})-e_{ l_1}^{(1)}(\widetilde y_{j'-1}))
\nonumber
\\
&\qquad\qquad\times
\ee^{\eta \widetilde z_{k-1}/2}
(e_{ l_2}^{(2)}(\widetilde z_k)-e_{ l_2}^{(2)}(\widetilde z_{k-1}))
\ee^{\eta \widetilde z_{k'-1}/2}
(e_{ l_2}^{(2)}(\widetilde z_{k'})-e_{ l_2}^{(2)}(\widetilde z_{k'-1}))
\nonumber
\\
&= 
\sum_{ l_1,  l_2 \ge 1}
\frac{1-\ee^{-\lambda_{l_1, l_2}\Delta}}{\lambda_{l_1, l_2}\mu_{l_1,l_2}^{\beta}}
\nonumber
\\
&\qquad\qquad
\times
\bigl(
D_\delta^2g_1(0)\cos(\pi l_1(\widetilde y_{j'-1}-\widetilde y_{j-1}))
\nonumber
\\
&\qquad\qquad\qquad
-D_\delta^2 \bigl[g_1(\cdot)\cos(\pi  l_1 
(\widetilde y_{j-1}+\widetilde y_{j'-1}+\cdot))\bigr](0)
\nonumber
\\
&\qquad\qquad\qquad
-g_1(\delta) D_\delta^2\bigl[\cos(\pi  l_1 
(\widetilde y_{j'-1}-\widetilde y_{j}+\cdot))\bigr](0)
\bigr)
\nonumber
\\
&\qquad\qquad
\times
\bigl(
D_\delta^2 g_2(0)\cos(\pi l_2(\widetilde z_{k'-1}-\widetilde z_{k-1}))
\nonumber
\\
&\qquad\qquad\qquad
-D_\delta^2 \bigl[g_2(\cdot)\cos(\pi  l_2 
(\widetilde z_{k-1}+\widetilde z_{k'-1}+\cdot))\bigr](0)
\nonumber
\\
&\qquad\qquad\qquad
-g_2(\delta) D_\delta^2\bigl[\cos(\pi  l_2 
(\widetilde z_{k'-1}-\widetilde z_{k}+\cdot))\bigr](0)
\bigr)
\nonumber
\\
&=
D_\delta^2 g_1(0) D_\delta^2 g_2(0) 
F_{\beta, \Delta}(\widetilde y_{j'-1}-\widetilde y_{j-1},\widetilde z_{k'-1}-\widetilde z_{k-1})
\nonumber
\\
&\qquad-
D_\delta^2 g_1(0)
\Bigl(
D_{2,\delta}^2
G_2(0,0:\widetilde y_{j'-1}-\widetilde y_{j-1},\widetilde z_{k-1}+\widetilde z_{k'-1})
\nonumber
\\
&\qquad\qquad
+g_2(\delta)D_{2,\delta}^2 
F_{\beta, \Delta}(\widetilde y_{j'-1}-\widetilde y_{j-1},\widetilde z_{k'-1}-\widetilde z_{k})
\Bigr)
\nonumber
\\
&\qquad-
D_\delta^2 g_2(0)
\Bigl(
D_{1,\delta}^2 
G_1(0,0:\widetilde y_{j-1}+\widetilde y_{j'-1},\widetilde z_{k'-1}-\widetilde z_{k-1})
\nonumber
\\
&\qquad\qquad
+g_1(\delta)
D_{1,\delta}^2 
F_{\beta, \Delta}(\widetilde y_{j'-1}-\widetilde y_{j},\widetilde z_{k'-1}-\widetilde z_{k-1})
\Bigr)
\nonumber
\\
&\qquad+
D_{1,\delta}^2D_{2,\delta}^2 
G_0(0,0:\widetilde y_{j-1}+\widetilde y_{j'-1},\widetilde z_{k-1}+\widetilde z_{k'-1})
\nonumber
\\
&\qquad
+g_2(\delta) 
D_{1,\delta}^2 D_{2,\delta}^2 
G_1(0,0:\widetilde y_{j-1}+\widetilde y_{j'-1},\widetilde z_{k'-1}-\widetilde z_{k})
\nonumber
\\
&\qquad
+g_1(\delta) D_{1,\delta}^2 D_{2,\delta}^2 
G_2(0,0:\widetilde y_{j'-1}-\widetilde y_{j},\widetilde z_{k-1}+\widetilde z_{k'-1})
\nonumber
\\
&\qquad
+g_1(\delta)g_2(\delta) D_{1,\delta}^2 D_{2,\delta}^2 
F_{\beta, \Delta}(\widetilde y_{j'-1}-\widetilde y_{j},
\widetilde z_{k'-1}-\widetilde z_{k}).
\end{align*}
Notice that $D_\delta^ l g_1(0) = \OO(\delta^l) = \OO(\Delta^{ l/2})$,
$\widetilde y_{j}+\widetilde y_{j'}\in (b,2-b)$
and $\widetilde y_{j}-\widetilde y_{j'}=(j-j')\delta$. 
It holds from Lemma \ref{lem6} that
\begin{equation*}
F_{\beta, \Delta}(\widetilde y_{j'-1}-\widetilde y_{j-1},
\widetilde z_{k'-1}-\widetilde z_{k-1})
= \OO(\Delta^{\beta \tand 1}),
\end{equation*}
\begin{equation*}
D_{1,\delta}^2 
F_{\beta, \Delta}(\widetilde y_{j'-1}-\widetilde y_{j},
\widetilde z_{k'-1}-\widetilde z_{k-1})
=\OO(\Delta^{\beta \tand 2})
\end{equation*}
uniformly in $j,j',k,k'$, and it follows from 
\begin{align*}
&D_{1,\delta}^2 D_{2,\delta}^ l G_{1,\beta,\Delta}(0,0:y',z')
\\
&=D_\delta^2 g_1(0) D_{2,\delta}^ l F_{\beta, \Delta}(2\delta+y',z')
+2D_\delta g_1(0) D_{1,\delta} D_{2,\delta}^ l F_{\beta, \Delta}(\delta+y',z')
\\
&\qquad
+g_1(0) D_{1,\delta}^2 D_{2,\delta}^ l F_{\beta, \Delta}(y',z'),
\end{align*}
$D_{2,\delta}^2 F_{\beta, \Delta}(y,-\delta) 
= 2 D_{2,\delta} F_{\beta, \Delta}(y,0)$ and  
Lemmas \ref{lem6}-\ref{lem8} that
\begin{align*}
&D_{1,\delta}^2 
G_{1,\beta,\Delta}(0,0:\widetilde y_{j-1}+\widetilde y_{j'-1}, 
\widetilde z_{k'-1}-\widetilde z_{k-1})
\\
&= D_\delta^2 g_1(0) F_{\beta, \Delta}
(2\delta+\widetilde y_{j-1}+\widetilde y_{j'-1},\widetilde z_{k'-1}-\widetilde z_{k-1})
\\
&\qquad 
+2D_\delta g_1(0) D_{1,\delta} 
F_{\beta, \Delta}(\delta+\widetilde y_{j-1}+\widetilde y_{j'-1},
\widetilde z_{k'-1}-\widetilde z_{k-1})
\\
&\qquad
+g_1(0) D_{1,\delta}^2 F_{\beta, \Delta}
(\widetilde y_{j-1}+\widetilde y_{j'-1},\widetilde z_{k'-1}-\widetilde z_{k-1})
\\
&= \OO(\Delta^{1+\beta \tand 1}) +\OO(\Delta^{1/2 +\beta \tand (3/2)})
+\OO(\Delta^{\beta \tand 2})
\\
&=\OO(\Delta^{\beta \tand 2})
\end{align*}
uniformly in $j,j',k,k'$.
Therefore, we have
\begin{align*}
&D_{1,\delta}^2 D_{2,\delta}^2 
G_{1,\beta,\Delta}(0,0:\widetilde y_{j-1}+\widetilde y_{j'-1}, 
\widetilde z_{k'-1}-\widetilde z_{k})
\\
&=D_\delta^2 g_1(0) D_{2,\delta}^2 F_{\beta, \Delta}
(2\delta+\widetilde y_{j-1}+\widetilde y_{j'-1},\widetilde z_{k'-1}-\widetilde z_{k})
\\
&\qquad
+2D_\delta g_1(0) D_{1,\delta} D_{2,\delta}^2 
F_{\beta, \Delta}(\delta+\widetilde y_{j-1}+\widetilde y_{j'-1},
\widetilde z_{k'-1}-\widetilde z_{k})
\\
&\qquad
+g_1(0) D_{1,\delta}^2 D_{2,\delta}^2 
F_{\beta, \Delta}(\widetilde y_{j-1}+\widetilde y_{j'-1},\widetilde z_{k'-1}-\widetilde z_{k}),
\\
&=D_\delta^2 g_1(0) D_{2,\delta}^2 F_{\beta, \Delta}
(\widetilde y_{j}+\widetilde y_{j'},(k'-k-1) \delta)
\\
&\qquad
+2D_\delta g_1(0) D_{1,\delta} D_{2,\delta}^2 
F_{\beta, \Delta}(\widetilde y_{j}+\widetilde y_{j'-1}, (k'-k-1) \delta)
\\
&\qquad
+g_1(0) D_{1,\delta}^2 D_{2,\delta}^2 
F_{\beta, \Delta}(\widetilde y_{j-1}+\widetilde y_{j'-1},(k'-k-1) \delta),
\\
&=2 \cdot \ind_{\{ k = k' \}}\Bigl(
D_\delta^2 g_1(0) D_{2,\delta} F_{\beta, \Delta}
(\widetilde y_{j}+\widetilde y_{j'},0)
\\
&\qquad
+2D_\delta g_1(0) D_{1,\delta} D_{2,\delta} 
F_{\beta, \Delta}(\widetilde y_{j}+\widetilde y_{j'-1}, 0)
\\
&\qquad
+g_1(0) D_{1,\delta}^2 D_{2,\delta} 
F_{\beta, \Delta}(\widetilde y_{j-1}+\widetilde y_{j'-1}, 0) \Bigr)
\\
&\qquad +\ind_{\{ k \neq k' \}}\Bigl(
D_\delta^2 g_1(0) D_{2,\delta}^2 F_{\beta, \Delta}
(\widetilde y_{j}+\widetilde y_{j'},(|k'-k|+1) \delta)
\\
&\qquad \qquad
+2D_\delta g_1(0) D_{1,\delta} D_{2,\delta}^2 
F_{\beta, \Delta}(\widetilde y_{j}+\widetilde y_{j'-1}, (|k'-k|+1) \delta)
\\
&\qquad \qquad
+g_1(0) D_{1,\delta}^2 D_{2,\delta}^2 
F_{\beta, \Delta}(\widetilde y_{j-1}+\widetilde y_{j'-1},(|k'-k|+1) \delta) \Bigr)
\\
&=\ind_{\{ k = k'\}}
\Bigl(
\OO(\Delta^{1 +\beta \tand 2}) 
+\OO(\Delta^{1/2 +1 +\beta \tand (3/2)})
+\OO(\Delta^{1+\beta \tand 2})
\Bigr)
\\
&\qquad +\ind_{\{k \neq k'\}}
\Biggl(
\OO(\Delta^{1+\beta \tand 2}) 
+\OO \biggl( \Delta^{1/2} \biggl(\frac{\Delta^{1/2+\beta \tand (3/2)}}{|k-k'|+1} 
\lor \Delta^{1+\beta \tand (3/2)} \biggr) \biggr)
\\
&\qquad \qquad
+\OO \biggl(\frac{\Delta^{1/2+\beta \tand 2}}{|k-k'|+1} 
\lor \Delta^{1+\beta \tand 2} \biggr)
\Biggr)
\\
&=
\OO(\Delta^{1+\beta \tand 2}) +\OO(\Delta^{3/2+\beta \tand (3/2)})
+\OO \biggl( \ind_{\{k \neq k'\}} 
\frac{\Delta^{1+\beta \tand (3/2)} \lor \Delta^{1/2+\beta \tand 2}}{|k-k'|+1} \biggr) 
\\
&=
\OO(\Delta^{1+\beta \tand 2})
+\OO \biggl( \ind_{\{k \neq k'\}} 
\frac{\Delta^{1/2+\beta \tand 2}}{|k-k'|+1} \biggr) 
\end{align*}
uniformly in $j,j'$.
We also see from the expression
\begin{align*}
&D_{1,\delta}^2 D_{2,\delta}^2 G_{0,\beta,\Delta}(0,0:y',z')
\\
&=
D_\delta^2 g_1(0)
\Bigl(
D_\delta^2 g_2(0)F_{\beta, \Delta}(2\delta+y',2\delta+z')
\\
&\qquad\qquad
+2D_\delta g_2(0) D_{2,\delta} F_{\beta, \Delta}(2\delta+y',\delta+z')
+g_2(0) D_{2,\delta}^2 F_{\beta, \Delta}(2\delta+y',z')
\Bigr)
\\
&\qquad+
2D_\delta g_1(0)
\Bigl(
D_\delta^2 g_2(0) D_{1,\delta} F_{\beta, \Delta}(\delta+y',2\delta+z')
\\
&\qquad\qquad
+2D_\delta g_2(0) D_{1,\delta}D_{2,\delta} F_{\beta, \Delta}(\delta+y',\delta+z')
+g_2(0) 
D_{1,\delta} D_{2,\delta}^2 F_{\beta, \Delta}(\delta+y',z')
\Bigr)
\\
&\qquad+
g_1(0)
\Bigl(
D_\delta^2 g_2(0)
D_{1,\delta}^2 F_{\beta, \Delta}(y',2\delta+z')
\\
&\qquad\qquad
+2D_\delta g_2(0) D_{1,\delta}^2 D_{2,\delta} F_{\beta, \Delta}(y',\delta+z')
+g_2(0) D_{1,\delta}^2 D_{2,\delta}^2 F_{\beta, \Delta}(y',z')
\Bigr)
\end{align*}
that
\begin{align*}
&D_{1,\delta}^2D_{2,\delta}^2 
G_{0,\beta,\Delta}(0,0:\widetilde y_{j-1}+\widetilde y_{j'-1},
\widetilde z_{k-1}+\widetilde z_{k'-1})
\\
&= \OO \Bigl( \Delta( 
\Delta^{1+\beta \tand 1} +\Delta^{1/2+\beta \tand (3/2)}
+\Delta^{\beta \tand 2}) \Bigr)
\\
&\qquad+
\OO \Bigl( \Delta^{1/2}( 
\Delta^{1+\beta \tand (3/2)} +\Delta^{1/2+\beta \tand 2}
+\Delta^{1+\beta \tand (3/2)}) \Bigr)
\\
&\qquad+
\OO ( \Delta^{1+\beta \tand 2} +\Delta^{1/2+ 1 +\beta \tand (3/2)}
+\Delta^{1+\beta \tand 2} )
\\
&= \OO ( \Delta^{2+\beta \tand 1} +\Delta^{3/2+\beta \tand (3/2)}
+\Delta^{1+\beta \tand 2} )
\\
&\qquad+
\OO ( \Delta^{3/2+\beta \tand (3/2)} +\Delta^{1+\beta \tand 2}
+\Delta^{(3/2)+\beta \tand (3/2)})
\\
&\qquad+
\OO ( \Delta^{1+\beta \tand 2} +\Delta^{3/2+\beta \tand (3/2)}
+\Delta^{1+\beta \tand 2} )
\\
&= \OO ( \Delta^{2+\beta \tand 1} +\Delta^{3/2+\beta \tand (3/2)}
+\Delta^{1+\beta \tand 2} )
\\
&= \OO ( \Delta^{1+\beta \tand 2} )
\end{align*}
uniformly in $j,j',k,k'$.
Hence, we obtain 
\begin{align*}
&\ee^{\kappa (\widetilde y_{j-1}+\widetilde y_{j'-1})/2}
\ee^{\eta (\widetilde z_{k-1}+\widetilde z_{k'-1})/2}
\Phi^{j,k}_{j',k'}(\Delta;\beta)
\\
&=D_\delta^2 g_1(0) D_\delta^2 g_2(0) 
F_{\beta, \Delta}(\widetilde y_{j'-1}-\widetilde y_{j-1},
\widetilde z_{k'-1}-\widetilde z_{k-1})
\nonumber
\\
&\qquad-
D_\delta^2 g_1(0)
\Bigl(
D_{2,\delta}^2
G_{2,\beta,\Delta}(0,0:\widetilde y_{j'-1}-\widetilde y_{j-1},
\widetilde z_{k-1}+\widetilde z_{k'-1})
\nonumber
\\
&\qquad\qquad
+g_2(\delta)D_{2,\delta}^2 
F_{\beta, \Delta}(\widetilde y_{j'-1}-\widetilde y_{j-1},
\widetilde z_{k'-1}-\widetilde z_{k})
\Bigr)
\nonumber
\\
&\qquad-
D_\delta^2 g_2(0)
\Bigl(
D_{1,\delta}^2 
G_{1,\beta,\Delta}(0,0:\widetilde y_{j-1}+\widetilde y_{j'-1},
\widetilde z_{k'-1}-\widetilde z_{k-1})
\nonumber
\\
&\qquad\qquad
+g_1(\delta)
D_{1,\delta}^2 
F_{\beta, \Delta}(\widetilde y_{j'-1}-\widetilde y_{j},
\widetilde z_{k'-1}-\widetilde z_{k-1})
\Bigr)
\nonumber
\\
&\qquad+
D_{1,\delta}^2D_{2,\delta}^2 
G_{0,\beta,\Delta}(0,0:\widetilde y_{j-1}+\widetilde y_{j'-1},
\widetilde z_{k-1}+\widetilde z_{k'-1})
\nonumber
\\
&\qquad
+g_2(\delta) 
D_{1,\delta}^2 D_{2,\delta}^2 
G_{1,\beta,\Delta}(0,0:\widetilde y_{j-1}+\widetilde y_{j'-1},
\widetilde z_{k'-1}-\widetilde z_{k})
\nonumber
\\
&\qquad
+g_1(\delta) D_{1,\delta}^2 D_{2,\delta}^2 
G_{2,\beta,\Delta}(0,0:\widetilde y_{j'-1}-\widetilde y_{j},
\widetilde z_{k-1}+\widetilde z_{k'-1})
\nonumber
\\
&\qquad
+g_1(\delta)g_2(\delta) D_{1,\delta}^2 D_{2,\delta}^2 
F_{\beta, \Delta}(\widetilde y_{j'-1}-\widetilde y_{j},
\widetilde z_{k'-1}-\widetilde z_{k})
\\
&=
\OO(\Delta^{2+\beta \tand 1})
+ \OO \Bigl( \Delta (\Delta^{\beta \tand 2} +\Delta^{\beta \tand 2}) \Bigr)
\\
&\qquad 
+ \OO \Bigl( \Delta (\Delta^{\beta \tand 2} +\Delta^{\beta \tand 2}) \Bigr)
+ \OO( \Delta^{1 +\beta \tand 2} )
\\
&\qquad 
+ \OO( \Delta^{1 +\beta \tand 2} )
+ \OO \biggl( \ind_{\{k \neq k'\}} 
\frac{\Delta^{1/2 +\beta \tand 2}}{|k-k'|+1} \biggr) 
\\
&\qquad 
+ \OO(\Delta^{1 +\beta \tand 2})
+ \OO \biggl( \ind_{\{k \neq k'\}} 
\frac{\Delta^{1/2 +\beta \tand 2}}{|j-j'|+1} \biggr) 
\\
&\qquad 
+\ee^{-\kappa \delta/2}\ee^{-\eta \delta/2} D_{1,\delta}^2 D_{2,\delta}^2 
F_{\beta, \Delta}(\widetilde y_{j'-1}-\widetilde y_{j},
\widetilde z_{k'-1}-\widetilde z_{k})
\\
&=
\ee^{-\kappa \delta/2}\ee^{-\eta \delta/2} D_{1,\delta}^2 D_{2,\delta}^2 
F_{\beta, \Delta}(\widetilde y_{j'-1}-\widetilde y_{j},
\widetilde z_{k'-1}-\widetilde z_{k})
+\OO(\Delta^{1 +\beta \tand 2}) 
\\
&\qquad 
+\OO \Biggl(
\Delta^{1/2 +\beta \tand 2}
\biggl(
\ind_{\{j \neq j'\}} \frac{1}{|j-j'|+1}
+ \ind_{\{k \neq k'\}} \frac{1}{|k-k'|+1}
\biggr) \Biggr).
\end{align*}

(2) The equation $D_{1,\delta}^2 D_{2,\delta}^2 F_{\beta, \Delta}(-\delta,-\delta) 
= 4D_{1,\delta} D_{2,\delta} F_{\beta, \Delta}(0,0)$ yields the desired result.

(3) Utilizing Lemmas \ref{lem7} and \ref{lem8}, we obtain 
\begin{align*}
&D_{1,\delta}^2 D_{2,\delta}^2 
F_{\beta, \Delta}(\widetilde y_{j'-1}-\widetilde y_{j},
\widetilde z_{k'-1}-\widetilde z_{k})
\\
&=
\begin{cases}
4 D_{1,\delta} D_{2,\delta} F_{\beta, \Delta}(0,0), 
& j=j' \text{ and } k = k',
\\
D_{1,\delta}^2 D_{2,\delta}^2 
F_{\beta, \Delta}((j'-j-1)\delta, (k'-k-1)\delta),
& otherwise
\end{cases}
\\
&=
\begin{cases}
\OO(\Delta^\beta) + \OO(\Delta^{1 +\beta \tand 2}),
& j=j' \text{ and } k = k',
\\
\OO \Bigl(\frac{\Delta^{\beta \tand 2}}{(|j-j'|+1)(|k-k'|+1)} \Bigr)
+\OO( \Delta^{1 +\beta \tand 2} ),
& otherwise,
\end{cases}
\end{align*}
which gives the desired result.

(4) Since it follows that for any $f \in \boldsymbol F_{\beta-1, \beta}^{J}$, 
\begin{equation*}
(J+1)x^\beta f(x),\ (J+1)x^{1+\beta} f'(x),\ (J+1)x^{2+\beta} f''(x)
\lesssim (J+1)x \ee^{-J x} \lesssim 1
\quad(x \to 0),
\end{equation*}
\begin{equation*}
(J+1)x^{1+\beta} f(x),\ (J+1)x^{2+\beta} f'(x),\ (J+1)x^{3+\beta} f''(x)
\lesssim (J+1) \ee^{-J x} \lesssim 1
\quad(x \to \infty),
\end{equation*}
one has $(J+1) f \in \boldsymbol F_{\beta,\beta}$. 
Therefore, it holds from 
$f_{J,\beta} \in \boldsymbol F_{\beta-1,\beta}^{J/2}$ that
\begin{align*}
&(J+1) \Phi_{J,j',k'}^{j,k}(\Delta;\beta) 
\\
&=
\Delta^{1+\beta} \sum_{ l_{1}, l_{2}\ge1}
(J+1)\widetilde f_{J,\beta}(\lambda_{l_1, l_2}\Delta, \mu_{l_1, l_2}\Delta)
\\
&\qquad\times
(e_{ l_1}^{(1)}(\widetilde y_j)-e_{ l_1}^{(1)}(\widetilde y_{j-1}))
(e_{ l_2}^{(2)}(\widetilde z_k)-e_{ l_2}^{(2)}(\widetilde z_{k-1}))
\\
&\qquad\times
(e_{ l_1}^{(1)}(\widetilde y_{j'})-e_{ l_1}^{(1)}(\widetilde y_{j'-1}))
(e_{ l_2}^{(2)}(\widetilde z_{k'})-e_{ l_2}^{(2)}(\widetilde z_{k'-1}))
\\
&=
\OO \Biggl(
\Delta^{\beta \tand 2}
\biggl(\Delta + \frac{1}{(|j-j'|+1)(|k-k'|+1)} \biggr) \Biggr)
\\
&\qquad
+ \OO \Biggl(
\Delta^{1/2 +\beta \tand 2}
\biggl(
\ind_{\{j \neq j'\}} \frac{1}{|j-j'|+1}
+ \ind_{\{k \neq k'\}} \frac{1}{|k-k'|+1}
\biggr) \Biggr)
\end{align*}
in the same way as (2).
\end{proof}

\begin{proof}[\bf Proof of Proposition \ref{prop1}]
We set
\begin{align*}
A_{i, l_1, l_2} &= 
-\langle X_0, e_{ l_1, l_2} \rangle (1-\ee^{-\lambda_{ l_1, l_2}\Delta})
\ee^{-(i-1)\lambda_{ l_1, l_2}\Delta},
\\
B_{i, l_1, l_2}^{(1)} &= 
-\frac{1-\ee^{-\lambda_{ l_1, l_2}\Delta}}
{\mu_{ l_1, l_2}^{\alpha/2}}
\int_0^{(i-1)\Delta} \ee^{-\lambda_{ l_1, l_2}((i-1)\Delta-s)}
\dd w_{ l_1, l_2}(s),
\\
B_{i, l_1, l_2}^{(2)} &= 
\frac{1}{\mu_{ l_1, l_2}^{\alpha/2}}
\int_{(i-1)\Delta}^{i\Delta} \ee^{-\lambda_{ l_1, l_2}(i\Delta-s)}
\dd w_{ l_1, l_2}(s),
\end{align*}
and $B_{i, l_1, l_2}=B_{i, l_1, l_2}^{(1)}+B_{i, l_1, l_2}^{(2)}$.
We then obtain the decomposition
\begin{equation*}
x_{l_1, l_2}(t_{i})-x_{ l_1, l_2}(t_{i-1})=
A_{i, l_1, l_2} +\epsilon B_{i, l_1, l_2}.
\end{equation*}
Note that
\begin{equation*}
T_{i,j,k}X
= \sum_{ l_1, l_2\ge1}
(x_{ l_1, l_2}(t_{i})-x_{ l_1, l_2}(t_{i-1}))
(e_{ l_1}^{(1)}(\widetilde y_{j})-e_{ l_1}^{(1)}(\widetilde y_{j-1}))
(e_{ l_2}^{(2)}(\widetilde z_{k})-e_{ l_2}^{(2)}(\widetilde z_{k-1})).
\end{equation*}
From the independence of $\{B_{i, l_1, l_2}\}_{ l_1, l_2 \ge 1}$
and $\EE[B_{i, l_1, l_2}]=0$, we have
\begin{align*}
\EE[(T_{i,j,k}X)^2]
&= \sum_{ l_1, l_1', l_2, l_2'\ge1}
A_{i, l_1, l_2}A_{i, l_1', l_2'}
\\
&\qquad\qquad\times
(e_{ l_1}^{(1)}(\widetilde y_{j})-e_{ l_1}^{(1)}(\widetilde y_{j-1}))
(e_{ l_2}^{(2)}(\widetilde z_{k})-e_{ l_2}^{(2)}(\widetilde z_{k-1}))
\\
&\qquad\qquad\times
(e_{ l_1'}^{(1)}(\widetilde y_{j})-e_{ l_1'}^{(1)}(\widetilde y_{j-1}))
(e_{ l_2'}^{(2)}(\widetilde z_{k})-e_{ l_2'}^{(2)}(\widetilde z_{k-1}))
\\
&\qquad
+\epsilon^2 \sum_{ l_1, l_2\ge1}
\EE[B_{i, l_1, l_2}^2]
(e_{ l_1}^{(1)}(\widetilde y_{j})-e_{ l_1}^{(1)}(\widetilde y_{j-1}))^2
(e_{ l_2}^{(2)}(\widetilde z_{k})-e_{ l_2}^{(2)}(\widetilde z_{k-1}))^2
\\
&= 
\Biggl(
\sum_{ l_{1}, l_{2}\ge1}
A_{i, l_{1}, l_{2}}
(e_{ l_{1}}^{(1)}(\widetilde y_{j})-e_{ l_{1}}^{(1)}(\widetilde y_{j-1}))
(e_{ l_{2}}^{(2)}(\widetilde z_{k})-e_{ l_{2}}^{(2)}(\widetilde z_{k-1}))
\Biggr)^2
\\
&\qquad
+\epsilon^2 \sum_{ l_1, l_2\ge1}
\EE[B_{i, l_1, l_2}^2]
(e_{ l_1}^{(1)}(\widetilde y_{j})-e_{ l_1}^{(1)}(\widetilde y_{j-1}))^2
(e_{ l_2}^{(2)}(\widetilde z_{k})-e_{ l_2}^{(2)}(\widetilde z_{k-1}))^2
\\
&=: \mathbb A_{i,j,k} + \epsilon^2 \mathbb B_{i,j,k}.
\end{align*}
Using [A1]$_{\alpha_0}$, that is, 
$\|A_{\theta}^{(1+\alpha_0)/2} X_0 \|^2
=\sum_{ l_1,  l_2 \ge 1}\lambda_{ l_1,  l_2}^{1+\alpha_0}
\langle X_0, e_{ l_1, l_2} \rangle^2 < \infty$,
we find from the Schwarz inequality that
\begin{align*}
\mathbb A_{i,j,k}
&=\Biggl(
\sum_{ l_{1}, l_{2}\ge1}
\lambda_{ l_1, l_2}^{(1 +\alpha_0)/2} 
\langle X_0, e_{ l_1, l_2} \rangle
\\
&\qquad\qquad
\times
\frac{1-\ee^{-\lambda_{ l_1, l_2}\Delta}}
{\lambda_{ l_1, l_2}^{(1 +\alpha_0)/2}}
\ee^{-(i-1)\lambda_{ l_1, l_2}\Delta}
(e_{ l_{1}}^{(1)}(\widetilde y_{j})-e_{ l_{1}}^{(1)}(\widetilde y_{j-1}))
(e_{ l_{2}}^{(2)}(\widetilde z_{k})-e_{ l_{2}}^{(2)}(\widetilde z_{k-1}))
\Biggr)^2
\\
&\le
\sum_{ l_{1}, l_{2}\ge1}
\lambda_{ l_1, l_2}^{1+\alpha_0}
\langle X_0, e_{ l_1, l_2} \rangle^2 
\\
&\qquad\times
\sum_{ l_{1}, l_{2}\ge1}
\frac{(1-\ee^{-\lambda_{ l_1, l_2}\Delta})^2}
{\lambda_{ l_1, l_2}^{1+\alpha_0}}
\ee^{-2(i-1)\lambda_{ l_1, l_2}\Delta}
(e_{ l_{1}}^{(1)}(\widetilde y_{j})-e_{ l_{1}}^{(1)}(\widetilde y_{j-1}))^2
(e_{ l_{2}}^{(2)}(\widetilde z_{k})-e_{ l_{2}}^{(2)}(\widetilde z_{k-1}))^2
\\
&\lesssim
\|A_{\theta}^{(1+\alpha_0)/2} X_0 \|^2 
\Phi_{2(i-1),j,k}^{j,k}(\Delta;\alpha_0).
\end{align*}
Here, we can estimate
\begin{align*}
\sum_{i=1}^N \Phi_{2(i-1),j,k}^{j,k}(\Delta;\beta)
&\le
\sum_{l_{1}, l_{2} \ge 1}
\frac{1-\ee^{-\lambda_{l_1, l_2} \Delta}}{\lambda_{l_1, l_2}\mu_{l_1, l_2}^{\beta}}
(e_{ l_{1}}^{(1)}(\widetilde y_{j})-e_{ l_{1}}^{(1)}(\widetilde y_{j-1}))^2
(e_{ l_{2}}^{(2)}(\widetilde z_{k})-e_{ l_{2}}^{(2)}(\widetilde z_{k-1}))^2
\\
&= \Phi_{j,k}^{j,k}(\Delta;\beta)
\end{align*}
for $\beta \in (0,3)$.
Furthermore, we have
\begin{equation*}
\EE[B_{i, l_1, l_2}^2]
=\frac{1-\ee^{-\lambda_{l_1, l_2}\Delta}}
{\lambda_{l_1, l_2} \mu_{l_1,l_2}^\alpha}
\biggl(
1-\frac{1-\ee^{-\lambda_{ l_1, l_2}\Delta}}{2}
\ee^{-2(i-1)\lambda_{l_1, l_2}\Delta}
\biggr),
\end{equation*}
and
\begin{equation*}
\mathbb B_{i,j,k} = \Phi_{j,k}^{j,k}(\Delta;\alpha) 
-\frac{1}{2} \Phi_{2(i-1),j,k}^{j,k}(\Delta;\alpha).
\end{equation*}

From Lemma \ref{lem9}-(1), we already know that for any $\beta \in (0,3)$,
\begin{equation*}
\Phi_{j,k}^{j,k}(\Delta;\beta)
= 4\ee^{-\kappa \overline y_j -\eta \overline z_k}
D_{1,\delta} D_{2,\delta} F_{\beta, \Delta} (0,0) + \OO(\Delta^{1 +\beta \tand 2}).
\end{equation*}
By the definition of the operators $D_{1,\delta}$, $D_{2,\delta}$ with 
$\delta = r \sqrt{\Delta}$, we have 
\begin{align*}
D_{1,\delta}D_{2,\delta}F_{\beta, \Delta}(0,0) 
&= \Delta^{1+\beta}
\sum_{ l_1, l_2\ge1}
\widetilde f_\beta (\lambda_{l_1, l_2} \Delta, \mu_{l_1,l_2} \Delta)
(\cos(\pi r l_1 \sqrt{\Delta})-1)(\cos(\pi r l_2 \sqrt{\Delta})-1)
\\
&= \theta_2^\beta \Delta^{1+\beta}
\sum_{ l_1, l_2\ge1}
\widetilde f_\beta(\lambda_{l_1, l_2} \Delta, \theta_2 \mu_{l_1,l_2} \Delta)
g(l_1 \sqrt{\Delta},l_2 \sqrt{\Delta}),
\end{align*}
where $g(x,y) = (\cos(\pi r x)-1)(\cos(\pi r y)-1)$ and $g \in \boldsymbol G_{2}$.
We find from Lemmas \ref{lem1} and \ref{lem13} that
\begin{align*}
&\Delta \sum_{ l_1, l_2\ge1}
\widetilde f_\beta (\lambda_{l_1, l_2} \Delta, \theta_2 \mu_{l_1,l_2} \Delta)
g(l_1 \sqrt{\Delta},l_2 \sqrt{\Delta})
\\
&= \Delta \sum_{ l_1, l_2 \ge 1}
f_{\beta}(\theta_2 \pi^2 (l_1^2 +l_2^2) \Delta) 
g(l_1 \sqrt{\Delta},l_2 \sqrt{\Delta}) +
\begin{cases}
\OO(\Delta), & \beta < 2, 
\\
\OO(-\Delta \log \Delta), & \beta = 2, 
\\
\OO(\Delta^{3-\beta}), & \beta > 2
\end{cases}
\\
&= \iint_{\mathbb R_{+}^2} f_{\beta}(\theta_2 \pi^2 (x^2+y^2)) g(x,y) \dd x \dd y +
\begin{cases}
\OO(\Delta), & \beta < 2, 
\\
\OO(-\Delta \log \Delta), & \beta = 2, 
\\
\OO(\Delta^{3-\beta}), & \beta > 2
\end{cases}
\\
&= \iint_{\mathbb R_{+}^2} f_{\beta}(\theta_2 \pi^2 (x^2+y^2)) g(x,y) \dd x \dd y 
+\OO( \Delta^{1-\beta +\beta \tand 2}).
\end{align*}
This integral can be rewritten by (4.7) in Tonaki et al.\,\cite{TKU2024b} as
\begin{align*}
\iint_{\mathbb R_{+}^2} f_{\beta}(\theta_2 \pi^2 (x^2+y^2)) g(x,y) \dd x \dd y
&= \frac{1}{2 \theta_2 \pi}
\int_0^\infty 
\frac{1-\ee^{-x^2}}{x^{1+2\beta}}
\biggl(
J_0\Bigl(\frac{\sqrt{2}r x}{\sqrt{\theta_2}}\Bigr)
-2J_0\Bigl(\frac{r x}{\sqrt{\theta_2}}\Bigr)+1
\biggr) \dd x
\\
&= \frac{\phi_{r,\beta}(\theta_2)}{4 \theta_2^\beta}.
\end{align*}
Therefore, we have
\begin{equation*}
D_{1,\delta}D_{2,\delta}F_{\beta, \Delta}(0,0) 
= \theta_2^\beta \Delta^{\beta}
\biggl( \frac{\phi_{r,\beta}(\theta_2)}{4 \theta_2^\beta} 
+ \OO(\Delta^{1-\beta +\beta \tand 2})
\biggr)
= \Delta^\beta \frac{\phi_{r,\beta}(\theta_2)}{4} + \OO(\Delta^{1+\beta \tand 2})
\end{equation*}
and
\begin{equation*}
\Phi_{j,k}^{j,k}(\Delta;\beta)
= \Delta^\beta
\ee^{-\kappa \overline y_j -\eta \overline z_k} \phi_{r,\beta}(\theta_2)
+\OO(\Delta^{1+\beta \tand 2}).
\end{equation*}
Since $\Delta^{1+\beta \tand 2} = \oo(\Delta^{\beta})$ for $\beta \in (0,3)$,
we note that $\Phi_{j,k}^{j,k}(\Delta;\beta) = \OO(\Delta^{\beta})$
for $\beta \in (0,3)$ and thus,
\begin{equation*}
\sum_{i=1}^N \Phi_{2(i-1),j,k}^{j,k}(\Delta;\beta) = \OO(\Delta^\beta)
\end{equation*}
for $\beta \in (0,3)$ uniformly in $j,k$.

Consequently, we obtain
\begin{align*}
\EE[(T_{i,j,k}X)^2]
&= \mathbb A_{i,j,k} + \epsilon^2 \mathbb B_{i,j,k}
\\
&= R_{i,j,k}(\alpha_0) + \epsilon^2( \Phi_{j,k}^{j,k}(\Delta;\alpha) 
+R_{i,j,k}(\alpha))
\\
&= R_{i,j,k}(\alpha_0) 
+ \epsilon^2 \bigl( \Delta^\alpha
\ee^{-\kappa \overline y_j -\eta \overline z_k} \phi_{r,\alpha}(\theta_2)
+\OO(\Delta^{1+\alpha \tand 2}) +R_{i,j,k}(\alpha) \bigr),
\end{align*}
where for $\beta \in (0,3)$,   the reminder $R_{i,j,k}(\beta)$ satisfies
\begin{equation*}
\sum_{i=1}^N R_{i,j,k}(\beta) = \OO(\Delta^{\beta})
\end{equation*}
uniformly in $j,k$.
\end{proof}

\begin{proof}[\bf Proof of Proposition \ref{prop2}]
(1),(2) See the proof of Lemma 4.11 in \cite{TKU2024b}.

(3),(4) We find from (2) that 
\begin{equation*}
\frac{1}{\epsilon^4 N\Delta^{2 \alpha}} 
\sum_{i,i'=1}^N \cov[T_{i,j,k} X, T_{i',j,k} X]^2
= \OO \Bigg( \frac{1}{N} \sum_{i,i'=1}^N \frac{1}{(|i-i'|+1)^2} \Biggr)
= \OO(1)
\end{equation*}
uniformly in  $j,k$, and
\begin{align*}
&\frac{1}{\epsilon^4 m N (\Delta^{\alpha \tand 2})^2}
\sum_{k,k'=1}^{m_2}\sum_{j,j'=1}^{m_1}\sum_{i,i' =1}^N 
\cov[T_{i,j,k}X, T_{i',j',k'}X]^2
\\
&=
\OO \Biggl( 
\frac{1}{N}\sum_{i,i'=1}^N \frac{1}{(|i-i'|+1)^2}
\\
&\qquad \qquad \times
\biggl(
m\Delta^2
+\frac{1}{m_1}\sum_{j,j'=1}^{m_1} \frac{1}{(|j-j'|+1)^2}
\times \frac{1}{m_2}\sum_{k,k'=1}^{m_2} \frac{1}{(|k-k'|+1)^2}
\biggr)
\Biggr)
\\
&\qquad
+\OO \Biggl( \frac{1}{N}\sum_{i,i'=1}^N \frac{1}{(|i-i'|+1)^2}
\biggl(
m_2\Delta
\times\frac{1}{m_1}\sum_{j,j'=1}^{m_1} \frac{1}{(|j-j'|+1)^2}
\\
&\qquad\qquad
+m_1\Delta
\times \frac{1}{m_2}\sum_{k,k'=1}^{m_2} \frac{1}{(|k-k'|+1)^2}
\biggr)
\Biggr)
\\
&=\OO(1)+ \OO(\Delta^{1/2})
\\
&=\OO(1).
\end{align*}
(5),(6) Using Isserlis's theorem, we have
\begin{align*}
\Sigma_{l_1,\ldots,l_4}^{l_1',\ldots,l_4'} 
&= \cov\biggl[\prod_{p=1}^2 (x_{ l_p, l_p'}(t_i)-x_{ l_p, l_p'}(t_{i-1})),
\prod_{p=3}^4 (x_{ l_p, l_p'}(t_{i'})-x_{ l_p, l_p'}(t_{i'-1}))
\biggr]
\\
&= \epsilon^2 \Bigl(
A_{i, l_1, l_1'} A_{i', l_3, l_3'} 
\cov[B_{i, l_2, l_2'}, B_{i', l_4, l_4'}]
\\
&\qquad\qquad+
A_{i, l_1, l_1'} A_{i', l_4, l_4'} 
\cov[B_{i, l_2, l_2'}, B_{i', l_3, l_3'}]
\\
&\qquad\qquad+
A_{i, l_2, l_2'} A_{i', l_3, l_3'} 
\cov[B_{i, l_1, l_1'}, B_{i', l_4, l_4'}]
\\
&\qquad\qquad+
A_{i, l_2, l_2'} A_{i', l_4, l_4'} 
\cov[B_{i, l_1, l_1'}, B_{i', l_3, l_3'}] \Bigr)
\\
&\qquad+
\epsilon^4 \Bigl(
\cov[B_{i, l_1, l_1'}, B_{i', l_3, l_3'}]
\cov[B_{i, l_2, l_2'}, B_{i', l_4, l_4'}]
\\
&\qquad\qquad+
\cov[B_{i, l_1, l_1'}, B_{i', l_4, l_4'}]
\cov[B_{i, l_2, l_2'}, B_{i', l_3, l_3'}] \Bigr),
\end{align*}
and 
\begin{align*}
&\cov[(T_{i,j,k}X)^2, (T_{i',j',k'}X)^2 ] 
\\
&= 
\sum_{ l_1, l_1' \ge 1} \cdots \sum_{ l_4, l_4' \ge 1}
\Sigma_{l_1,\ldots,l_4}^{l_1',\ldots,l_4'}
\prod_{p=1}^2 
(e_{ l_p}^{(1)}(\widetilde y_j)-e_{ l_p}^{(1)}(\widetilde y_{j-1}))
(e_{ l_p'}^{(2)}(\widetilde z_k)-e_{ l_p'}^{(2)}(\widetilde z_{k-1}))
\\
&\qquad\qquad \times 
\prod_{p=3}^4 
(e_{ l_p}^{(1)}(\widetilde y_{j'})-e_{ l_p}^{(1)}(\widetilde y_{j'-1}))
(e_{ l_p'}^{(2)}(\widetilde z_{k'})-e_{ l_p'}^{(2)}(\widetilde z_{k'-1}))
\\
&= 
4\epsilon^2 \sum_{ l_1, l_1' \ge 1} A_{i, l_1, l_1'} 
(e_{ l_1}^{(1)}(\widetilde y_j)-e_{ l_1}^{(1)}(\widetilde y_{j-1}))
(e_{ l_1'}^{(2)}(\widetilde z_k)-e_{ l_1'}^{(2)}(\widetilde z_{k-1}))
\\
&\qquad\qquad\times \sum_{ l_2, l_2' \ge 1} A_{i', l_2, l_2'} 
(e_{ l_2}^{(1)}(\widetilde y_{j'})-e_{ l_2}^{(1)}(\widetilde y_{j'-1}))
(e_{ l_2'}^{(2)}(\widetilde z_{k'})-e_{ l_2'}^{(2)}(\widetilde z_{k'-1}))
\\
&\qquad\qquad\times \sum_{ l_3, l_3' \ge 1} 
\cov[B_{i, l_3, l_3'}, B_{i', l_3, l_3'}]
\\
&\qquad\qquad\qquad \times 
(e_{ l_3}^{(1)}(\widetilde y_j)-e_{ l_3}^{(1)}(\widetilde y_{j-1}))
(e_{ l_3'}^{(2)}(\widetilde z_k)-e_{ l_3'}^{(2)}(\widetilde z_{k-1}))
\\
&\qquad\qquad\qquad \times 
(e_{ l_3}^{(1)}(\widetilde y_{j'})-e_{ l_3}^{(1)}(\widetilde y_{j'-1}))
(e_{ l_3'}^{(2)}(\widetilde z_{k'})-e_{ l_3'}^{(2)}(\widetilde z_{k'-1}))
\\
&\qquad+2 \epsilon^4 \biggl( \sum_{ l_1, l_1' \ge 1} 
\cov[B_{i, l_1, l_1'}, B_{i', l_1, l_1'}]
\\
&\qquad\qquad\qquad \times 
(e_{ l_1}^{(1)}(\widetilde y_j)-e_{ l_1}^{(1)}(\widetilde y_{j-1}))
(e_{ l_1'}^{(2)}(\widetilde z_k)-e_{ l_1'}^{(2)}(\widetilde z_{k-1}))
\\
&\qquad\qquad\qquad \times 
(e_{ l_1}^{(1)}(\widetilde y_{j'})-e_{ l_1}^{(1)}(\widetilde y_{j'-1}))
(e_{ l_1'}^{(2)}(\widetilde z_{k'})-e_{ l_1'}^{(2)}(\widetilde z_{k'-1}))
\biggr)^2
\\
&= 
4\cov[T_{i,j,k} X, T_{i',j',k'} X] 
\\
&\qquad \times 
\sum_{ l_1, l_1' \ge 1} A_{i, l_1, l_1'} 
(e_{ l_1}^{(1)}(\widetilde y_j)-e_{ l_1}^{(1)}(\widetilde y_{j-1}))
(e_{ l_1'}^{(2)}(\widetilde z_k)-e_{ l_1'}^{(2)}(\widetilde z_{k-1}))
\\
&\qquad\times \sum_{ l_2, l_2' \ge 1} A_{i', l_2, l_2'} 
(e_{ l_2}^{(1)}(\widetilde y_{j'})-e_{ l_2}^{(1)}(\widetilde y_{j'-1}))
(e_{ l_2'}^{(2)}(\widetilde z_{k'})-e_{ l_2'}^{(2)}(\widetilde z_{k'-1}))
\\
&\qquad+2 \cov[T_{i,j,k} X, T_{i',j',k'} X]^2.
\end{align*}
We see from the proof of Proposition \ref{prop1} that 
for $\alpha_0 \in (0,3)$, 
\begin{equation*}
{\mathbf A}_{N,j,k} = 
\sum_{i=1}^N 
\biggl(
\sum_{ l_1, l_1' \ge 1} A_{i, l_1, l_1'} 
(e_{ l_1}^{(1)}(\widetilde y_j)-e_{ l_1}^{(1)}(\widetilde y_{j-1}))
(e_{ l_1'}^{(2)}(\widetilde z_k)-e_{ l_1'}^{(2)}(\widetilde z_{k-1}))
\biggr)^2
=\OO(\Delta^{\alpha_0})
\end{equation*}
uniformly in $j,k$. 
Therefore, we find from the Schwarz inequality, (3) and (4) that
\begin{align*}
&\biggl| \sum_{i,i'=1}^N \cov[(T_{i,j,k}X)^2, (T_{i',j,k}X)^2 ] \biggr|
\\
&\le 4 \biggl( \sum_{i,i'=1}^N \cov[T_{i,j,k} X, T_{i',j,k} X]^2 \biggr)^{1/2}
\mathbf A_{N,j,k}
+2\sum_{i,i'=1}^N \cov[T_{i,j,k} X, T_{i',j,k} X]^2
\\
&=\OO(\epsilon^2 \sqrt{N} \Delta^{\alpha} \cdot \Delta^{\alpha_0})
+\OO(\epsilon^4 N \Delta^{2\alpha})
\\
&=\OO (\epsilon^4 N \Delta^{2\alpha}
( \epsilon^{-2} \Delta^{\alpha_0 -\alpha +1/2} \lor 1 ))
\end{align*}
uniformly in $j,k$, and
\begin{align*}
&\biggl| \sum_{k,k'=1}^{m_2} \sum_{j,j'=1}^{m_1} \sum_{i,i'=1}^N 
\cov[(T_{i,j,k}X)^2, (T_{i',j',k'}X)^2 ] \biggr|
\\
&\le 4 \biggl( \sum_{k,k'=1}^{m_2} \sum_{j,j'=1}^{m_1} \sum_{i,i'=1}^N 
\cov[T_{i,j,k} X, T_{i',j',k'} X]^2 \biggr)^{1/2}
\sum_{k=1}^{m_2} \sum_{j=1}^{m_1} \mathbf A_{N,j,k}
\\
&\qquad+2\sum_{k,k'=1}^{m_2} \sum_{j,j'=1}^{m_1} \sum_{i,i'=1}^N 
\cov[T_{i,j,k} X, T_{i',j',k'} X]^2
\\
&= \OO(\epsilon^2 \sqrt{m N} \Delta^{\alpha \tand 2} \cdot m \Delta^{\alpha_0}) 
+\OO(\epsilon^4 m N (\Delta^{\alpha \tand 2})^2)
\\
&=\OO(\epsilon^4 m N (\Delta^{\alpha \tand 2})^2 
(\epsilon^{-2} \Delta^{\alpha_0 -\alpha \tand 2} \lor 1 )).
\end{align*}
\end{proof}

\subsubsection{Auxiliary results for Riemann summation and Fourier series}
The proof is similar to that in \cite{TKU2024b}, 
but in the case of SPDEs with a small noise, 
it is necessary to consider [A1]$_{\alpha_0}$ with $\alpha_0 > 2$
since the setting of the initial value affects the convergence rate of the estimator. 
In this subsubsection, we therefore give more general results 
than those in \cite{TKU2024b}.

We note that the following properties.
\begin{itemize}
\item[(1)]
$\displaystyle \sum_{1 \le l \le L} \frac{1}{l^\beta} =
\begin{cases}
\OO(L^{1-\beta}), & \beta < 1,
\\
\OO(\log L), & \beta = 1,
\\
\OO(1), & \beta >1,
\end{cases}
$

\item[(2)]
$\displaystyle \sum_{l > L}\frac{1}{l^\beta}
=\OO\Bigl(\frac{1}{L^{\beta-1}}\Bigr), \quad \beta>1,
$

\item[(3)]
$\displaystyle \sum_{1 \le l_1^2 + l_2^2 \le L^2} 
\frac{1}{(l_1^2 +l_2^2)^\beta} =
\begin{cases}
\OO(L^{2(1-\beta)}), & \beta < 1,
\\
\OO(\log L), & \beta = 1,
\\
\OO(1), & \beta >1,
\end{cases}
$

\item[(4)]
$\displaystyle \sum_{l_1^2 +l_2^2 > L^2} 
\frac{1}{(l_1^2 +l_2^2)^\beta}
=\OO \Bigl(\frac{1}{L^{2(\beta-1)}}\Bigr), \quad \beta>1.
$
\end{itemize}

\begin{lem}\label{lem1}
Let $\beta, h >0$ and $f \in \boldsymbol F_{\beta,\beta}$.
\begin{itemize}
\item[(1)]
For $\gamma \ge 0$, $g \in \boldsymbol G_\gamma$ and $\varphi \in \boldsymbol B_b$, 
it holds that as $h \to 0$, 
\begin{align}
&\sum_{l_1, l_2 \ge 1} 
f(\lambda_{ l_1, l_2}h^2) g(l_1 h, l_2 h) \varphi(l_1 y, l_2 z)
\nonumber
\\
&=
\sum_{ l_1, l_2\ge1} 
f(\theta_2\pi^2( l_1^2+ l_2^2)h^2) 
g( l_1 h, l_2 h) \varphi(l_1 y,l_2 z)
+
\begin{cases}
\OO(1), & \beta < \gamma,
\\
\OO(-\log h), & \beta = \gamma,
\\
\OO(h^{2(\gamma-\beta)}), & \beta > \gamma
\end{cases}
\label{eq-lem1-1}
\end{align}
uniformly in $y,z\in\mathbb R$. In particular, as $h \to 0$,
\begin{equation}
\sum_{ l_1, l_2\ge1} 
f(\lambda_{ l_1, l_2}h^2) 
g( l_1 h, l_2 h) \varphi(l_1 y, l_2 z)
=
\begin{cases}
\OO(h^{-2}), & \beta < \gamma+1,
\\
\OO(-h^{-2}\log h), & \beta = \gamma+1,
\\
\OO(h^{2(\gamma-\beta)}), & \beta > \gamma+1
\end{cases}
\label{eq-lem1-2}
\end{equation}
uniformly in $y,z\in\mathbb R$.

\item[(2)]
Let $\gamma, \gamma_1, \gamma_2 \ge 0$ with $\gamma = \gamma_1 +\gamma_2$ and 
$g \in \widetilde {\boldsymbol G}_{\gamma_1,\gamma_2} 
\cap \boldsymbol G_\gamma^2$. 
If $\beta < \gamma +1$, then the function 
$(x,y) \mapsto f(x^2+y^2)g(x,y)$ is integrable on $\mathbb R_{+}^2$ and 
it follows that 
for $\underline{\gamma} = \gamma_1 \land \gamma_2$
and $\overline{\gamma} = \gamma_1 \lor \gamma_2$, 
\begin{align}
&h^2 \sum_{ l_1, l_2\ge1}
f(( l_1^2+ l_2^2)h^2) g( l_1 h,  l_2 h)
\nonumber
\\
&=\iint_{\mathbb R_{+}^2} f(x^2+y^2) g(x,y) \dd x \dd y +
\begin{cases}
\OO(h^{2 \land (2 \underline{\gamma} +1)}), 
& \beta < \gamma \lor (\overline{\gamma} +1/2),
\\
\OO(-h^{2 \land (2 \underline{\gamma} +1)} \log h), 
& \beta = \gamma \lor (\overline{\gamma} +1/2),
\\
\OO(h^{2(\gamma -\beta +1)}), 
& \beta > \gamma \lor (\overline{\gamma} +1/2)
\end{cases}
\label{eq-lem1-3}
\end{align}
as $h \to 0$.
\end{itemize}
\end{lem}

\begin{proof}
(1) See Lemma 4.1-(1) in \cite{TKU2024b}.
(2) Define $E_{l _1,l _2}
=((l _1-1/2)h, (l _1+1/2)h]
\times ((l _2-1/2)h,(l _2+1/2)h]\subset \mathbb R_{+}^2$
and $E=\bigcup_{l _1,l _2\ge1}E_{l _1,l _2}$.
Since $f \in \boldsymbol F_{\beta,\beta}$ and $g \in \boldsymbol G_\gamma^2$,
the function $G(x,y) = f(x^2+y^2) g(x,y)$ can be controlled as follows. 
\begin{equation}\label{eq-4-2}
|\pd^j G(x,y)| \lesssim 
\begin{cases}
(x^2+y^2)^{\gamma-\beta-j/2},
\quad (x, y \to 0),
\\
\frac{1}{(x^2+y^2)^{1+\beta+j/2}},
\quad (x, y \to \infty)
\end{cases}
\end{equation}
for $j=0,1,2$. 
Note that $G \in L^1(\mathbb R_{+}^2)$ if $\beta < \gamma +1$.
Using the Taylor expansion, we obtain
\begin{align}
&\Biggl|
h^2 \sum_{l _1,l _2\ge1}
G(l _1 h, l _2 h) -\iint_{E} G(x,y) \dd x \dd y
\Biggr|
\nonumber
\\
&=
\Biggl|
\sum_{l _1,l _2\ge1}
\iint_{E_{l _1,l _2}}
(G(l _1 h,l _2h)-G(x,y)) \dd x \dd y
\Biggr|
\nonumber
\\
&\le
\Biggl|
\sum_{l _1,l _2\ge1}
\iint_{E_{l _1,l _2}}
\pd G(l _1 h,l _2h)
\begin{pmatrix}
x-l _1 h
\\
y-l _2 h
\end{pmatrix}
\dd x \dd y
\Biggr|
\nonumber
\\
&\qquad+
\sum_{l _1,l _2\ge1}
\iint_{E_{l _1,l _2}} 
\biggl|
\int_0^1
(1-u)\pd^2 G(l _1 h +u(x-l _1 h), 
l _2 h +u(y-l _2 h))
\dd u  
\biggr|
\nonumber
\\
&\qquad\qquad \times 
\left|
\begin{pmatrix}
x-l _1 h
\\
y-l _2 h
\end{pmatrix}
\right|^2
\dd x \dd y 
\nonumber
\\
&\le h^2
\sum_{l _1,l _2\ge1}
\iint_{E_{l _1,l _2}} 
\sup_{v_1,v_2 \in [-1/2,1/2]} |\pd^2 G((l _1 +v_1)h, (l _2+v_2) h)|
\dd x \dd y 
\nonumber
\\
&= \OO \biggl(
h^2 \iint_E |\partial^2 G(x,y)|\dd x \dd y
\biggr).
\label{eq-4-3}
\end{align}
Furthermore, we see that for $j = 0,1,2$, 
\begin{align*}
\iint_E |\partial^j G(x,y)|\dd x \dd y
&= \OO \biggl( 
\iint_{h^2 \le x^2 +y^2 < \infty} |\partial^j G(x,y)|\dd x \dd y
\biggr)
\\
&= \OO \biggl( 
\iint_{h^2 \le x^2 +y^2 \le 1} |\partial^j G(x,y)|\dd x \dd y
\biggr) +\OO(1)
\\
&= \OO \biggl( 
\iint_{h^2 \le x^2 +y^2 \le 1} (x^2 +y^2)^{\gamma -\beta -j/2} \dd x \dd y
\biggr) +\OO(1)
\\
&= \OO \biggl( 
\int_{h}^1 r^{2\gamma -2\beta -j +1} \dd r
\biggr) +\OO(1)
\\
&= 
\begin{cases}
\OO (1), & -2\gamma +2\beta +j -1 < 1,
\\
\OO (-\log h), & -2\gamma +2\beta +j -1 = 1,
\\
\OO(h^{2+2\gamma -2\beta -j}), & -2\gamma +2\beta +j -1 > 1
\end{cases}
\\
&= 
\begin{cases}
\OO (1), & \beta < \gamma +1 -j/2,
\\
\OO (-\log h), & \beta = \gamma +1 -j/2,
\\
\OO(h^{2+2\gamma -2\beta -j}), & \beta > \gamma +1 -j/2,
\end{cases}
\end{align*}
and
\begin{equation*}
h^2 \iint_E |\pd^j G(x,y)|\dd x \dd y = 
\begin{cases}
\OO (h^2), & \beta < \gamma +1 -j/2,
\\
\OO (-h^2 \log h), & \beta = \gamma +1 -j/2,
\\
\OO(h^{4+2\gamma -2\beta -j}), & \beta > \gamma +1 -j/2.
\end{cases}
\end{equation*}
By the definition of $\widetilde {\boldsymbol G}_{\gamma_1,\gamma_2}$,
there exist $g_1,g_2 \in C(\mathbb R)$
such that $g(x,y) = g_1(x)g_2(y)$, and it follows that
\begin{align*}
\int_0^{h/2}\int_0^{h/2}
G(x,y)\dd x \dd y
&=\OO \Biggl(
\int_0^{h/2}\int_0^{h/2}
(x^2+y^2)^{\gamma-\beta} \dd x \dd y
\Biggr)
\\
&=\OO \Biggl(
\int_0^{h} r^{2\gamma-2\beta+1} \dd r
\Biggr)
=\OO (h^{2(\gamma-\beta+1)}),
\\
\int_{h/2}^\infty \int_0^{h/2}
G(x,y)\dd x \dd y
&\lesssim 
\int_0^{h/2} x^{2\gamma_1} \dd x
\int_{h/2}^\infty |f(\theta_2 \pi^2 y^2)| |g_2(y)| \dd y
\\
&=\OO \Biggl(
h^{2 \gamma_1 +1}
\biggl(
\int_{h}^1 
y^{2 \gamma_2-2\beta}\dd y
+\int_1^\infty 
\frac{\dd y}{y^{2+2\beta}}
\biggr)
\Biggr)
\\
&=
\begin{cases}
\OO(h^{2\gamma_1 +1}), & \beta < \gamma_2 +1/2,
\\
\OO(-h^{2 \gamma_1 +1} \log h), 
& \beta = \gamma_2 +1/2,
\\
\OO(h^{2(\gamma -\beta +1)}), & \beta > \gamma_2 +1/2,
\end{cases}
\\
\int_0^{h/2}\int_{h/2}^\infty 
G(x,y)\dd x \dd y
&\lesssim 
\begin{cases}
\OO(h^{2\gamma_2 +1}), & \beta < \gamma_1 +1/2,
\\
\OO(-h^{2 \gamma_2 +1} \log h), 
& \beta = \gamma_1 +1/2,
\\
\OO(h^{2(\gamma -\beta +1)}), & \beta > \gamma_1 +1/2,
\end{cases}
\end{align*}
and 
\begin{equation}\label{eq-4-4}
\iint_{\mathbb R_{+}^2\setminus E} G(x,y) \dd x \dd y = 
\begin{cases}
\OO(h^{2 \underline{\gamma} +1}), & \beta < \overline{\gamma} +1/2,
\\
\OO(-h^{2 \underline{\gamma} +1} \log h), 
& \beta = \overline{\gamma} +1/2,
\\
\OO(h^{2(\gamma -\beta +1)}), & \beta > \overline{\gamma} +1/2.
\end{cases}
\end{equation}
Therefore, we obtain
\begin{align*}
h^2 \sum_{l _1,l _2\ge1} G(l _1 h, l _2 h) 
&= \iint_{E} G(x,y) \dd x \dd y
+\OO \biggl(h^2 \iint_E |\partial^2 G(x,y)|\dd x \dd y \biggr)
\\
&= \iint_{\mathbb R_{+}^2} G(x,y) \dd x \dd y
-\iint_{\mathbb R_{+}^2\setminus E} G(x,y) \dd x \dd y
\\
&\qquad
+\OO \biggl(h^2 \iint_E |\partial^2 G(x,y)|\dd x \dd y \biggr)
\\
&= \iint_{\mathbb R_{+}^2} G(x,y) \dd x \dd y +
\begin{cases}
\OO(h^{2 \underline{\gamma} +1}), & \beta < \overline{\gamma} +1/2,
\\
\OO(-h^{2 \underline{\gamma} +1} \log h), 
& \beta = \overline{\gamma} +1/2,
\\
\OO(h^{2(\gamma -\beta +1)}), & \beta > \overline{\gamma} +1/2
\end{cases}
\\
&\qquad+
\begin{cases}
\OO (h^2), & \beta < \gamma,
\\
\OO (-h^2 \log h), & \beta = \gamma,
\\
\OO(h^{2(\gamma -\beta +1)}), & \beta > \gamma
\end{cases}
\\
&= \iint_{\mathbb R_{+}^2} G(x,y) \dd x \dd y 
\\
&\qquad+
\begin{cases}
\OO(h^{2 \land (2 \underline{\gamma} +1)}), 
& \beta < \gamma \lor (\overline{\gamma} +1/2),
\\
\OO(-h^{2 \land (2 \underline{\gamma} +1)} \log h), 
& \beta = \gamma \lor (\overline{\gamma} +1/2),
\\
\OO(h^{2(\gamma -\beta +1)}), 
& \beta > \gamma \lor (\overline{\gamma} +1/2).
\end{cases}
\end{align*}
\end{proof}

\begin{lem}[Lemma 4.2 in \cite{TKU2024b}]\label{lem13}
Let $\beta >0$ and $f \in C^2(\mathbb R_{+})$ such that
$f(x) = f_1(x) f_2(x)$, 
\begin{equation*}
f_1(x) \lesssim 1 \ (x \to 0), 
\quad
x f_1(x) \lesssim 1 \ (x \to \infty), 
\quad
f_2(x) = x^{-\beta}.
\end{equation*}
For $\gamma \ge 0$, $g \in \boldsymbol G_\gamma$ 
and $\varphi \in \boldsymbol B_b$, 
it holds that as $h \to 0$, 
\begin{align*}
&\sum_{ l_1, l_2\ge1} 
f_1(\lambda_{ l_1, l_2}h^2) f_2(\theta_2 \mu_{ l_1, l_2}h^2) 
g( l_1 h, l_2 h) \varphi(l_1 y, l_2 z)
\nonumber
\\
&=
\sum_{l_1,l_2\ge1} 
f(\lambda_{l_1, l_2}h^2)
g(l_1 h,l_2 h) \varphi(l_1 y, l_2 z)
+
\begin{cases}
\OO(1), & \beta < \gamma,
\\
\OO(-\log h), & \beta = \gamma,
\\
\OO(h^{2(\gamma-\beta)}), & \beta > \gamma
\end{cases}
\end{align*}
uniformly in $y,z\in\mathbb R$. 
\end{lem}

\begin{lem}\label{lem5}
Let $\beta >0$, $\gamma, \gamma_1, \gamma_2 \ge 0$ with $\gamma = \gamma_1 +\gamma_2$ 
and $\tau_1, \tau_2 \in \{\sin,\cos\}$. 
For $f \in \boldsymbol F_{\beta,\beta}$ and 
$g \in \widetilde {\boldsymbol G}_{\gamma_1,\gamma_2} \cap \boldsymbol G_\gamma^2$, 
define $G(x,y)=f(x^2+y^2)g(x,y)$.

\begin{itemize}
\item[(1)]
For $y\in(0,2)$, as $h\to0$,
\begin{equation}\label{eq-lem5-1}
\sum_{l _1,l _2 \ge 1}
G(l _1h,l _2h) \tau_1(\pi l _1 y) =
\begin{cases}
\OO \Bigl( \frac{h^{0 \land (2\gamma_1-1)}}{(y \land (2-y))^2} \Bigr), 
& \beta < \gamma \lor (\gamma_2 +1/2),
\\
\OO \Bigl( \frac{-h^{0 \land (2\gamma_1-1)} \log h}{(y \land (2-y))^2} \Bigr), 
& \beta = \gamma \lor (\gamma_2 +1/2),
\\
\OO \Bigl( \frac{h^{2(\gamma-\beta)}}{(y \land (2-y))^2} \Bigr), 
& \beta > \gamma \lor (\gamma_2 +1/2).
\end{cases}
\end{equation}

\item[(2)]
For $y,z \in(0,2)$, as $h \to 0$,
\begin{equation}\label{eq-lem5-2}
\sum_{l _1,l _2 \ge 1} 
G(l _1h,l _2h) \tau_1(\pi l _1 y) \tau_2(\pi l _2 z) =
\begin{cases}
\OO \Bigl( \frac{h^{0 \land (2 \underline{\gamma}-1)}}
{(y \land (2-y))(z \land (2-z))} \Bigr), 
& \beta < \gamma \lor (\overline{\gamma} +1/2),
\\
\OO \Bigl( \frac{-h^{0 \land (2\underline{\gamma} -1)} \log h}
{(y \land (2-y))(z \land (2-z))} \Bigr), 
& \beta = \gamma \lor (\overline{\gamma} +1/2),
\\
\OO \Bigl( \frac{h^{2(\gamma-\beta)}}
{(y \land (2-y))(z \land (2-z))} \Bigr), 
& \beta > \gamma \lor (\overline{\gamma} +1/2).
\end{cases}
\end{equation}
\end{itemize}
\end{lem}

\begin{proof}
(1) It holds from
$f \in \boldsymbol F_{\beta,\beta}$ and 
$g \in \widetilde {\boldsymbol G}_{\gamma_1,\gamma_2}$ that 
\begin{align}
\sum_{l _2\ge1}|G(h,l _2 h)|
&\lesssim
h^{2 \gamma_1}
\sum_{l _2\ge1}|f((l _2^2+1)h^2)| ((l _2 h)^2 \land 1)^{\gamma_2}
\nonumber
\\
&\lesssim
h^{2(\gamma -\beta)}
\sum_{1\le l _2 \le 1/h}
\frac{1}{l _2^{2(\beta -\gamma_2)}}
+h^{2(\gamma_2 -\beta-1)}
\sum_{l _2 > 1/h}
\frac{1}{l _2^{2\beta+2}}
\nonumber
\\
& =
\begin{cases}
\OO(h^{2\gamma_1-1}), & \beta < \gamma_2 +1/2,
\\
\OO(-h^{2\gamma_1-1} \log h), & \beta = \gamma_2 +1/2,
\\
\OO(h^{2(\gamma-\beta)}), & \beta > \gamma_2 +1/2.
\end{cases}
\label{eq-5-1}
\end{align}
Since we can estimate
\begin{equation*}
|D_{1,h}^k D_{2,h}^m G(h x,h y)| 
\lesssim 
\min_{j=0,\ldots,(k+m) \land 2} h^{j} 
\sup_{u,v\in[0,1]} |\pd^{j} G(h(x+u), h(y+v))|
\end{equation*}
for $G \in C^2(\mathbb R_{+}^2)$, we obtain
\begin{align*}
K(h) &= \Biggl|
\sum_{l _2 \ge 1} 
D_{1,h} G(h,l _2 h)
\Biggr|
+\sum_{l _1 \ge 1} \Biggl|
\sum_{l _2 \ge 1}
D_{1,h}^2 G(l _1 h,l _2 h)
\Biggr|
\\
&\lesssim
\sum_{l _2 \ge 1} 
\sup_{u\in[0,1]} |G(h(1+u),l _2 h)|
\\
&\qquad 
+ h^2 \sum_{l _1,l _2 \ge 1}
\sup_{u,v\in[0,1]}
|\pd_x^2 G((l _1+u) h,(l _2+v) h)|
\\
&\lesssim
\sum_{l _2 \ge 1} |G(h,l _2 h)| + \iint_E |\pd_x^2 G(x,y)| \dd x \dd y
\\
&=
\begin{cases}
\OO(h^{2\gamma_1-1}), & \beta < \gamma_2 +1/2,
\\
\OO(-h^{2(\gamma-\beta)} \log h), & \beta = \gamma_2 +1/2,
\\
\OO(h^{2(\gamma-\beta)}), & \beta > \gamma_2 +1/2
\end{cases}
\\
&\qquad +
\begin{cases}
\OO (1), & \beta < \gamma,
\\
\OO (-\log h), & \beta = \gamma,
\\
\OO(h^{2(\gamma -\beta)}), & \beta > \gamma
\end{cases}
\\
&=
\begin{cases}
\OO(h^{0 \land (2\gamma_1-1)}), & \beta < \gamma \lor (\gamma_2 +1/2),
\\
\OO(-h^{0 \land (2\gamma_1-1)} \log h), & \beta = \gamma \lor (\gamma_2 +1/2),
\\
\OO(h^{2(\gamma-\beta)}), & \beta > \gamma \lor (\gamma_2 +1/2),
\end{cases}
\end{align*}
which together with Lemma 4.4 in \cite{TKU2024b} yields the desired result. 

(2) Note that
\begin{align*}
K(h) 
&= |G(h, h)|
+\sum_{l _1 \ge 1} |G(l _1h, h)|
+\sum_{l _2 \ge 1} |G(h, l _2 h)|
\\
&\qquad
+\sum_{l _1,l _2 \ge 2}
|D_{1,h} D_{2,h} G((l _1-1)h,(l _2-1)h)|
\\
&\lesssim
|G(h, h)|
+\sum_{l _1 \ge 1} |G(l _1h, h)|
+\sum_{l _2 \ge 1} |G(h, l _2 h)|
\\
&\qquad
+h^2 \sum_{l _1,l _2 \ge 1}
|\pd^2 G(l _1 h, l _2 h)|.
\end{align*}
In the same way as \eqref{eq-5-1}, we have
$|G(h, h)| = \OO (h^{2(\gamma-\beta)})$ and
\begin{align*}
h^2 \sum_{l _1,l _2 \ge 1} 
|\pd^2 G(l _1h, l _2 h)|
&\lesssim
h^{2(\gamma-\beta)}
\sum_{1 \le l _1^2 + l _2^2 \le 1/h^2}
\frac{1}{(l _1^2+l _2^2)^{\beta-\gamma+1}}
\\
&\qquad
+h^{-2(1+\beta)}
\sum_{l _1^2 + l _2^2 > 1/h^2}
\frac{1}{(l _1^2+l _2^2)^{\beta+2}}
\\
&=
\begin{cases}
\OO(1), & \beta < \gamma,
\\
\OO(- \log h), & \beta = \gamma,
\\
\OO(h^{2(\gamma-\beta)}), & \beta >\gamma.
\end{cases}
\end{align*}
Therefore, we see
\begin{align*}
K(h) &= \OO (h^{2(\gamma-\beta)})+
\begin{cases}
\OO(h^{2\gamma_2-1}), & \beta < \gamma_1 +1/2,
\\
\OO(-h^{2\gamma_2-1} \log h), & \beta = \gamma_1 +1/2,
\\
\OO(h^{2(\gamma-\beta)}), & \beta > \gamma_1 +1/2
\end{cases}
\\
&\qquad+
\begin{cases}
\OO(h^{2\gamma_1-1}), & \beta < \gamma_2 +1/2,
\\
\OO(-h^{2\gamma_1-1} \log h), & \beta = \gamma_2 +1/2,
\\
\OO(h^{2(\gamma-\beta)}), & \beta > \gamma_2 +1/2
\end{cases}
\\
&\qquad+
\begin{cases}
\OO(1), & \beta < \gamma,
\\
\OO(- \log h), & \beta = \gamma,
\\
\OO(h^{2(\gamma-\beta)}), & \beta >\gamma
\end{cases}
\\
&=
\begin{cases}
\OO(h^{2 \underline{\gamma} -1}), & \beta < \overline{\gamma} +1/2,
\\
\OO(-h^{2 \underline{\gamma} -1} \log h), 
& \beta = \overline{\gamma} +1/2,
\\
\OO(h^{2(\gamma -\beta)}), & \beta > \overline{\gamma} +1/2.
\end{cases}
\\
&\qquad+
\begin{cases}
\OO(1), & \beta < \gamma,
\\
\OO(- \log h), & \beta = \gamma,
\\
\OO(h^{2(\gamma-\beta)}), & \beta >\gamma
\end{cases}
\\
&=
\begin{cases}
\OO(h^{0 \land (2 \underline{\gamma}-1)}), 
& \beta < \gamma \lor (\overline{\gamma} +1/2),
\\
\OO(-h^{0 \land (2\underline{\gamma} -1)} \log h), 
& \beta = \gamma \lor (\overline{\gamma} +1/2),
\\
\OO(h^{2(\gamma-\beta)}), 
& \beta > \gamma \lor (\overline{\gamma} +1/2).
\end{cases}
\end{align*}
Using Lemma 4.6 in \cite{TKU2024b}, we obtain \eqref{eq-lem5-2}.
\end{proof}

\subsubsection{Boundedness of two-dimensional Fourier cosine series and its differences}
Let $\beta >0$ and $r \in (0,\infty)$.
For $f \in \boldsymbol F_{\beta,\beta}$ and $y,z \in [0,2)$, define
\begin{equation}\label{eq-6-1}
F_{\Delta}(y,z)=
\Delta^{1+\beta}
\sum_{ l_1, l_2\ge1}
f(\lambda_{ l_1, l_2}\Delta) \cos(\pi  l_1 y)\cos(\pi  l_2 z).
\end{equation}
Let
\begin{align*}
&g_1(x,y) = 1,&
&g_2(x,y) = \sin(x),&
&g_3(x,y) = 1-\cos(x),
\\
&g_4(x,y) = \sin^2(x),&
&g_5(x,y) = \sin(x)\sin(y),&
&g_6(x,y) = \sin(x)(\cos(y)-1),
\\
&g_7(x,y) = \sin^2(x)\sin(y),&
&g_8(x,y) = \sin^2(x)(\cos(y)-1),&
&g_9(x,y) = \sin^2(x)\sin^2(y).
\end{align*}
We then have 
\begin{equation}\label{eq-6-2}
g_1 \in \boldsymbol G_0,
\quad
g_2 \in \boldsymbol G_{1/2},
\quad
g_3,g_4,g_5 \in \boldsymbol G_1,
\end{equation}
\begin{equation}\label{eq-6-3}
g_6 \in \widetilde {\boldsymbol G}_{1/2,1} 
\subset \boldsymbol G_{3/2},
\quad
g_7 \in \widetilde {\boldsymbol G}_{1,1/2} 
\subset \boldsymbol G_{3/2},
\quad
g_8,g_9 \in \widetilde {\boldsymbol G}_{1,1} 
\subset \boldsymbol G_2.
\end{equation}
The series given in \eqref{eq-6-1} and its differences have the following asymptotic results.
\begin{lem}\label{lem6}
For $\beta > 0$ and $\delta=r\sqrt{\Delta}$, it follows that
\begin{itemize}
\item[(1)]
$F_{\Delta}(y,z) = \OO(\Delta^{\beta \tand 1})$ uniformly in $y,z \in [0,2)$,

\item[(2)]
$D_{1,\delta} F_{\Delta}(y,z) = \OO(\Delta^{\beta \tand (3/2)})$ 
uniformly in $y \in (0,2)$ and $z \in [0,2)$,

\item[(3)]
$D_{1,\delta} F_{\Delta}(0,z) = \OO(\Delta^{\beta \tand 2})$ 
uniformly in $z \in [0,2)$,

\item[(4)]
$D_{1,\delta}^2 F_{\Delta}(y,z) = \OO(\Delta^{\beta \tand 2})$ 
uniformly in $y, z \in [0,2)$,

\item[(5)]
$D_{1,\delta}D_{2,\delta} F_{\Delta}(y,z) = \OO(\Delta^{\beta \tand 2})$ 
uniformly in $y,z \in (0,2)$.
\end{itemize}
\end{lem}

\begin{proof}
(1) Since it holds that
\begin{equation*}
|F_{\Delta}(y,z)| \le 
\Delta^{1+\beta} \sum_{ l_1, l_2\ge1} |f(\lambda_{l_1, l_2}\Delta)|
= \Delta^{1+\beta} \sum_{ l_1, l_2\ge1} 
|f(\lambda_{l_1, l_2}\Delta) g_1(l_1 \sqrt{\Delta}, l_2 \sqrt{\Delta})|,
\end{equation*}
and from \eqref{eq-6-2} and \eqref{eq-lem1-2} with $\gamma=0$,
\begin{align*}
\sum_{ l_1, l_2\ge1} 
|f(\lambda_{l_1, l_2}\Delta) g_1(l_1 \sqrt{\Delta}, l_2 \sqrt{\Delta})|
&=
\begin{cases}
\OO (\Delta^{-1}), & \beta < 1,
\\
\OO (-\Delta^{-1} \log \Delta), & \beta = 1,
\\
\OO (\Delta^{-\beta}), & \beta >1,
\end{cases}
\end{align*}
we get the desired result. 

(2) Using the trigonometric identity
\begin{equation}\label{eq-6-4}
\cos(a+b)-\cos(a)= -2 \sin\Bigl(\frac{b}{2}\Bigr)\sin\Bigl(a+\frac{b}{2}\Bigr),
\end{equation}
and $\delta=r\sqrt{\Delta}$, we have
\begin{align*}
|D_{1,\delta} F_{\Delta}(y,z)| 
&\lesssim \Delta^{1+\beta}
\sum_{ l_1, l_2\ge1} 
|f (\lambda_{ l_1, l_2}\Delta)| 
\Bigl|\sin\Bigl(\frac{\pi r l_1\sqrt{\Delta}}{2}\Bigr)\Bigr|
\\
&\lesssim \Delta^{1+\beta}
\sum_{l_1, l_2\ge1} 
|f(\lambda_{l_1, l_2}\Delta)g_2(l_1 \sqrt{\Delta}, l_2 \sqrt{\Delta})|.
\end{align*}
Since we find from \eqref{eq-6-2} and \eqref{eq-lem1-2} with $\gamma =1/2$ that
\begin{align*}
\sum_{l_1, l_2\ge1} 
|f(\lambda_{l_1, l_2}\Delta)g_2(l_1 \sqrt{\Delta}, l_2 \sqrt{\Delta})|
&=
\begin{cases}
\OO (\Delta^{-1}), & \beta < 3/2,
\\
\OO (-\Delta^{-1} \log \Delta), & \beta = 3/2,
\\
\OO (\Delta^{1/2-\beta}), & \beta > 3/2,
\end{cases}
\end{align*}
the desired result can be obtained.

(3)-(5) 
By the trigonometric identities \eqref{eq-6-4} and 
\begin{equation}\label{eq-6-5}
\cos(a+2b)-2\cos(a+b)+\cos(a)=-4\sin^2\Bigl(\frac{b}{2}\Bigr)\cos(a+b),
\end{equation}
we find from \eqref{eq-6-2} and \eqref{eq-lem1-2} with $\gamma=1$ that
\begin{align*}
|D_{1,\delta} F_{\Delta}(0,z)| 
&\le \Delta^{1 +\beta}
\sum_{ l_1, l_2\ge1} 
|f (\lambda_{ l_1, l_2}\Delta)| (1-\cos(\pi r l_1\sqrt{\Delta})) 
\\
&\lesssim \Delta^{1+\beta}
\sum_{l_1, l_2\ge1} 
|f(\lambda_{l_1, l_2}\Delta) g_3(l_1 \sqrt{\Delta}, l_2 \sqrt{\Delta})|
\\
&=
\begin{cases}
\OO (\Delta^\beta), & \beta <2,
\\
\OO (-\Delta^\beta \log \Delta), & \beta =2,
\\
\OO(\Delta^2), & \beta >2,
\end{cases}
\\
|D_{1,\delta}^2 F_{\Delta}(y,z)|
&\lesssim 
\Delta^{1 +\beta}
\sum_{ l_1, l_2\ge1} 
|f(\lambda_{ l_1, l_2}\Delta)| 
\sin^2\Bigl(\frac{\pi r l_1\sqrt{\Delta}}{2}\Bigr) 
\\
&\lesssim \Delta^{1+\beta}
\sum_{l_1, l_2\ge1} 
|f(\lambda_{l_1, l_2}\Delta) g_4(l_1 \sqrt{\Delta}, l_2 \sqrt{\Delta})|
\\
&=
\begin{cases}
\OO (\Delta^\beta), & \beta <2,
\\
\OO (-\Delta^\beta \log \Delta), & \beta =2,
\\
\OO(\Delta^2), & \beta >2,
\end{cases}
\\
|D_{1,\delta}D_{2,\delta} F_{\Delta}(y,z)|
&\lesssim 
\Delta^{1+\beta}
\sum_{ l_1, l_2\ge1} 
|f(\lambda_{ l_1, l_2}\Delta)| 
\biggl|\sin\Bigl(\frac{\pi r l_1\sqrt{\Delta}}{2}\Bigr)\biggr|
\biggl|\sin\Bigl(\frac{\pi r l_2\sqrt{\Delta}}{2}\Bigr)\biggr| 
\\
&\lesssim \Delta^{1+\beta}
\sum_{l_1, l_2\ge1} 
|f(\lambda_{l_1, l_2}\Delta) g_5(l_1 \sqrt{\Delta}, l_2 \sqrt{\Delta})|
\\
&=
\begin{cases}
\OO (\Delta^\beta), & \beta <2,
\\
\OO (-\Delta^\beta \log \Delta), & \beta =2,
\\
\OO(\Delta^2), & \beta >2.
\end{cases}
\end{align*}
\end{proof}

\begin{lem}\label{lem7}
For $\beta >0$, $\delta = r \sqrt{\Delta}$ and $y,z\in(0,2-\delta)$,
it holds that
\begin{itemize}
\item[(1)]
$D_{1,\delta} D_{2,\delta} F_{\Delta}(y,0)
= \OO \Bigl(\frac{\Delta^{1+ \beta \tand (3/2)}}{(y \land (2-\delta-y))^2}\Bigr)$,

\item[(2)]
$D_{1,\delta}^2 D_{2,\delta} F_{\Delta}(y,0)
= \OO \Bigl(
\frac{\Delta^{1+\beta \tand 2}}{(y \land (2-\delta-y))^2}
\Bigr)$,

\item[(3)]
$D_{1,\delta}^2 D_{2,\delta} F_{\Delta}(y,z)
= \OO \Bigl(
\frac{\Delta^{1+\beta \tand (3/2)}}{(y \land (2-\delta-y))(z \land (2-\delta-z))}
\Bigr)$,

\item[(4)]
$D_{1,\delta}^2 D_{2,\delta}^2 F_{\Delta}(y,z) 
= \OO \Bigl(
\frac{\Delta^{1+\beta \tand 2}}{(y \land (2-\delta-y))(z \land (2-\delta-z))}
\Bigr)$.

\end{itemize}
\end{lem}

\begin{proof}
(1), (2) Using \eqref{eq-6-3}-\eqref{eq-6-5},
\eqref{eq-lem1-1} and \eqref{eq-lem5-1} with 
$(\gamma_1, \gamma_2,\gamma) = (1/2, 1, 3/2)$ or $\gamma = 2$,
we obtain
\begin{align*}
&|D_{1,\delta} D_{2,\delta} F_{\Delta}(y,0)|
\\
&\lesssim
\Delta^{1+\beta}
\Biggl|
\sum_{ l_1, l_2\ge1} 
f(\lambda_{ l_1, l_2}\Delta) 
\sin\Bigl(\frac{\pi r l_1\sqrt{\Delta}}{2}\Bigr)
(\cos(\pi r  l_2\sqrt{\Delta})-1)
\sin\Bigl(\pi  l_1 \Bigl(y+\frac{\delta}{2}\Bigr)\Bigr)
\Biggr|
\\
&\lesssim
\Delta^{1+\beta}
\Biggl|
\sum_{ l_1, l_2\ge1} 
f(\lambda_{ l_1, l_2}\Delta) 
g_6(l_1 \sqrt{\Delta}, l_2 \sqrt{\Delta})
\sin\Bigl(\pi  l_1 \Bigl(y+\frac{\delta}{2}\Bigr)\Bigr)
\Biggr|
\\
&=
\OO \biggl( \frac{\Delta^{1+\beta \tand (3/2)}}{(y \land (2-\delta-y))^2} \biggr)
+\OO ( \Delta^{1+\beta \tand (3/2)} )
\\
&= \OO \biggl( \frac{\Delta^{1+\beta \tand (3/2)}}{(y \land (2-\delta-y))^2} \biggr)
\end{align*}
and
\begin{align*}
&|D_{1,\delta}^2 D_{2,\delta} F_{\Delta}(y,0)|
\\
&\lesssim
\Delta^{1+\beta}
\Biggl|
\sum_{ l_1, l_2\ge1} 
f(\lambda_{ l_1, l_2}\Delta) 
\sin^2 \Bigl(\frac{\pi r l_1\sqrt{\Delta}}{2}\Bigr)
(\cos(\pi r  l_2\sqrt{\Delta})-1)
\cos(\pi  l_1 (y+\delta))
\Biggr|
\\
&\lesssim
\Delta^{1+\beta}
\Biggl|
\sum_{ l_1, l_2\ge1} 
f(\lambda_{ l_1, l_2}\Delta) 
g_8(l_1\sqrt{\Delta}, l_2\sqrt{\Delta})
\cos(\pi l_1 (y+\delta))
\Biggr|
\\
&=
\OO \biggl( \frac{\Delta^{1+\beta \tand 2}}{(y \land (2-\delta-y))^2} \biggr)
+\OO ( \Delta^{1+\beta \tand 2} )
\\
&= \OO \biggl( \frac{\Delta^{1+\beta \tand 2}}{(y \land (2-\delta-y))^2} \biggr).
\end{align*}

(3), (4) 
It holds from \eqref{eq-6-3}-\eqref{eq-6-5}, \eqref{eq-lem1-1} and \eqref{eq-lem5-2} 
with $(\gamma_1, \gamma_2, \gamma) = (1, 1/2, 3/2)$ or $\gamma = 2$ that
\begin{align*}
&|D_{1,\delta}^2 D_{2,\delta} F_{\Delta}(y,z)|
\\
&\lesssim
\Delta^{1/2+\beta}
\Biggl|
\Delta^{1/2}
\sum_{ l_1, l_2\ge1} 
f(\lambda_{ l_1, l_2}\Delta) 
\sin^2 \Bigl(\frac{\pi r l_1\sqrt{\Delta}}{2}\Bigr)
\sin\Bigl(\frac{\pi r l_2\sqrt{\Delta}}{2}\Bigr)
\\
&\qquad \qquad \times
\cos(\pi  l_1 (y+\delta))
\sin\Bigl(\pi  l_2 \Bigl(z+\frac{\delta}{2}\Bigr)\Bigr)
\Biggr|
\\
&\lesssim
\Delta^{1/2+\beta}
\Biggl|
\Delta^{1/2}
\sum_{ l_1, l_2\ge1} 
f(\lambda_{ l_1, l_2}\Delta) 
g_7(l_1\sqrt{\Delta}, l_2\sqrt{\Delta})
\cos(\pi  l_1 (y+\delta))
\sin\Bigl(\pi  l_2 \Bigl(z+\frac{\delta}{2}\Bigr)\Bigr)
\Biggr|
\\
&= \OO \biggl(
\frac{\Delta^{1+\beta \tand (3/2)}}{(y \land (2-\delta-y))(z \land (2-\delta-z))}
\biggr) +\OO ( \Delta^{1+\beta \tand (3/2)} )
\\
&= \OO \biggl(
\frac{\Delta^{1+\beta \tand (3/2)}}{(y \land (2-\delta-y))(z \land (2-\delta-z))}
\biggr)
\end{align*}
and
\begin{align*}
&|D_{1,\delta}^2 D_{2,\delta}^2 F_{\Delta}(y,z)|
\\
&\lesssim
\Delta^{1+\beta}
\Biggl|
\sum_{ l_1, l_2\ge1} 
f(\lambda_{ l_1, l_2}\Delta) 
\sin^2 \Bigl(\frac{\pi r l_1\sqrt{\Delta}}{2}\Bigr)
\sin^2 \Bigl(\frac{\pi r l_2\sqrt{\Delta}}{2}\Bigr)
\cos(\pi  l_1 (y+\delta))
\cos(\pi  l_2 (z+\delta))
\Biggr|
\\
&\lesssim
\Delta^{1+\beta}
\Biggl|
\sum_{ l_1, l_2\ge1} 
f(\lambda_{ l_1, l_2}\Delta) 
g_9(l_1\sqrt{\Delta},l_2\sqrt{\Delta})
\cos(\pi  l_1 (y+\delta))
\cos(\pi  l_2 (z+\delta))
\Biggr|
\\
&= 
\OO \biggl(
\frac{\Delta^{1+\beta \tand 2}}{(y \land (2-\delta-y))(z \land (2-\delta-z))}
\biggr) + \OO (\Delta^{1+\beta \tand 2})
\\
&= 
\OO \biggl(
\frac{\Delta^{1+\beta \tand 2}}{(y \land (2-\delta-y))(z \land (2-\delta-z))}
\biggr).
\end{align*}
\end{proof}

\begin{lem}\label{lem8}
For $\beta >0$, $\delta=r\sqrt{\Delta}$, $y\in(0,2-\delta)$ 
and $0 \le K_1,K_2 \le 1/\delta$, it holds that
\begin{itemize}
\item[(1)]
$D_{1,\delta} D_{2,\delta}^2 F_{\Delta}(y, K_2\delta)
= \OO \Bigl(
\frac{\Delta^{1/2+\beta \tand (3/2)}}{(y \land (2-\delta-y))(K_2 +1)}
\Bigr) + \OO(\Delta^{1+\beta \tand (3/2)})$,

\item[(2)]
$D_{1,\delta}^2 D_{2,\delta}^2 F_{\Delta}(y, K_2\delta)
= \OO \Bigl(
\frac{\Delta^{1/2+\beta \tand 2}}{(y \land (2-\delta-y))(K_2 +1)}
\Bigr) + \OO(\Delta^{1+\beta \tand 2})$,

\item[(3)]
$D_{1,\delta}^2 D_{2,\delta}^2 F_{\Delta}(K_1\delta, K_2\delta)
= \OO \Bigl(
\frac{\Delta^{\beta \tand 2}}{(K_1 + 1)(K_2 +1)}
\Bigr) + \OO(\Delta^{1+\beta \tand 2})$.
\end{itemize}

\end{lem}
\begin{proof}
In a similar way to the proof of (3) in Lemma \ref{lem7}, 
we get the first statement:
\begin{align*}
|D_{1,\delta} D_{2,\delta}^2 F_{\Delta}(y,K_2 \delta)|
&=
\OO \biggl(
\frac{\Delta^{1+\beta \tand (3/2)}}{(y \land (2-\delta-y))(K_2 +1)\delta}
\biggr)
+ \OO(\Delta^{1+\beta \tand (3/2)})
\\
&=
\OO \biggl(
\frac{\Delta^{1/2+\beta \tand (3/2)}}{(y \land (2-\delta-y))(K_2 +1)}
\biggr)
+ \OO(\Delta^{1+\beta \tand (3/2)}).
\end{align*}
Similarly, we can show the rest of the statements by using (4) in Lemma \ref{lem7}.
\end{proof}

\subsubsection{Auxiliary results for consistency of the estimator}
We need the following result to get consistency of the estimators 
$\widehat \theta_2$, $\widehat \theta_1$ and $\widehat \eta_1$.
\begin{lem}\label{lem11}
Let $\alpha \in (0,3)$ and $r \in (0,\infty)$. 
$\phi_{r,\alpha}(\theta_2)$ given in \eqref{psi} is strictly decreasing 
for $\theta_2>0$.
\end{lem}
\begin{proof}
Let
\begin{equation*}
\phi_{\alpha}(t)
= \frac{1}{t} \int_0^\infty 
\frac{1-\ee^{-t x^2}}{x^{1+2\alpha}} (J_0(\sqrt{2}x)-2J_0(x)+1) \dd x, 
\quad t>0.
\end{equation*}
Noting that
\begin{equation*}
(1+x)\ee^{-x}-1 \lesssim 
\begin{cases}
x^2 & (x \to 0),
\\
1 & (x \to \infty),
\end{cases}
\end{equation*}
\begin{equation*}
J_0(\sqrt{2}x)-2J_0(x)+1 \lesssim 
\begin{cases}
x^4 & (x \to 0),
\\
1 & (x \to \infty),
\end{cases}
\end{equation*}
we find that for $\alpha \in (0,3)$,
$t \mapsto \phi_{\alpha}(t)$ is differentiable and 
\begin{align*}
\phi_{\alpha}'(t)
&= \frac{1}{t^2} \biggl(
\int_0^\infty \frac{t x^2 \ee^{-t x^2}}{x^{1+2\alpha}}
(J_0(\sqrt{2}x)-2J_0(x)+1) \dd x
\\
&\qquad -\int_0^\infty \frac{1-\ee^{-t x^2}}{x^{1+2\alpha}}
(J_0(\sqrt{2}x)-2J_0(x)+1) \dd x
\biggr)
\\
&= \frac{1}{t^2} \int_0^\infty \frac{(1+t x^2) \ee^{-t x^2}-1}{x^{1+2\alpha}}
(J_0(\sqrt{2}x)-2J_0(x)+1) \dd x.
\end{align*}
By the inequalities $(1+x) \ee^{-x} \le 1$ and 
\begin{equation*}
J_0(\sqrt{2} x)-2J_0(x)+1
= \frac{2}{\pi}\int_0^{\pi/2} (1 -\cos(x \cos(t)))(1 -\cos(x \sin(t))) \dd t
\ge 0
\end{equation*}
for $x \ge 0$, we obtain $\phi_\alpha'(t) < 0$ for $t >0$.
Therefore, $\phi_\alpha(t)$ is strictly decreasing for $t>0$.
The representation
$\phi_{r,\alpha}(\theta_2) = \frac{2 r^{2(\alpha -1)}}{\pi} \phi_\alpha(\theta_2/r^2)$ 
yields the desired result.
\end{proof}

\subsection{Proof of Theorem \ref{th2}}
For $\mf l_1, \mf l_2 \ge 1$ such that $\lb X_0, e_{\mf l_1, \mf l_2} \rb \neq 0$,
we define 
\begin{equation*}
\mathcal N_{i,\lambda}
= X_{\widetilde t_i}-\ee^{-\lambda \Delta_n} X_{\widetilde t_{i-1}},
\end{equation*}
\begin{equation*}
v_{\mf l_1,\mf l_2}(y,z:\kappa,\eta) 
= 2\sin(\pi \mf l_1 y) \ee^{\kappa y/2} \sin(\pi \mf l_2 z) \ee^{\eta z/2},
\end{equation*}
\begin{equation*}
v_{\mf l_1,\mf l_2}^*(y,z) = v_{\mf l_1,\mf l_2}(y,z:\kappa^*,\eta^*),
\quad
\widehat v_{\mf l_1,\mf l_2}(y,z) = v_{\mf l_1,\mf l_2}(y,z:\widehat \kappa, \widehat \eta),
\end{equation*}
\begin{align*}
\widehat {\mathcal M}_{i, \lambda, \mf l_1,\mf l_2} &= 
\widehat x_{\mf l_1,\mf l_2}(\widetilde t_i) 
-\ee^{-\lambda \Delta_n} \widehat x_{\mf l_1,\mf l_2}(\widetilde t_{i-1})
\\
&= \sum_{j=1}^{M_1} \sum_{k=1}^{M_2} \int_{z_{k-1}}^{z_k} \int_{y_{j-1}}^{y_j} 
\mathcal N_{i,\lambda}(y_{j-1},z_{k-1})
\widehat v_{\mf l_1,\mf l_2}(y,z) \dd y \dd z
\\
&= \int_0^1 \int_0^1 \Psi_M \mathcal N_{i,\lambda}(y,z)
\widehat v_{\mf l_1,\mf l_2}(y,z) \dd y \dd z,
\\
\mathcal M_{i, \lambda, \mf l_1,\mf l_2} &= 
x_{\mf l_1,\mf l_2}(\widetilde t_i) 
-\ee^{-\lambda \Delta_n} x_{\mf l_1,\mf l_2}(\widetilde t_{i-1})
\\
&=
\int_0^1 \int_0^1 \mathcal N_{i,\lambda} (y,z)
v_{\mf l_1,\mf l_2}^*(y,z) \dd y \dd z,
\end{align*}
\begin{align*}
\widehat {\mathcal A}_{n,\mf l_1,\mf l_2} &= \sup_{\lambda} \sum_{i=1}^n 
|\widehat {\mathcal M}_{i, \lambda, \mf l_1,\mf l_2} 
-\mathcal M_{i, \lambda, \mf l_1,\mf l_2}|^2,
\\
\widehat {\mathcal B}_{n,\mf l_1,\mf l_2} &= \sum_{i=1}^n 
|\widehat x_{\mf l_1,\mf l_2}(\widetilde t_i) 
-x_{\mf l_1,\mf l_2}(\widetilde t_i)|^2,
\\
\mathcal C_{n,\mf l_1,\mf l_2} &= 
\sup_{\lambda} \sum_{i=1}^n \mathcal M_{i, \lambda, \mf l_1,\mf l_2}^2.
\end{align*}
We then obtain the following result.
\begin{prop}\label{prop3}
Let $\alpha, \alpha_0 \in (0,3)$.
Assume that [A1]$_{\alpha_0}$, [A2] and [B1]$_{\alpha, \alpha_0}$ hold.
\begin{enumerate}
\item[(1)]
For a sequence $\{ r_{n,\epsilon} \}$ such that
\begin{equation*}
n r_{n,\epsilon} 
\biggl( \frac{1}{(M_1 \land M_2)^{2(\alpha_0 \tand 1)}} 
\lor \frac{1}{n^{\alpha_0 \tand 2} \mathcal R_{\alpha,\alpha_0}} \biggr)
\to 0, 
\end{equation*}
\begin{equation*}
n \epsilon^2 r_{n,\epsilon} \biggl( 
\frac{1}{(M_1 \land M_2)^{2(\alpha \tand 1)}}
\lor \frac{1}{n^{\alpha \tand 1} \mathcal R_{\alpha,\alpha_0}} \biggr)
\to 0,
\end{equation*}
it follows that $r_{n,\epsilon} \widehat {\mathcal A}_{n,\mf l_1, \mf l_2} = \oo_\PP(1)$
as $n \to \infty$, $M_1 \land M_2 \to \infty$ and $\epsilon \to 0$.

\item[(2)]
For a sequence $\{ s_{n,\epsilon} \}$ such that
\begin{equation*}
n s_{n,\epsilon} 
\biggl( \frac{1}{(M_1 \land M_2)^{2(\alpha_0 \tand 1)}}
\lor \frac{1}{\mathcal R_{\alpha,\alpha_0}} \biggr) \to 0,
\quad
\frac{n \epsilon^2 s_{n,\epsilon}}{(M_1 \land M_2)^{2(\alpha \tand 1)}} \to 0,
\end{equation*}
it follows that $s_{n,\epsilon} \widehat {\mathcal B}_{n,\mf l_1,\mf l_2} = \oo_\PP(1)$
as $n \to \infty$, $M_1 \land M_2 \to \infty$ and $\epsilon \to 0$.
\end{enumerate}
\end{prop}
Proposition \ref{prop3} yields
\begin{align*}
n \widehat {\mathcal A}_{n,\mf l_1, \mf l_2} = \oo_\PP(1)
\quad \text{under [C1]$_{\alpha, \alpha_0}$},
\\
(n \epsilon)^2 \widehat {\mathcal A}_{n,\mf l_1, \mf l_2} = \oo_\PP(1)
\quad \text{under [C2]$_{\alpha, \alpha_0}$},
\\
(n \epsilon^4)^{-1} \widehat {\mathcal A}_{n,\mf l_1, \mf l_2} = \oo_\PP(1)
\quad \text{under [C3]$_{\alpha, \alpha_0}$},
\\
n \epsilon^{-2} \widehat {\mathcal A}_{n,\mf l_1, \mf l_2} = \oo_\PP(1)
\quad \text{under [C4]$_{\alpha, \alpha_0}$},
\\
(1 \lor (n\epsilon^2)^{-1}) \widehat {\mathcal B}_{n,\mf l_1, \mf l_2} = \oo_\PP(1)
\quad \text{under [C5]$_{\alpha, \alpha_0}$}.
\end{align*}
We will show Theorem \ref{th2} by using the above results. 
We define
\begin{align*}
\mathcal K_n^{\epsilon} (\lambda,\mu:\textbf x)
&=
\begin{pmatrix}
\epsilon^2 \partial_\lambda^2 \mathcal V_n^{\epsilon} (\lambda,\mu:\textbf x) 
& \frac{\epsilon}{\sqrt n}
\partial_\lambda \partial_\mu \mathcal V_n^{\epsilon} (\lambda,\mu:\textbf x)
\\
\frac{\epsilon}{\sqrt n}
\partial_\mu \partial_\lambda \mathcal V_n^{\epsilon} (\lambda,\mu:\textbf x)
& \frac{1}{n} \partial_\mu^2 \mathcal V_n^{\epsilon} (\lambda,\mu:\textbf x)
\end{pmatrix},
\\
\mathcal L_n^{\epsilon} (\lambda,\mu:\textbf x)
&=
\begin{pmatrix}
-\epsilon\partial_\lambda \mathcal V_n^{\epsilon} (\lambda,\mu:\textbf x)
\\
-\frac{1}{\sqrt n}\partial_\mu \mathcal V_n^{\epsilon} (\lambda,\mu:\textbf x)
\end{pmatrix}.
\end{align*}
We consider the following differences between the contrast function,
the score function, and the observed information 
based on the approximate process $\widehat {\textbf x}_{\mf l_1,\mf l_2}$
and those based on the process ${\textbf x}_{\mf l_1,\mf l_2}$, respectively.
\begin{align*}
\widehat {\mathbb U}_{1,\mf l_1,\mf l_2}(n,\epsilon)
&=\epsilon^{2}\sup_{\lambda,\mu}
\bigl|\mathcal V_n^{\epsilon} (\lambda,\mu:\widehat {\textbf x}_{\mf l_1,\mf l_2})
-\mathcal V_n^{\epsilon} (\lambda,\mu:{\textbf x}_{\mf l_1,\mf l_2})\bigr|,
\\
\widehat {\mathbb U}_{2,\mf l_1,\mf l_2}(n,\epsilon)
&=\frac{1}{n}\sup_{\lambda,\mu}
|\mathcal V_n^{\epsilon} (\lambda,\mu:\widehat {\textbf x}_{\mf l_1,\mf l_2})
-\mathcal V_n^{\epsilon} (\lambda,\mu:{\textbf x}_{\mf l_1,\mf l_2}) |,
\\
\widehat {\mathbb U}_{3,\mf l_1,\mf l_2}(n,\epsilon)
&=\sup_{\lambda,\mu}
|\mathcal K_n^{\epsilon} (\lambda,\mu:\widehat {\textbf x}_{\mf l_1,\mf l_2})
-\mathcal K_n^{\epsilon} (\lambda,\mu:{\textbf x}_{\mf l_1,\mf l_2})|,
\\
\widehat {\mathbb U}_{4,\mf l_1,\mf l_2}(n,\epsilon)
&=\sup_{\lambda,\mu}
|\mathcal L_n^{\epsilon} (\lambda,\mu:\widehat {\textbf x}_{\mf l_1,\mf l_2})
-\mathcal L_n^{\epsilon} (\lambda,\mu:{\textbf x}_{\mf l_1,\mf l_2})|.
\end{align*}
According to \cite{TKU2024a}, we can show consistency and asymptotic normality
of the estimators $\widehat \theta_0$ and $\widehat \mu_0$ by verifying 
$\widehat {\mathbb U}_{j,\mf l_1,\mf l_2}(n,\epsilon)$ ($j=1,\ldots,4$) 
converge to $0$ in probability under regular conditions.

\textit{Proof of (2).}
In order to show that the estimators have consistency, we show
$\widehat {\mathbb U}_{j,\mf l_1,\mf l_2}(n,\epsilon) = \oo_\PP(1)$ for $j = 1, 2$.
Using (5.44) and (5.45) in \cite{TKU2024a}, we obtain
\begin{align*}
&\widehat {\mathbb U}_{1,\mf l_1,\mf l_2}(n,\epsilon) 
\lor \widehat {\mathbb U}_{2,\mf l_1,\mf l_2}(n,\epsilon)
\\
& \lesssim 
(n\lor \epsilon^{-2}) \widehat {\mathcal A}_{n,\mf l_1, \mf l_2}
+\Bigl((n\lor \epsilon^{-2})^2(\epsilon^2 \lor n^{-1})
\widehat {\mathcal A}_{n,\mf l_1, \mf l_2} \Bigr)^{1/2}
\Bigl( (\epsilon^2 \lor n^{-1})^{-1} \mathcal C_{n,\mf l_1, \mf l_2}\Bigr)^{1/2}.
\end{align*}
Since $(\epsilon^2 \lor n^{-1})^{-1} \mathcal C_{n,\mf l_1, \mf l_2} = \OO_\PP(1)$, 
$n \lor \epsilon^{-2} \lor (n\lor \epsilon^{-2})^2(\epsilon^2 \lor n^{-1})
=(n\epsilon^4)^{-1} \lor (n\epsilon)^2$ and 
\begin{equation*}
\bigl((n\epsilon^4)^{-1} \lor (n\epsilon)^2 \bigr) 
\widehat {\mathcal A}_{n,\mf l_1, \mf l_2}
=\oo_\PP(1)
\end{equation*}
under [C2]$_{\alpha, \alpha_0}$ and [C3]$_{\alpha, \alpha_0}$,
we get the desired result.

For showing asymptotic normality of the estimators 
$\widehat \theta_0$ and $\widehat \mu_0$, we show
$\widehat {\mathbb U}_{j,\mf l_1,\mf l_2}(n,\epsilon) = \oo_\PP(1)$ 
for $j = 1,\ldots,4$. 
From the proof of (2)-(b) of Theorem 3.6 in \cite{TKU2024a}, it suffices to show 
\begin{equation*}
\bigl( (n\epsilon^4)^{-1} \lor (n\epsilon)^2 \lor n\epsilon^{-2} \bigr) 
\widehat {\mathcal A}_{n,\mf l_1, \mf l_2} = \oo_\PP(1),
\quad
\bigl( 1 \lor (n\epsilon^2)^{-1} \bigr) 
\widehat {\mathcal B}_{n,\mf l_1, \mf l_2} = \oo_\PP(1).
\end{equation*}
We find that these properties hold under 
[C2]$_{\alpha, \alpha_0}$-[C5]$_{\alpha, \alpha_0}$.
In particular, the asymptotic variance of the estimators 
coincides with that of the result for the true coordinate process, 
see Proposition \ref{prop4} below.

\textit{Proof of (1).}
Since it holds that 
\begin{equation*}
\widehat{\mathbb{U}}_{1,\mf l_1,\mf l_2}(n,\epsilon)
\lesssim 
n \widehat {\mathcal A}_{n,\mf l_1, \mf l_2}
+ \Bigl( ((n \epsilon)^2 \lor n )
\widehat {\mathcal A}_{n,\mf l_1, \mf l_2} \Bigr)^{1/2}
\Bigl( (\epsilon^2 \lor n^{-1})^{-1} \mathcal C_{n,\mf l_1, \mf l_2}\Bigr)^{1/2}
\end{equation*}
and
\begin{equation*}
\bigl((n\epsilon)^2 \lor n \bigr) \widehat {\mathcal A}_{n,\mf l_1, \mf l_2}
=\oo_\PP(1)
\end{equation*}
under [C1]$_{\alpha, \alpha_0}$ and [C2]$_{\alpha, \alpha_0}$,
we see that the estimator $\widehat \theta_0$ is consistent.

According to the proof of Theorem 3.6 in \cite{TKU2024a}, 
it is enough to prove 
\begin{equation*}
\bigl( (n \epsilon)^2 \lor n \epsilon^{-2} \bigr) 
\widehat {\mathcal A}_{n,\mf l_1, \mf l_2} = \oo_\PP(1),
\quad
\bigl( 1 \lor (n\epsilon^2)^{-1} \bigr) 
\widehat {\mathcal B}_{n,\mf l_1, \mf l_2} = \oo_\PP(1)
\end{equation*}
in order to show that the estimator $\widehat \theta_0$ have asymptotic normality. 
We see that these properties hold 
under [C2]$_{\alpha, \alpha_0}$, [C4]$_{\alpha, \alpha_0}$ 
and [C5]$_{\alpha, \alpha_0}$.

\subsubsection{Proof of Proposition \ref{prop3}}
For $\beta >0$ and $q \ge 0$, we define
\begin{align*}
\mathbf T_n(\beta, q) &=
\sum_{l_1,l_2 \ge 1}  
\frac{(1-\ee^{-\lambda_{l_1,l_2} \Delta_n})^q}
{\lambda_{l_1,l_2}\mu_{l_1,l_2}^{\beta}},
\\
\mathbf U_{M}(\beta) 
&= \sup_{\substack{|y -y'| \le 1/M_1, \\ |z -z'| \le 1/M_2}} \sum_{l_1,l_2 \ge 1} 
\frac{(e_{l_1,l_2}(y,z)-e_{l_1,l_2}(y',z'))^2}
{\lambda_{l_1,l_2} \mu_{l_1,l_2}^{\beta}},
\\
\mathbf S_n(q) 
&= \sum_{l_1,l_2 \ge 1} (1-\ee^{-\lambda_{l_1,l_2} \Delta_n})^q x_{l_1,l_2}(0)^2.
\end{align*}
Before proving Proposition \ref{prop3}, we prepare the following lemma.
\begin{lem}\label{lem00}
For $\beta >0$ and $q \ge 0$, we have
\begin{equation*}
\mathbf T_n(\beta, q) = \OO(n^{- \beta \tand q}),
\quad
\mathbf U_{M}(\beta) = \OO((M_1 \land M_2)^{-2(\beta \tand 1)})
\end{equation*}
In particular, it follows that under [A1]$_{\alpha_0}$,
\begin{equation*}
\mathbf S_n(q) = \OO(n^{- \alpha_0 \tand q}).
\end{equation*}
\end{lem}
\begin{proof}
Note that for $L > 1$, 
\begin{align*}
\frac{1}{L^q}\sum_{l_1^2+l_2^2 \le L} \frac{1}{(l_1^2+l_2^2)^{1+\beta-q}} 
&= L^{-q} \times
\begin{cases}
\OO(L^{q-\beta}), & \beta < q,
\\
\OO(\log L), & \beta = q,
\\
\OO(1), & \beta > q
\end{cases}
\\
&= \OO(L^{-\beta \tand q}).
\end{align*}
Using the relation
\begin{equation*}
1-\ee^{-\lambda_{l_1,l_2} \Delta_n} \sim 1 \land \frac{l_1^2 +l_2^2}{n},
\end{equation*}
one has
\begin{align}
\mathbf T_n(\beta, q) & \sim 
\sum_{l_1^2+l_2^2 \ge n} 
\frac{1}{(l_1^2+l_2^2)^{1+\beta}} +n^{-q}\sum_{l_1^2+l_2^2 < n} 
\frac{1}{(l_1^2+l_2^2)^{1+\beta-q}}
\nonumber
\\
&= \OO(n^{-\beta}) + \OO (n^{- \beta \tand q})
\nonumber
\\
&= \OO (n^{- \beta \tand q}).
\label{eq-33-01}
\end{align}
Since it follows that
\begin{equation*}
(e_{l_1,l_2}(y,z)-e_{l_1,l_2}(y',z'))^2
\lesssim 1 \land \frac{l_1^2 +l_2^2}{(M_1 \land M_2)^2}
\end{equation*}
uniformly in $|y -y'| \le 1/M_1$ and $|z-z'| \le 1/M_2$, and 
\begin{equation*}
\ee^{-\lambda_{l_1,l_2} \Delta_n}-\ee^{-\lambda \Delta_n} \lesssim 1
\end{equation*}
uniformly in $l_1,l_2 \ge 1$ and $\lambda \in \Lambda$, we have
\begin{align}
\mathbf U_{M}(\beta) 
&\sim \frac{1}{(M_1 \land M_2)^2}
\sum_{l_1^2+l_2^2 \le (M_1 \land M_2)^2} 
\frac{1}{(l_1^2 +l_2^2)^{\beta}}
+\sum_{l_1^2+l_2^2 > (M_1 \land M_2)^2} 
\frac{1}{(l_1^2+l_2^2)^{1+\beta}}
\nonumber
\\
&= \OO ((M_1 \land M_2)^{-2(\beta \tand 1)} )
+\OO ((M_1 \land M_2)^{-2\beta} )
\nonumber
\\
&= \OO ((M_1 \land M_2)^{-2(\beta \tand 1)} ).
\label{eq-33-03}
\end{align}
Under [A1]$_{\alpha_0}$, we obtain 
\begin{align*}
\mathbf S_n(q) 
&= \sum_{l_1,l_2 \ge 1} 
\frac{(1-\ee^{-\lambda_{l_1,l_2} \Delta_n})^q}{\lambda_{l_1,l_2}^{1+\alpha_0}}
\times \lambda_{l_1,l_2}^{1+\alpha_0} x_{l_1,l_2}(0)^2
\nonumber
\\
&\lesssim 
\sum_{l_1,l_2 \ge 1} 
\frac{(1-\ee^{-\lambda_{l_1,l_2} \Delta_n})^q}
{\lambda_{l_1,l_2} \mu_{l_1,l_2}^{\alpha_0}}
= \mathbf T_n(\alpha_0,q),
\end{align*}
which together with \eqref{eq-33-01} yields the desired result.
\end{proof}

We will now begin the proof of Proposition \ref{prop3}.
Note that 
\begin{equation*}
\EE[x_{l_1,l_2}(t)] = \ee^{-\lambda_{l_1,l_2} t} x_{l_1,l_2}(0), 
\quad
\EE[x_{l_1,l_2}(t)^2] = 
\ee^{-2\lambda_{l_1,l_2}t}x_{l_1,l_2}(0)^2
+ \frac{\epsilon^2(1-\ee^{-2\lambda_{l_1,l_2}t})}
{\lambda_{l_1,l_2}\mu_{l_1,l_2}^\alpha}
\end{equation*}
and $\sup_{t \in [0,1]} \EE[ \| X_t \|^2 ] \lesssim 1$.

(1) We obtain
\begin{align*}
&\sup_{\lambda} |\widehat {\mathcal M}_{i,\lambda, \mf l_1, \mf l_2}
-\mathcal M_{i,\lambda, \mf l_1, \mf l_2}|^2
\\
&\lesssim 
\sup_{\lambda} 
\biggl(
\int_0^1 \int_0^1 (\Psi_M \mathcal N_{i,\lambda}(y,z) -\mathcal N_{i,\lambda}(y,z))
\widehat v_{\mf l_1,\mf l_2}(y,z) 
\dd y \dd z
\biggr)^2 
\\
&\qquad+
\sup_{\lambda} \biggl(
\int_0^1 \int_0^1 \mathcal N_{i,\lambda} (y,z)
(\widehat v_{\mf l_1,\mf l_2}(y,z) -v_{\mf l_1,\mf l_2}^*(y,z))
\dd y \dd z
\biggr)^2
\\
&\le 
\sup_{\lambda} \int_0^1 \int_0^1 
(\Psi_M \mathcal N_{i,\lambda} (y,z) -\mathcal N_{i,\lambda}(y,z))^2
\dd y \dd z
\int_0^1 \int_0^1\widehat v_{\mf l_1,\mf l_2}(y,z)^2 \dd y \dd z
\\
&\qquad +
\sup_{\lambda} \int_0^1 \int_0^1 \mathcal N_{i,\lambda} (y,z)^2 \dd y \dd z
\int_0^1 \int_0^1 
(\widehat v_{\mf l_1,\mf l_2}(y,z) -v_{\mf l_1,\mf l_2}^*(y,z))^2
\dd y \dd z
\\
&\lesssim 
\sup_{\lambda} \| \Psi_M \mathcal N_{i,\lambda} -\mathcal N_{i,\lambda} \|^2
+(|\widehat \kappa -\kappa^*|^2 +|\widehat \eta -\eta^*|^2) \times
\sup_{\lambda} \| \mathcal N_{i,\lambda} \|^2
\\
&=: \boldsymbol A_{1,i} 
+(|\widehat \kappa -\kappa^*|^2 +|\widehat \eta -\eta^*|^2) \boldsymbol A_{2,i}.
\end{align*}
We set
$\mathcal J_{i,\lambda} = S_{\Delta_n} X_{\widetilde t_{i-1}} 
-\ee^{-\lambda \Delta_n} X_{\widetilde t_{i-1}}$ and
$\overline X_{t,s} = \int_s^t S_{t-u} \dd W_u^{Q}$.
Since we can write 
\begin{equation*}
\mathcal N_{i,\lambda} = \mathcal J_{i,\lambda}
+\epsilon \overline X_{\widetilde t_i, \widetilde t_{i-1}},
\end{equation*}
we have
\begin{equation*}
\max_{i=1,\ldots,n} \EE [\boldsymbol A_{2,i}]
\lesssim  
\max_{i=1,\ldots,n} \EE \Bigl[ \sup_{\lambda} \| \mathcal J_{i,\lambda} \|^2 \Bigr]
+ \epsilon^2 
\max_{i=1,\ldots,n} 
\EE \bigl[ \| \overline X_{\widetilde t_i, \widetilde t_{i-1}} \|^2 \bigr] 
=: \boldsymbol A_{2}^{(1)} +\epsilon^2 \boldsymbol A_{2}^{(2)}
\end{equation*}
and
\begin{align*}
\boldsymbol A_{2}^{(1)}
&\lesssim \sup_{t \in [0,1]} \EE \bigl[\| S_{\Delta_n} X_t -X_t \|^2 \bigr]
+\sup_{t \in [0,1]} 
\EE \Bigl[ \sup_{\lambda} \| (1-\ee^{-\lambda \Delta_n}) X_t \|^2 \Bigr]
\\
&=: \boldsymbol A_{2}^{(1,1)} + \boldsymbol A_{2}^{(1,2)}.
\end{align*}
Noting that
\begin{equation*}
S_{\Delta_n} X_t -X_t = 
\sum_{l_1,l_2 \ge 1} (\ee^{-\lambda_{l_1,l_2} \Delta_n} -1) x_{l_1,l_2}(t) e_{l_1,l_2},
\end{equation*}
we see from Lemma \ref{lem00} that
\begin{align*}
\boldsymbol A_{2}^{(1,1)}
&= \sup_{t \in [0,1]} 
\sum_{l_1,l_2 \ge 1} (1-\ee^{-\lambda_{l_1,l_2} \Delta_n})^2 \EE[x_{l_1,l_2}(t)^2]
\\
&= \sup_{t \in [0,1]} \sum_{l_1,l_2 \ge 1} (1-\ee^{-\lambda_{l_1,l_2} \Delta_n})^2 
\biggl(
\ee^{-2\lambda_{l_1,l_2}t}x_{l_1,l_2}(0)^2
+ \frac{\epsilon^2(1-\ee^{-2\lambda_{l_1,l_2}t})}
{\lambda_{l_1,l_2}\mu_{l_1,l_2}^\alpha}
\biggr)
\\
&= \OO( \mathbf S_n(2,\alpha_0)) + \OO( \epsilon^2 \mathbf T_n(2,\alpha))
\\
&= \OO( n^{-(\alpha_0 \tand 2)}) +\OO(\epsilon^2 n^{-\alpha}).
\end{align*}
Furthermore, we have
$\boldsymbol A_{2}^{(1,2)} = \OO(n^{-2})$ and
\begin{align*}
\boldsymbol A_{2}^{(2)} &= 
\max_{i=1,\ldots,n} \EE \Biggl[ \biggl \| 
\int_{\widetilde t_{i-1}}^{\widetilde t_i} S_{\widetilde t_i -s} \dd W_s^{Q} 
\biggr\|^2 \Biggr] 
\\
&= \max_{i=1,\ldots,n} \sum_{l_1,l_2 \ge 1} \mu_{l_1,l_2}^{-\alpha} 
\int_{\widetilde t_{i-1}}^{\widetilde t_i} 
\ee^{-2\lambda_{l_1,l_2}(\widetilde t_i -s)} \dd s
\\
&= \sum_{l_1,l_2 \ge 1} 
\frac{1-\ee^{-2\lambda_{l_1,l_2} \Delta_n}}{2\lambda_{l_1,l_2} \mu_{l_1,l_2}^{\alpha}}
\\
&= \OO (\mathbf T_n(1,\alpha)) = \OO( n^{-(\alpha \tand 1)} ).
\end{align*}
Therefore, we obtain
\begin{align}
\max_{i=1,\ldots,n} \EE[ \boldsymbol A_{2,i}]
& \lesssim \boldsymbol A_{2}^{(1,1)} +\boldsymbol A_{2}^{(1,2)}
+\epsilon^2 \boldsymbol A_{2}^{(2)}
\nonumber
\\
&= \OO( n^{-(\alpha_0 \tand 2)}) +\OO(\epsilon^2 n^{-\alpha})
+ \OO(n^{-2}) + \OO( \epsilon^2 n^{-(\alpha \tand 1)} )
\nonumber
\\
&= \OO( n^{-(\alpha_0 \tand 2)}) +\OO( \epsilon^2 n^{-(\alpha \tand 1)} ).
\label{eq-XXXX}
\end{align}

On the other hand, we have
\begin{align*}
\max_{i=1,\ldots,n} \EE[\boldsymbol A_{1,i}] 
&\lesssim
\max_{i=1,\ldots,n} \EE \Bigl[ 
\sup_{\lambda} \| \Psi_M \mathcal J_{i,\lambda} -\mathcal J_{i,\lambda} \|^2
\Bigr]
\\
&\qquad + \epsilon^2 
\max_{i=1,\ldots,n} 
\EE \bigl[\| \Psi_M \overline X_{\widetilde t_i,\widetilde t_{i-1}} 
-\overline X_{\widetilde t_i,\widetilde t_{i-1}} \|^2 \bigr]
\\
&=: \boldsymbol A_{1}^{(1)} +\epsilon^2 \boldsymbol A_{1}^{(2)}.
\end{align*}
Note that for a continuous function $g: D \to \mathbb R$,
\begin{align*}
\| \Psi_M g -g \|^2
&\lesssim
\sum_{j=1}^{M_1} \sum_{k=1}^{M_2} \int_{z_{k-1}}^{z_k} \int_{y_{j-1}}^{y_j}
( g(y_{j-1},z_{k-1}) -g(y,z))^2 \dd y \dd z
\\
&\le \sup_{\substack{|y -y'| \le 1/M_1, \\ |z -z'| \le 1/M_2}} 
|g(y,z) -g(y',z') |^2.
\end{align*}
Since
\begin{equation*}
S_{\Delta_n} X_t -\ee^{-\lambda \Delta_n} X_t
= \sum_{l_1,l_2 \ge 1} (\ee^{-\lambda_{l_1,l_2} \Delta_n} -\ee^{-\lambda \Delta_n}) 
x_{l_1,l_2}(t) e_{l_1,l_2},
\end{equation*}
we have
\begin{align*}
\boldsymbol A_{1}^{(1)}
&\lesssim \max_{i=1,\ldots,n} \sup_{\substack{|y -y'| \le 1/M_1, \\ |z -z'| \le 1/M_2}}
\EE \bigl[(\mathcal J_{i,\lambda}(y,z) -\mathcal J_{i,\lambda}(y',z'))^2 \bigr]
\\
&\le \sup_{t \in [0,1]} \sup_{\substack{|y -y'| \le 1/M_1, \\ |z -z'| \le 1/M_2}}
\sum_{l_1,l_2 \ge 1} (\ee^{-\lambda_{l_1,l_2} \Delta_n}
-\ee^{-\lambda \Delta_n})^2 
\EE[x_{l_1,l_2}(t)^2]
\\
&\qquad \qquad \times
(e_{l_1,l_2}(y,z) -e_{l_1,l_2}(y',z'))^2
\\
&\qquad+
\sup_{t \in [0,1]} \sup_{\substack{|y -y'| \le 1/M_1, \\ |z -z'| \le 1/M_2}}
\biggl(
\sum_{l_1,l_2 \ge 1} (\ee^{-\lambda_{l_1,l_2} \Delta_n}-\ee^{-\lambda \Delta_n})
\EE[x_{l_1,l_2}(t)] 
\\
&\qquad \qquad \times
(e_{l_1,l_2}(y,z)-e_{l_1,l_2}(y',z'))
\biggr)^2
\\
&=: \boldsymbol A_{1}^{(1,1)} +\boldsymbol A_{1}^{(1,2)}.
\end{align*}
Since it follows that
\begin{align*}
\boldsymbol A_{1}^{(1,1)} &\le 
\sup_{t \in [0,1]} \sup_{\substack{|y -y'| \le 1/M_1, \\ |z -z'| \le 1/M_2}}
\sum_{l_1,l_2 \ge 1}
(\ee^{-\lambda_{l_1,l_2} \Delta_n}-\ee^{-\lambda \Delta_n})^2
\ee^{-2\lambda_{l_1,l_2}t}x_{l_1,l_2}(0)^2
\\
&\qquad \qquad \times
(e_{l_1,l_2}(y,z)-e_{l_1,l_2}(y',z'))^2
\\
&\qquad+
\epsilon^2 
\sup_{t \in [0,1]} \sup_{\substack{|y -y'| \le 1/M_1, \\ |z -z'| \le 1/M_2}}
\sum_{l_1,l_2 \ge 1}
(\ee^{-\lambda_{l_1,l_2} \Delta_n}-\ee^{-\lambda \Delta_n})^2
\frac{1-\ee^{-2\lambda_{l_1,l_2}t}}{\lambda_{l_1,l_2}\mu_{l_1,l_2}^\alpha}
\\
&\qquad \qquad \times
(e_{l_1,l_2}(y,z)-e_{l_1,l_2}(y',z'))^2
\\
&= \OO(\mathbf U_{M}(\alpha_0)) +\OO(\epsilon^2 \mathbf U_{M}(\alpha))
\\
&= \OO ((M_1 \land M_2)^{-2(\alpha_0 \tand 1)})
+\OO (\epsilon^2 (M_1 \land M_2)^{-2(\alpha \tand 1)})
\end{align*}
and
\begin{align*}
\boldsymbol A_{1}^{(1,2)} &= 
\sup_{t \in [0,1]} \sup_{\substack{|y -y'| \le 1/M_1, \\ |z -z'| \le 1/M_2}}
\Biggl(
\sum_{l_1,l_2 \ge 1} (\ee^{-\lambda_{l_1,l_2} \Delta_n}-\ee^{-\lambda \Delta_n})
\ee^{-\lambda_{l_1,l_2}t} x_{l_1,l_2}(0)
\\
&\qquad \qquad \times
(e_{l_1,l_2}(y,z)-e_{l_1,l_2}(y',z'))
\Biggr)^2
\\
&\le
\sum_{l_1,l_2 \ge 1}  
\lambda_{l_1,l_2}^{1+\alpha_0} x_{l_1,l_2}(0)^2
\\
&\qquad \times
\sup_{t \in [0,1]} \sup_{\substack{|y -y'| \le 1/M_1, \\ |z -z'| \le 1/M_2}}
\sum_{l_1,l_2 \ge 1} 
\frac{\ee^{-2\lambda_{l_1,l_2}t}
(\ee^{-\lambda_{l_1,l_2} \Delta_n}-\ee^{-\lambda \Delta_n})^2}
{\lambda_{l_1,l_2}^{1+\alpha_0}}
\\
&\qquad \qquad \times
(e_{l_1,l_2}(y,z)-e_{l_1,l_2}(y',z'))^2
\\
&=
\| A_{\theta}^{(1+\alpha_0)/2} X_0 \|^2
\\
&\qquad \times
\sup_{t \in [0,1]} 
\sup_{\substack{|y -y'| \le 1/M_1, \\ |z -z'| \le 1/M_2}} \sum_{l_1,l_2 \ge 1} 
\frac{\ee^{-2\lambda_{l_1,l_2}t}
(\ee^{-\lambda_{l_1,l_2} \Delta_n}-\ee^{-\lambda \Delta_n})^2}
{\lambda_{l_1,l_2}^{1+\alpha_0}}
\\
&\qquad \qquad \times
(e_{l_1,l_2}(y,z)-e_{l_1,l_2}(y',z'))^2
\\
&= \OO(\mathbf U_{M}(\alpha_0))
\\
&= \OO ((M_1 \land M_2)^{-2(\alpha_0 \tand 1)}),
\end{align*}
we obtain
\begin{equation*}
\boldsymbol A_{1}^{(1)}
\lesssim \boldsymbol A_{1}^{(1,1)} +\boldsymbol A_{1}^{(1,2)}
= \OO ((M_1 \land M_2)^{-2(\alpha_0 \tand 1)} )
+\OO ( \epsilon^2(M_1 \land M_2)^{-2(\alpha \tand 1)} ).
\end{equation*}
We also have
\begin{align*}
\boldsymbol A_{1}^{(2)}
&\lesssim 
\max_{i=1,\ldots,n} \sup_{\substack{|y -y'| \le 1/M_1, \\ |z -z'| \le 1/M_2}}
\EE \bigl[( \overline X_{\widetilde t_i,\widetilde t_{i-1}}(y,z)
-\overline X_{\widetilde t_i,\widetilde t_{i-1}}(y',z'))^2 \bigr]
\\
&= 
\max_{i=1,\ldots,n} \sup_{\substack{|y -y'| \le 1/M_1, \\ |z -z'| \le 1/M_2}}
\EE \Biggl[ \biggl(
\sum_{l_1,l_2 \ge 1} \mu_{l_1,l_2}^{-\alpha/2}
(e_{l_1,l_2}(y,z) -e_{l_1,l_2}(y',z')) 
\\
&\qquad \qquad \times 
\int_{\widetilde t_{i-1}}^{\widetilde t_i} 
\ee^{-\lambda_{l_1,l_2}(\widetilde t_i-s)} \dd w_{l_1,l_2}(s)
\biggr)^2 \Biggr]
\\
&= \max_{i=1,\ldots,n} \sup_{\substack{|y -y'| \le 1/M_1, \\ |z -z'| \le 1/M_2}}
\sum_{l_1,l_2 \ge 1} \mu_{l_1,l_2}^{-\alpha}
(e_{l_1,l_2}(y,z) -e_{l_1,l_2}(y',z'))^2 
\\
&\qquad \qquad \times 
\int_t^{t+\Delta_n} \ee^{-2\lambda_{l_1,l_2}(\widetilde t_i-s)} \dd s
\\
&= \sup_{\substack{|y -y'| \le 1/M_1, \\ |z -z'| \le 1/M_2}}
\sum_{l_1,l_2 \ge 1} 
\frac{1-\ee^{-2\lambda_{l_1,l_2} \Delta_n}}{\lambda_{l_1,l_2} \mu_{l_1,l_2}^\alpha}
(e_{l_1,l_2}(y,z) -e_{l_1,l_2}(y',z'))^2 
\\
&= \OO(\mathbf U_{M}(\alpha))
\\
&= \OO ((M_1 \land M_2)^{-2(\alpha \tand 1)}).
\end{align*}
Therefore, we obtain
\begin{equation*}
\max_{i=1,\ldots,n} \EE[\boldsymbol A_{1,i}]
\lesssim \boldsymbol A_{1}^{(1)} +\epsilon^2 \boldsymbol A_{1}^{(2)}
=
\OO \biggl(\frac{1}{(M_1 \land M_2)^{2(\alpha_0 \tand 1)}} \biggr)
+\OO \biggl( \frac{\epsilon^2}{(M_1 \land M_2)^{2(\alpha \tand 1)}} \biggr),
\end{equation*}
which together with \eqref{eq-XXXX} and Theorem \ref{th1} yields 
\begin{align*}
r_{n,\epsilon} \widehat{\mathcal A}_{n,\mf l_1,\mf l_2}
&\lesssim
r_{n,\epsilon} \sum_{i=1}^n \boldsymbol A_{1,i} 
+(|\widehat \kappa -\kappa^*|^2 +|\widehat \eta -\eta^*|^2) 
r_{n,\epsilon} \sum_{i=1}^n \boldsymbol A_{2,i}
\\
&=
\OO_\PP \biggl(\frac{n r_{n,\epsilon} }{(M_1 \land M_2)^{2(\alpha_0 \tand 1)}} \biggr)
+\OO_\PP \biggl(
\frac{n \epsilon^2 r_{n,\epsilon}}{(M_1 \land M_2)^{2(\alpha \tand 1)}} \biggr)
\\
&\qquad +\OO_\PP \biggl( \frac{n^{1-\alpha_0 \tand 2} r_{n,\epsilon}}
{\mathcal R_{\alpha,\alpha_0}} \biggr) 
+\OO_\PP \biggl( \frac{n^{1- \alpha \tand 1} \epsilon^2 r_{n,\epsilon}}
{\mathcal R_{\alpha,\alpha_0}} \biggr) 
\\
&=\oo_\PP(1).
\end{align*}

(2) Noting that
\begin{align*}
\widehat x_{\mf l_1,\mf l_2}(t) &= 
\int_0^1 \int_0^1 \Psi_M X_t(y,z) \widehat v_{\mf l_1,\mf l_2}(y,z) \dd y \dd z,
\\
x_{\mf l_1,\mf l_2}(t) &= 
\int_0^1 \int_0^1 X_t(y,z) v_{\mf l_1,\mf l_2}^*(y,z) \dd y \dd z,
\end{align*}
we have
\begin{align*}
&|\widehat x_{\mf l_1,\mf l_2}(t) -x_{\mf l_1,\mf l_2}(t)|^2
\\
&\le \int_0^1 \int_0^1 (\Psi_M X_t(y,z) -X_t(y,z))^2 \dd y \dd z
\int_0^1 \int_0^1 \widehat v_{\mf l_1,\mf l_2}(y,z)^2 \dd y \dd z
\\
&\quad+ \int_0^1 \int_0^1 X_t(y,z)^2 \dd y \dd z
\int_0^1 \int_0^1 (\widehat v_{\mf l_1,\mf l_2}(y,z)
-v_{\mf l_1,\mf l_2}^*(y,z))^2 \dd y \dd z
\\
&\lesssim
\| \Psi_M X_t -X_t \|^2 
+(|\widehat \kappa -\kappa^*|^2 +|\widehat \eta -\eta^*|^2) \times \| X_t \|^2.
\end{align*}
Since
\begin{align*}
&\EE[(X_t(y,z) -X_t(y',z'))^2] 
\\
&= 
\sum_{l_1,l_2 \ge 1} \sum_{l_1',l_2' \ge 1}
\EE[x_{l_1,l_2}(t) x_{l_1',l_2'}(t)]
\\
&\qquad \times
(e_{l_1,l_2}(y,z)-e_{l_1,l_2}(y',z'))
(e_{l_1',l_2'}(y,z)-e_{l_1',l_2'}(y',z'))
\\
&\le
\sum_{l_1,l_2 \ge 1}
\EE[x_{l_1,l_2}(t)^2]
(e_{l_1,l_2}(y,z)-e_{l_1,l_2}(y',z'))^2
\\
&\qquad+
\biggl(
\sum_{l_1,l_2 \ge 1}
\EE[x_{l_1,l_2}(t)] (e_{l_1,l_2}(y,z)-e_{l_1,l_2}(y',z'))
\biggr)^2
\\
&=: \boldsymbol B_1 +\boldsymbol B_2,
\end{align*}
\begin{align*}
\boldsymbol B_1 
&=
\sum_{l_1,l_2 \ge 1}
\ee^{-2\lambda_{l_1,l_2}t}x_{l_1,l_2}(0)^2
(e_{l_1,l_2}(y_{j-1},z_{k-1})-e_{l_1,l_2}(y,z))^2
\\
&\qquad +\epsilon^2
\sum_{l_1,l_2 \ge 1}
\frac{1-\ee^{-2\lambda_{l_1,l_2}t}}{\lambda_{l_1,l_2}\mu_{l_1,l_2}^\alpha}
(e_{l_1,l_2}(y_{j-1},z_{k-1})-e_{l_1,l_2}(y,z))^2
\\
&=\OO(\mathbf U_{M}(\alpha_0)) +\OO(\epsilon^2 \mathbf U_{M}(\alpha))
\end{align*}
and
\begin{align*}
\boldsymbol B_2 &= \Biggl( \sum_{l_1,l_2 \ge 1}
\EE[x_{l_1,l_2}(t)] (e_{l_1,l_2}(y_{j-1},z_{k-1})-e_{l_1,l_2}(y,z))
\Biggr)^2
\\
&= \Biggl(
\sum_{l_1,l_2 \ge 1} \ee^{-\lambda_{l_1,l_2}t} x_{l_1,l_2}(0)
(e_{l_1,l_2}(y_{j-1},z_{k-1})-e_{l_1,l_2}(y,z))
\Biggr)^2
\\
&\le
\sum_{l_1,l_2 \ge 1}  
\lambda_{l_1,l_2}^{1+\alpha_0} x_{l_1,l_2}(0)^2
\sum_{l_1,l_2 \ge 1} \frac{\ee^{-2\lambda_{l_1,l_2}t}}{\lambda_{l_1,l_2}^{1+\alpha_0}}
(e_{l_1,l_2}(y_{j-1},z_{k-1})-e_{l_1,l_2}(y,z))^2
\\
&=
\| A_{\theta}^{(1+\alpha_0)/2} X_0 \|^2 \times
\sum_{l_1,l_2 \ge 1} \frac{\ee^{-2\lambda_{l_1,l_2}t}}{\lambda_{l_1,l_2}^{1+\alpha_0}}
(e_{l_1,l_2}(y_{j-1},z_{k-1})-e_{l_1,l_2}(y,z))^2
\\
&= \OO ((M_1 \land M_2)^{-2(\alpha_0 \tand 1)}),
\end{align*}
it holds that
\begin{equation*}
\sup_{t \in [0,1]} \EE \bigl[\| \Psi_M X_t -X_t \|^2 \bigr]
= \OO ((M_1 \land M_2)^{-2(\alpha_0 \tand 1)} )
+\OO (\epsilon^2(M_1 \land M_2)^{-2(\alpha \tand 1)} )
\end{equation*}
and
\begin{align*}
s_{n,\epsilon} \widehat{\mathcal B}_{n,\mf l_1,\mf l_2} &=
\OO_\PP \biggl(\frac{n s_{n,\epsilon}}{(M_1 \land M_2)^{2(\alpha_0 \tand 1)}} \biggr) 
+\OO_\PP \biggl(
\frac{n \epsilon^2 s_{n,\epsilon}}{(M_1 \land M_2)^{2(\alpha \tand 1)}} \biggr)
+\OO_\PP \Bigl( \frac{n s_{n,\epsilon}}{\mathcal R_{\alpha,\alpha_0}^2} \Bigr) 
\\
&=\oo_\PP(1).
\end{align*}

\subsubsection{Estimation for an OU process with a small dispersion parameter}
This subsubsection is devoted to parametric estimation for 
the Ornstein-Uhlenbeck process
\begin{equation}
\dd x (t) = -\lambda x(t) \dd t +\epsilon \mu^{-\alpha/2} \dd w(t),
\end{equation}
where $(\lambda, \mu) \in \Xi$ are unknown parameters, 
the parameter space $\Xi$ is a compact convex subset of $\mathbb R \times (0,\infty)$,
$(\lambda^*, \mu^*) \in \mathrm{Int}\,(\Xi)$ are the true values of $(\lambda, \mu)$,
and $\epsilon \in (0,1)$, $x(0) \neq 0$ and $\alpha \in (0,3)$ are known constants.
$\{ w(t) \}_{t \ge 0}$ is a one-dimensional standard Brownian motion.
For details of parametric estimation of the Ornstein-Uhlenbeck process, 
see the appendix section in \cite{TKU2024a}.

Suppose that we have the discrete observations 
${\textbf x} = \{ x(i h_n) \}_{i=0}^{n}$ with $h_n = 1/n$.
We construct the contrast function as follows.
\begin{equation}\label{CF}
\mathcal V_{n}^{\epsilon} (\lambda,\mu:{\textbf x})
=\sum_{i=1}^{n}
\frac{(x(i h_n)- \ee^{-\lambda h_n x(i h_n))^2}}
{\frac{\epsilon^2(1-\ee^{-2\lambda h_n})}{2\lambda \mu^\alpha}}
+n \log \frac{1-\ee^{-2\lambda h_n}}{2\lambda \mu^\alpha h_n}.
\end{equation}
If $\mu$ is known, then we set $\mu = \mu^*$ and define
\begin{equation*}
\widehat \lambda = \underset{\lambda}{\mathrm{arginf}\,} 
\mathcal V_{n}^{\epsilon} (\lambda,\mu^*:{\textbf x})
\end{equation*}
as the estimator of $\lambda$, or if $\mu$ is unknown, then we set
\begin{equation*}
(\widehat \lambda, \widehat \mu) = \underset{\lambda, \mu}{\mathrm{arginf}\,} 
\mathcal V_{n}^{\epsilon} (\lambda,\mu:{\textbf x}).
\end{equation*}
as the estimators of $\lambda$ and $\mu$.

\begin{prop}[Theorem A.2 in \cite{TKU2024a}]\label{prop4}
If $\mu$ is known,
then as $n \to \infty$ and $\epsilon \to 0$, 
\begin{equation*}
\epsilon^{-1} (\widehat \lambda -\lambda^*) 
\dto N (0, G(\lambda^*, \mu^*, x(0))^{-1}).
\end{equation*}
If $\mu$ is unknown, then as $n \to \infty$ and $\epsilon \to 0$, 
\begin{equation*}
\begin{pmatrix}
\epsilon^{-1} (\widehat \lambda -\lambda^*)
\\
\sqrt{n} (\widehat \mu -\mu^*)
\end{pmatrix}
\dto N (0, I(\lambda^*, \mu^*, x(0))^{-1}).
\end{equation*}
\end{prop}



\end{document}